\documentclass[reqno,10pt,centertags]{amsart}
\usepackage{amsmath,amsthm,amscd,amssymb,latexsym,upref,enumerate,mathtools,mathrsfs}
\usepackage{bbm}
\usepackage{nicefrac}
\usepackage{esint}
\usepackage[usenames,dvipsnames,svgnames,table]{xcolor}

\usepackage{hyperref}

\usepackage{tikz}

\newcommand*{\mailto}[1]{\href{mailto:#1}{\nolinkurl{#1}}}
\newcommand{\arxiv}[1]{\href{http://arxiv.org/abs/#1}{arXiv:#1}}

\newcommand{\bbC}{{\mathbb{C}}}

\newcommand{\bbN}{{\mathbb{N}}}

\newcommand{\bbR}{{\mathbb{R}}}

\newcommand{\bbZ}{{\mathbb{Z}}}

\newcommand{\cB}{{\mathcal B}}

\newcommand{\cH}{{\mathcal H}}

\newcommand{\cM}{{\mathcal M}}

\DeclareMathOperator{\ran}{ran}
\DeclareMathOperator{\dom}{dom}

\DeclareMathOperator{\Pf}{Pf}
\DeclareMathOperator{\sech}{sech}

\DeclareMathOperator{\KdV}{KdV}

\DeclareMathOperator{\mKdV}{mKdV}

\renewcommand{\Im}{\text{\rm Im}}
\renewcommand{\ln}{\text{\rm ln}}

\newcommand{\ind}{\operatorname{index}}
\newcommand{\no}{\notag}
\newcommand{\lb}{\label}
\newcommand{\f}{\cfrac}

\newcommand{\ol}{\overline}

\newcommand{\wti}{\widetilde}
\newcommand{\Oh}{O}
\newcommand{\oh}{o}
\newcommand{\hatt}{\widehat}
\newcommand{\dott}{\,\cdot\,}

\renewcommand{\dot}{\overset{\textbf{\Large.}}}
\renewcommand{\ddot}{\overset{\textbf{\Large..}}}
\renewcommand{\dotplus}{\overset{\textbf{\Large.}} +}

\newcommand{\bi}{\bibitem}

\renewcommand{\ge}{\geqslant}
\renewcommand{\le}{\leqslant}

\let\geq\geqslant
\let\leq\leqslant

\makeatletter
\def\theequation{\@arabic\c@equation}

\newcommand{\msF}{\mathscr{F}} 
\newcommand{\mCC}{\mathcal{C}} 
\newcommand{\CH}{\mathcal{H}} 
\newcommand{\CL}{\mathcal{L}} 
\newcommand{\tmu}{\tilde{\mu}} 
\newcommand{\tCC}{\tilde{\mCC}} 

\newcommand{\CC}{\mathbb{C}} 
\newcommand{\RR}{\mathbb{R}} 
\newcommand{\BSS}{\mathbb{S}} 
\newcommand{\dW}{\dot{W}} 

\newcommand{\lf}{\left}
\newcommand{\rt}{\right}
\newcommand{\tit}{\textit}
\newcommand{\tbf}{\textbf}

\newcommand{\de}{\delta}
\newcommand{\De}{\Delta}

\newcommand{\Om}{\Omega}
\newcommand{\nab}{\nabla}
\newcommand{\la}{\lambda}
\newcommand{\ga}{\gamma}
\newcommand{\sg}{\sigma}
\newcommand{\Sg}{\Sigma}
\newcommand{\p}{\partial}

\newcommand{\tr}{\rm{tr}}

\allowdisplaybreaks
\numberwithin{equation}{section}

\newtheorem{theorem}{Theorem}[section]

\newtheorem{lemma}[theorem]{Lemma}

\newtheorem{hypothesis}[theorem]{Hypothesis}

\theoremstyle{remark}
\newtheorem{remark}[theorem]{Remark}

\numberwithin{equation}{section}

\begin{document}

\title[Applications of a Commutation Formula II]{Applications of a Commutation Formula II}

\dedicatory{Dedicated, with admiration, to Barry Simon, mentor and friend, \\ on the occasion of his 80th birthday.}

\author[P.\ A.\ Deift]{Percy A.\  Deift}
 \address{Department of Mathematics, 
Courant Institute of Mathematical Sciences, 
New York University, 251 Mercer Str., New York, NY 10012, USA}
\email{\mailto{deift@cims.nyu.edu}}
\urladdr{\url{http://math.nyu.edu/faculty/deift/}}

\author[F.\ Gesztesy]{Fritz Gesztesy}
\address{Department of Mathematics,
Baylor University, Sid Richardson Bldg., 1410 S.\,4th Street, Waco, TX 76706, USA}
\email{\mailto{Fritz\_Gesztesy@baylor.edu}}
\urladdr{\url{https://math.artsandsciences.baylor.edu/person/fritz-gesztesy-phd}}

\date{\today}



\begin{abstract}
The principal {\it commutation formula} underlying the paper by Deift \cite{De78}, is of the type 
\[
- z (AB - z I)^{-1} + A (BA - z I)^{-1} B = I, \quad z \in [\rho(AB) \cap \rho(BA)]\backslash\{0\}, 
\]
where, temporarily, and for simplicity only, $A,B$ represent bounded operators in a Banach space. A fundamental consequence of this commutation formula is the fact that $AB$ and $BA$ are {\it essentially isospecral} in the sense that 
\[   
\sigma(AB) \backslash \{0\} = \sigma(BA) \backslash \{0\}. 
\]
Moreover, under appropriate trace class hypotheses, this implies the fundamental identity for (Fredholm) determinants
\[
\det(I - AB) = \det(I - BA),
\]
which extends to modified (or regularized) Fredholm determinants. 

In this paper we briefly recall some of the applications of these commutation formulas to isospectrality, integrable systems, regularization of PDEs, and estimates in Statistical Mechanics and Quantum Field Theory discussed in \cite{De78}. Subsequently, we review some of the ramifications of commutation formulas that were established over the past fifty years, including determinants and dimension reduction, random matrix theory and dimension expansion, integrable operators, Poncelet's porism, interpolation, eigenvalue computation, KPZ, Pfaffians, the focusing Gross--Pitaevskii equation, one-dimensional scattering theory, Floquet theory, and the Miura transformation connecting the $\KdV$ and $\mKdV$ hierarchies. 
\end{abstract}

\maketitle

{\scriptsize{\tableofcontents}}

\section{Introduction} \lb{s1}

In linear algebra, the Woodbury matrix identity -- named after M.\ A.\ 
Woodbury -- asserts that
\begin{equation}\label{1.1}
\lf(W+UCV\rt)^{-1} = W^{-1}-W^{-1} U \lf(C^{-1} + VW^{-1}\,U\rt)^{-1} \, V \, W^{-1}, 
\end{equation}
where $W,\, U,\, C$ and $V$ are conformable matrices:  $W$ is $n\times n$, $C$ is $k\times k$,
$U$ is $n\times k$, and $V$ is $k\times n$.

In the case that the dimension of $C$ is small compared to the dimension of $W$, the identity
\eqref{1.1} allows cheap computation of inverses and solutions to linear equations.  It is this 
feature that makes the identity useful in the theory and applications of Kalman filter theory, which
in turn famously plays an outstanding role in all manner of problems in control theory.

In the 1970's, while working on his Ph.D. thesis in scattering theory, one of us (P.D.)  came across the following
\tit{commutation formula,}
\begin{equation}\label{1.2} 
- z (AB - z I)^{-1} + A (BA - z I)^{-1} B = I, \quad z \in [\rho(AB) \cap \rho(BA)]\backslash\{0\}, 
\end{equation}
which can be obtained from \eqref{1.1} by setting $W=I$, $C=I$, $U=A$ and $V=-B/z$.
One observes that \eqref{1.1} can be recovered in turn from \eqref{1.2} by setting
$A \to W^{-1} A$, $B \to CB$, and then $A\to U$, $B\to V$ and $z \to -1$.

The commutation formula extends easily to bounded operators
$A \in \cB (X_1, X_2)$ and $B \in \cB (X_2, X_1)$ where $X_1$ and $X_2$ are Banach
spaces, and also to certain unbounded (densely defined, closed) operators
$A$ in a Hilbert space with $B=A^*$. This is discussed in detail in Subsections~\ref{sec12} and \ref{sec13}. 

One of the fundamental consequences of the commutation formula \eqref{1.2} results in the {\it essential isospectrality} of $AB$ and $BA$ in the form,
\begin{equation}   
\sigma(AB) \backslash \{0\} = \sigma(BA) \backslash \{0\}.     \lb{1.3} 
\end{equation}

In \cite{De78}, one of us (P.D.) presented a variety of applications of \eqref{1.2} to problems
in different areas of mathematics and physics. These applications addressed a variety of problems that we turn to  next: We start with  the notion of {\bf essential isospectrality} (Subsection~\ref{sec1}), possible by the fact that the point $0$ is excluded in \eqref{1.3}, leading to the construction of the class of reflectionless potentials for one-dimensional (short-range) Schr\"odinger operators supporting precisely $N \in \bbN$ strictly negative and simple eigenvalues. In turn, this class of reflectionless potentials gives rise to the class of $N$-soliton solutions of the Korteweg--de Vries ($\KdV$) hierarchy of nonlinear evolution equations, an {\bf infinite-dimensional integrable system} (Subsection~\ref{sec2}). We also recall the process of {\bf regularizing PDEs} (Subsection~\ref{sec3}), which, in the particular case of acoustic scattering in nonhomogeneous media, avoids smoothness assumptions on the speed of sound and the density coefficient modeling the underlying medium. Commutation formulas also enter {\bf operator estimates in statistical mechanics and quantum field theory} (subsection\ref{sec4}) in connection with Segal's Lemma and the Golden--Thompson inequality. These applications are all briefly recalled in Section~\ref{s2}.

In Section~\ref{s3} we describe a panorama of applications of the commutation formula \eqref{1.2} over the past fifty years. We start with {\bf determinants and dimension reduction} (Subsection~\ref{sec5}), noticing that if $AB$ and $BA$ are trace class, and if $A \in \cB(\cH_1, \cH_2)$ and $B \in \cB(\cH_2, \cH_1)$, then the fundamental (Fredholm) determinant formula
\begin{equation}\label{1.4a}
{\det}_{\cH_2} \lf(I-AB \rt) = {\det}_{\cH_1} \lf(I-BA\rt), 
\end{equation}
holds. Here we explicitly indicated the different Hilbert spaces $\cH_1$ and $\cH_2$, and hence, possibly the different underlying dimensions involved. Analogous formulas hold in the context of modified (or regularized) Fredholm determinants. This yields an effective analysis of outliers for small low rank perturbations of independent identically (i.i.d) matrices. 

On the contrary, {\bf random matrix theory and dimension expansion} (Subsection~\ref{sec6}) are natural in connection with the limit $N \to \infty$ of invariant ensembles of random $N \times N$ Hermitian matrices. 

For background on the notions of traces and infinite determinants, see, for instance, \cite[Chs.~IV, VI, IX]{GGK00}, \cite[Ch.~IV]{GK69}, \cite[Sect.~XIII.17]{RS78}, \cite[Chs.~3, 9]{Si05}, \cite[Ch.~3]{Si15}.

Next, we discuss the notion of {\bf Integrable Operators} (Subsection~\ref{sec7}) of the form
\begin{equation}\label{1.5}
K(z,z')=  \Bigg[ \sum^N_{j=1} f_j(z)\, g_j(z') \Bigg]\Bigg{/}(z-z'), \quad z, z' \in \Sg, 
\end{equation}
where $\Sg$ is an oriented contour in $\CC$, and $K$ acts in $L^2 \lf(\Sg, |dz|\rt)$. This class of operators was  identified by Its, Izergin, Korepin, and Slavnov \cite{IIKS90} in 1990. One then infers that also $R=(I-K)^{-1} - I$ is an integrable operator with kernel
\begin{align}
\begin{split} 
& R(z,z') = \Bigg(\sum^N_{j=1} \, F_j (z) \, G_j(z') \Bigg) \Bigg{/} (z-z'),   \\
& \quad \text{where } F_j = (I-K)^{-1} f_j, \quad G_j= \big(I-K^\top\big)^{-1} g_j, \quad
1\le j \le N.
\end{split}
\end{align}
Remarkably, $F_j$, $G_j$, and hence the inverse $(I-K)^{-1}$ of the integrable
operator $K$ can be computed in terms of a {\it canonical auxiliary Riemann--Hilbert Problem} (RHP), combined with commutation. In this context we note that the structure of operators inverse to integrable operators was already considered by L.\ A.\ Sakhnovich \cite{Sa68} in 1968; this was one of the first examples where the method of operator identities was employed. 

{\bf Poncelet's porism}\footnote{A type of proposition bridging the gap between a problem and a theorem, typically  describing a situation where a geometric construction is possible not just once, but in an infinite number of ways.} (Subsection~\ref{sec8}), states that whenever a polygon is inscribed in one conic section and circumscribes another one, the polygon must be part of an infinite family of polygons with the same property. An elegant generalization of Poncelet's porism, due to Gibson, Saldanha, and Tomei \cite{GST24}, involves commutation in a fundamental way.

Certain {\bf Interpolation estimates} (Subsection~\ref{sec9}) can be obtained via commutation, bypassing complex interpolation methods.

{\bf Eigenvalue Computations} (Subsection~\ref{sec10}) employing the $QR$ algorithm, naturally exploit commutation.

The {\bf Kardar--Parisi--Zhang (KPZ) Equation} (Subsection~\ref{sec11}), a foundational, nonlinear stochastic partial differential equation, introduced in 1986, models how random interfaces (such as, burning paper or crystal growth) grow and fluctuate over time. In this context, Amir, Corwin and Quastel \cite{ACQ11} analyzed the free energy of the continuum directed random polymer in $1+1$ dimensions -- once more this analysis uses commutation in a prominent manner, together with an associate RHP. Moreover, we recall that Cafasso and Claeys \cite{CC22} used the same RHP to evaluate precise lower tail asymptotics for the Cole--Hopf solution of the KPZ equation with narrow wedge initial data. In this context we also mention the specific reduction of an operator-valued RHP by Its, Bothner, Simon, and Kozlowsky \cite{IBSK26}, when they consider the Emptiness Formation Probability in the one-dimensional  impenetrable Bose gas. In particular, they use commutation $AB \to BA$ to reduce the operator-valued RHP in  \cite{IBSK26} to a standard $2\times 2$ RHP associated with the deformed sine-kernel integrable operator. 

{\bf Pfaffians, especially, Fredholm Pfaffians} (Subsection~\ref{sec3.8}) have found applications in a great variety of problems in mathematics and mathematical physics. In the case of a skew-symmetric $n \times n$ matrix $A$ there exists a function, $\Pf(A)$, the Pfaffian of $A$, which is a polynomial in the entries $A_{j,k}$ of $A$, such that Cayley's formula  
\begin{equation}
\Pf(A)^2 = \det(A)    \lb{1.7A} 
\end{equation}   
holds. An extension of this concept to infinite dimensions can proceed as follows: given a measure space $(X;d\mu)$, let 
\begin{equation}
K(x,y) = \begin{pmatrix} K_{1,1}(x,y) & K_{1,2}(x,y) \\ K_{2,1}(x,y) & K_{2,2}(x,y) \end{pmatrix}, \quad x, y \in X,
\end{equation}
be a $2 \times 2$ matrix-valued skew-symmetric kernel that induces an integral operator on $L^2(X; d\mu) \times L^2(X; d\mu)$, in the standard manner, 
\begin{align}
\begin{split} 
& (K f)(x) = \int_{X} K(y,x) f(y) \, d\mu(y), \quad \text{for a.e.~$x \in X$,}    \\ 
& f(\dott) = (f_1(\dott), f_2(\dott))^{\top} \in L^2(X; d\mu) \times L^2(X; d\mu). 
\end{split} 
\end{align}
In this context skew-symmetry means that 
\begin{align}
& K_{1,1}(x,y) = - K_{1,1}(y,x), \; K_{2,2}(x,y) = - K_{2,2}(y,x), \; K_{1,2}(x,y) = - K_{2,1}(y,x), \no \\
& \hspace*{8.5cm} \text{for a.e.~$x, y \in X$.}     \lb{1.10A} 
\end{align}
Introducing the operator 
\begin{equation}
(J f)(\dott) = (f_2(\dott), - f_1(\dott))^{\top}, \quad f = (f_1, f_2)^\top \in L^2(X; d\mu) \times L^2(X; d\mu),
\end{equation}
and assuming that $K$ is a trace class operator on $L^2(X; d\mu) \times L^2(X; d\mu)$, the expression 
\begin{equation}
\Pf(J-K) := I + \sum_{\ell=1}^{\infty} \f{(-1)^{\ell}}{\ell!} \int_{X^{\ell}} \Pf\big((K(x_j,x_k))_{j,k=1}^{\ell}\big) \, \prod_{j=1}^{\ell} d\mu(x_j)
\end{equation}
defines the {\it Fredholm Pfaffian of $K$}, were $\Pf(\dott)$ under the integral denotes the ordinary Pfaffian of the $2 \ell \times 2 \ell$ skew-symmetric matrix $(K(x_j,x_k))_{j,k=1}^{\ell}$.  

Direct computation (see \cite{OQR17}) then yields 
\begin{equation}
(\Pf(J-K))^2 = \det(I - M),     \lb{1.13A}
\end{equation}
where $\det(I - M)$ denotes the Fredholm determinant of the trace class operator 
\begin{equation}
M = \begin{pmatrix} - K_{2,1}(x,y) & -K_{2,2}(x,y) \\ K_{1,1}(x,y) & K_{1,2}(x,y) \end{pmatrix}.    \lb{1.14A} 
\end{equation}
Formula \eqref{1.13A} represents the analog of Cayley's formula \eqref{1.7A}. Subsection~\ref{sec3.8} then continues to make connections between $\det(I-M)$ and a canonical Riemann--Hilbert problem (RHP) in the context of the so-called symplectic derived class, following Bothner and Jaconelli \cite{BJ25}. Additional simplifications, leading to a RHP of Zakharov--Shabat-type, are briefly discussed.

In connection with $\det(I-M)$, and appropriate factorizations of $M$, the commutation determinant formula $\det(I - AB) = \det(I - BA)$ enters at a crucial point. 

{\bf The initial boundary value problem for the focusing Gross--Pitaevskii equation} (Subsection~\ref{sec3.9}), is concerned with 
solutions to the focusing Gross--Pitaevskii equation with a delta potential supported at the origin,
\begin{equation}
\begin{cases} i u_t + 2^{-1} u_{xx} + |u|^2 u + q \delta_0(\dott) u = 0,  \quad q \in \bbR,    \\
u(x,0) = v_{\lambda} (x) + w(x); \quad (x,t) \in \bbR^2,
\end{cases}     \lb{1.15A}
\end{equation}
where $|q|$ is small and $w(\dott)$ is even and of order $\Oh(q)$. Following the treatment by Holmer and Zworski \cite{HZ09}, $v_{\lambda}(\dott)$ has the special form,
\begin{equation}
v_{\lambda}(x) = \lambda \sech\big(\lambda |x| + \tanh^{-1}(q/\lambda)\big), \quad \lambda > |q|,
\end{equation}
and corresponds to the nonlinear ground state of a condensate variational problem in one dimension. Associated with $v_{\lambda}$ one has the stationary solution
\begin{equation}
u_{\lambda}(x,t) = e^{i \lambda^2 t/2} v_{\lambda}(x), \quad (x,t) \in \bbR^2,
\end{equation}
for \eqref{1.15A} corresponding to $w \equiv 0$. The main result in \cite{HZ09} concerns the stability of this ground state condensate under even perturbations $w= \Oh(q)$, $q \ll 1$. In the case where $u(x,0)$, and hence $u(x,t)$ is even in $x$, \eqref{1.15A} reduces to the initial boundary value problem (IBV) for the focusing NLS equation on the half-line,
\begin{equation}
\begin{cases}
i u_t + 2^{-1} u_{xx} + |u|^2 u = 0, &(x,t) \in (0,\infty)^2, \\
u(x,0) = u_0(x), &x \in (0,\infty),   \\
u_x(0,t) + q u(0,t) = 0, &t \in (0,\infty), \, \text{a Robin boundary condition at $x=0$.}
\end{cases}     \lb{1.18A} 
\end{equation}

As observed by A.\ Fokas \cite{Fo89}, \cite{Fo02}, the IBV \eqref{1.18A} is ``integrable'' in the sense that it can be solved by using linear equations only, see \cite{IS13}. Taking a different approach, Deift and Park \cite{DP11} utilize remarkable computations of Bikbaev, Khabibullin, and Tarasov to introduce a nonlinear method of images which extends the IBV problem in a quadrant to a problem on the full line to which the familiar Riemann--Hilbert/steepest descent method for the longtime behavior of NLS applies. 

In the context of isospectral deformations of the NLS Lax operator, the extreme case of ``auto-commutation,'' where $AB \longrightarrow BA=AB$, enters naturally. 

The subsections {\bf A Survey of Commutation Formulas; The Banach Space Case} (Subsection~\ref{sec12} )and {\bf A Survey of Commutation Formulas; The Hil- bert Space Case} (Subsection~\ref{sec13}) revisit commutation formulas for $AB$ and $BA$ from scratch in the Banach and Hilbert space setting, respectively, by making the connection with the $2 \times 2$ block operator matrix 
\begin{equation}
Q = \begin{pmatrix} 0 & B \\ A & 0 \end{pmatrix} 
\end{equation}
(sometimes called a supersymmetric charge in supersymmetric quantum mechanics). The identity  
\begin{equation}
Q^2 = \begin{pmatrix} BA & 0   \\   0 & AB \end{pmatrix} , 
\end{equation}
makes it plain that the (spectral) properties of operators $AB$, $BA$, and $Q$ are necessarily intimately intertwined, which is thoroughly exploited in these subsections.  

These abstract subsections are contrasted by a concrete situation, {\bf Applications to One-Dimensional Scattering Theory} (Subsection~\ref{sec14}), in which $A$ in $L^2(\bbR)$ is of the type $A = (d/dx) + \phi$, with $\phi$ a real-valued, locally absolutely continuous function that approaches asymptotes $\phi_{\pm} \in \bbR$ sufficiently fast.  Hence, $B=A^* = - (d/dx) + \phi$, and the associated Schr\"odinger operators $H_j$, $j=1,2$, in $L^2(\bbR)$ are then given by
\begin{align}
& H_1 = A^*A = - \f{d^2}{dx^2} + V_1, \quad \dom(H_1) = H^2(\bbR),     \lb{1.9} \\
& H_2 = AA^* = - \f{d^2}{dx^2} + V_2, \quad \dom(H_2) = H^2(\bbR),     \lb{1.10} \\
& V_j(x) = \phi(x)^2 + (-1)^j \phi'(x), \quad x \in \bbR, \; j=1,2.    \lb{1.11}
\end{align}
The maps 
\begin{equation} 
\phi \mapsto V_j = \phi^2 + (-1)^j \phi', \quad j=1,2,     \lb{1.12} 
\end{equation}
represent the celebrated Miura transforms. Stationary scattering theory for $H_1, H_2$ and the Dirac-type operator $Q = \left(\begin{smallmatrix} 0 & A^* \\ A & 0 \end{smallmatrix}\right)$ is then developed in great detail. 

This is followed with another concrete subsection, {\bf Applications to Floquet Theory} (Subsection~\ref{sec15}) which treats the case of one-dimensional periodic Schr\"o\-dinger operators $H_j$, $j=1,2$, and the periodic Dirac-type operator $Q$. In this context $A$ in $L^2(\bbR)$ is again of the type $A = (d/dx) + \phi$, with $\phi$ a real-valued, locally absolutely continuous function such that $\phi, \phi' \in L^{\infty}(\bbR)$ and for some $\omega \in (0,\infty)$, $\phi(x+\omega)=\phi(x)$, $x \in \bbR$, and $H_j$, $j=1,2$, and $Q$ are defined as in Subsection~\ref{sec14}. We then develop Floquet and spectral theory for $H_j$, $j=1,2$, and $Q$ in detail.  

Our final subsection on {\bf Applications to the $\KdV$ and $\mKdV$ Hierarchy} (Subsection~\ref{sec16}), then exploits the Miura transform \eqref{1.12} to connect solutions between the two hierarchies. After developing the $\KdV$ and $\mKdV$ hierarchies in detail, we show that for each $n \in \bbN_0$, 
\begin{equation}
\mKdV_n(\phi) = 0 \, \text{ implies } \, \KdV_n(V_j) = 0, \quad j=1,2.     \lb{1.13} 
\end{equation}
and conversely,
\begin{equation}
\KdV_n(V_1) = 0 \, \text{ implies } \, \KdV_n(V_2) = 0 \, \text{ and } \, \mKdV_n(\phi) = 0.    \lb{1.14} 
\end{equation}
Here, once again $V_j$, $j=1,2$, and $\phi$ are related via Miura's transformation \eqref{1.12}. 

In Section \ref{s4}, a brief epilogue, we comment on commutation methods and their mysterious connection to the uncertainty principle in quantum mechanics.

We necessarily had to choose a selection of results in Section~\ref{s3} and therefore left out a description of many very relevant results. The reader will find a host of applications of commutation formulas, for instance, in \cite{Ad94}, \cite{AU22}, \cite{AW25}, \cite{AC90}, \cite{Ar90}, \cite{AB87}, \cite{Ba86}, \cite{BL00}, \cite{BBW10}, \cite{DT79}, \cite{EF85}, \cite{Fl51}, \cite{Ge93}, \cite{GST96}, \cite{GS95}, \cite{GT96}, \cite{GW93}, \cite{G-UKM10}, \cite{Gu20}, \cite{HKRV23}, \cite{KST12}, \cite{LP86}, \cite{Mc81}, \cite{Mc85}, \cite{Mc86}, \cite{Mc87}, \cite{Mc89}, \cite{MvM75}, \cite{Mo23}, \cite{Oh87}, \cite{Oh88}, \cite{Oh95}, \cite{Oh99}, \cite{OU06}, \cite{ONR06}, \cite{Pa03}, \cite[Ch.~5]{PT87}, \cite{Pu86}, \cite{Pu86a}, \cite{Pu87}, \cite{Ry22}, \cite[Sect.~1.1]{Sa71}, \cite{Sa01}, \cite{Sa17}, \cite[Chs.~6--8]{SSR13}, \cite{Sc23}, \cite{Sc03}, \cite{Sc10}, \cite{Si09}, \cite{Su85}, \cite{Su85a}, \cite{Su86}, \cite{Tr88}, and \cite{VS93}, and the references cited therein.

\medskip
\noindent 
{\bf Notation.} By $X_1 \dotplus X_2$ we denote the direct sum of complex Banach spaces $X_j$, $j=1,2$. The complex Banach space of (everywhere defined) bounded linear operators from $X$ to $Y$, $X, Y$ complex Banach spaces, will be denotted by $\cB(X,Y)$. 

The inner product in a separable, complex Hilbert space $\cH$ is denoted by $(\,\cdot\,,\,\cdot\,)$ and is assumed to be linear with respect to the second argument; the symbol $I$ denotes the identity operator in $\cH$.   If $T$ is a linear operator mapping (a subspace of) a Hilbert space into another, then $\dom(T)$ and $\ran(T)$ denote the domain and range of $T$, respectively.  The resolvent set, spectrum, point, and discrete spectrum of a closed linear operator $T$ in $\cH$ are abbreviated by $\rho(T)$, $\sigma(T)$, $\sigma_p(T)$, and $\sigma_d(T)$, respectively.  The complex Banach space of bounded (resp., compact) linear operators on $\cH$ is denoted by $\cB(\cH)$ (resp., $\cB_{\infty}(\cH)$). For $p\in [1,\infty)$, the corresponding $\ell^p$-based trace ideals will be denoted by $\cB_p (\cH)$ with norms abbreviated by $\|\,\cdot\,\|_{\cB_p(\cH)}$. Finally, we use the short-hand notation $L^p((a,b)) = L^p((a,b); dt)$ whenever Lebesgue measure is understood.

\section{A Brief Look Back at \cite{De78}} \lb{s2}

In this section we briefly recall some of the applications discussed in P.\ A.\ Deift, {\it Applications of a commutation formula} \cite{De78}.

\subsection{Essential Isospectrality}\label{sec1}
\eqref{1.2} holds in the sense that if $0 \neq z \in \rho(BA)$,
then also $z \in \rho(AB)$ and 
\begin{equation} 
-\frac{1}{z} \lf(I-A (BA- z I)^{-1} B\rt)
\end{equation} 
is the inverse of $AB - z I$.  Interchanging $A \leftrightarrow B$, one concludes that
\begin{equation}\label{eq3}
\sg (AB) \backslash   \{0\} = \sg(BA) \backslash  \{0\}, 
\end{equation}
and thus $AB \to BA$ provides a canonical \tit{isospectral action}.

\subsection{Integrable Systems}\label{sec2}
We consider the one-dimensional Schr\"{o}dinger operator $H= H_0 + q(\dott)$ in $L^2(\RR)$, with 
$H_0 = - (d^2/dx^2)$, $\dom(H_0) = H^2(\bbR)$. 
If $H \varphi = 0$ for some positive function $\varphi > 0$, with $\varphi, \varphi' \in AC_{loc}(\bbR)$ (the locally absolutely continuous functions on $\bbR$), then $H$ can be factored into 
\begin{equation} 
H=-\frac{1}{\varphi} \frac{d}{dx} \varphi^2\frac{d}{dx} \frac{1}{\varphi}=BA = H_0 + q(\dott), 
\end{equation} 
where $A= \varphi \frac{d}{dx} \frac{1}{\varphi}$ and $B=A^*=-\frac{1}{\varphi} 
\frac{d}{dx} \varphi$. If $0\neq \lambda \in \sigma(H)$, then by \eqref{eq3},
$\lambda \in \sg\big(\widetilde{H}\big)$, where
\begin{equation} 
\widetilde{H} =AB = H_0 + q(\dott) - 2 \bigg(\frac{d^2}{dx^2} \log(\varphi(\dott))\bigg) 
= H - 2 \bigg(\frac{d^2}{dx^2} \log(\varphi(\dott))\bigg).
\end{equation} 
However, by \eqref{eq3} one notes that there is a loophole, namely, it is possible that 
$0\in \rho(H)$ but $0\in \sg (\widetilde{H})$.  Repeated iterations of this loophole makes it
possible to construct explicitly the class of reflectionless potentials $q_N$ supporting precisely $N \in \bbN_0$ strictly negative and simple eigenvalues $- \kappa_j^2 \in (-\infty,0)$, $1 \leq j \leq N$, for $H_N = H_0 + q_N(\dott)$ in $L^2(\bbR^n)$, with $q_N$ of the explicit form, 
\begin{equation}
q_N(x) = \begin{cases} - 2 \cfrac{d^2}{dx^2} \ln({\det}_{\bbC^N} (I_N + C_N(x))), &N \in \bbN,   \\
0, & N=0, 
\end{cases}  \quad x \in \bbR,    \lb{3.1a}
\end{equation}
where $I_N$ denotes the identity matrix in $\bbC^N$, $N \in \bbN$, and 
\begin{align} 
\begin{split} 
&C_N(x) = (C_{N,j,\ell}(x))_{1 \leq j,\ell \leq N}, \quad C_{N,j,\ell}(x) = c_j c_{\ell} (\kappa_j+\kappa_{\ell})^{-1} 
e^{-(\kappa_j+\kappa_{\ell})x}     \lb{3.2} \\
& c_j, \kappa_j  \in (0,\infty), \; \kappa_j \neq \kappa_{\ell} \, \text{ for } \, j \neq \ell, \; 1 \leq j, \ell \leq N, \; N\in \bbN, \; x \in \bbR.
\end{split} 
\end{align}
An appropriate $t$-dependence of $c_j$, $1 \leq j \leq N$, then yields $N$-soliton solutions for the $\KdV$ hierarchy, starting from $H_0 =-d^2/dx^2$. Indeed, replacing $c_j = c_j(0)$ by 
\begin{equation}
c_j(t) = c_j e^{\mp \sum_{m=0}^n c_{n-m} (-1)^m (-\lambda_j)^{m+(1/2)} t}, \quad t \in \bbR, \; 1 \leq j \leq N,    
\lb{8.1} 
\end{equation}
in \eqref{3.2}, gives rise to 
\begin{equation}
q_N(x,t) = - 2 \cfrac{d^2}{dx^2} \ln({\det}_{\bbC^N} (I_N + C_{N,\pm}(x,t))), \quad N \in \bbN, \; (x,t) \in \bbR^2,    \lb{8.2}
\end{equation}
satisfying\begin{equation}
\KdV_n(q_N) =0. 
\end{equation}
The family $q_N$, $N \in \bbN$, represents the celebrated {\it $N$-soliton $\KdV_n$ solutions}. 
Here the integration constants $c_0=1$, $c_{\ell} \in \bbR$, $\ell = 1,\dots,n$, are inherited from the $n$th $\KdV$ equation (cf.\  \eqref{1.2.20}, \eqref{1.2.22}). 

Using two commutations in succession (using the double commutation method), one can also insert Wigner--von Neumann eigenvalues $\lambda > 0$ into the continuous spectrum.

\subsection{Regularization of PDEs}\label{sec3}
The acoustic operator $H_b$ in $L^2(\RR^n)$ is given by $H_b=-\nab \cdot b^2\,\nab =
(b\nabla)^*\, b\nab$ where $b(x)$ is a strictly positive definite matrix satisfying
$c_1\le b(x)\le c_2$ for some $0<c_1 < c_2 < \infty$.  In the scattering theory for the 
acoustic equation 
\begin{equation} 
\ddot{u} + H_b \, u=0, 
\end{equation} 
one assumes that $b(x) \to I$ at a suitable rate as $|x| \to\infty$.  For unitarity of
the scattering operator in the trace class one needs that
\begin{equation}\label{eq4}
R\equiv (H_b+I)^{-1} - (-\De + I)^{-1}
\end{equation}
lies in an appropriate Schatten class $\mathcal{B}_p \lf(L^2\lf(\RR^n\rt) \rt)$ for some
$p \in [1,\infty)$.
Naive resolvent estimates using the  resolvent identity require some smoothness on $b(x)$, but
by \tit{commutation} one has
\begin{align}
\lf(H_b +I\rt)^{-1} = \lf((b\nab)^{*}\, b\nab +I \rt)^{-1} 
&=I -\nab^*b \lf(b\nab\,\nab^* b+I\rt)^{-1} \, b\nab \notag\\
&= I - \nab^* \lf(\nab \nab^* + b^{-2} I \rt)^{-1}\, \nab.  \label{eq5}
\end{align}
Inserting \eqref{eq5} into \eqref{eq4}, one finds
\begin{align}
R &= \nab^* \lf[ \lf(\nab\nab^* + I \rt)^{-1} - \lf(\nab\nab^* + b^{-2} I\rt) \rt] \nab\\
&= \nab^* \lf[ \lf( \nab\nab^* +I\rt)^{-1} \lf(b^{-2} I - I\rt) \lf(\nab\nab^*+b^{-2} I\rt)\rt]\nab,
\end{align}
and so commutation reduces $H_b$ to a perturbation of $H=-\Delta$ by a 
\tit{potential} term $\lf[b^{-2}(x)-1]\rt] I$.  This \tit{regularizes} the acoustic
scattering theory, allowing for scattering in rough environments without any smoothness
assumptions on $b(x)$.

\subsection{Operator Estimates in Statistical Mechanics and Quantum Field Theory}\label{sec4}
One can use commutation to prove certain basic estimates in statistical mechanics and quantum field theory, without
using the machinery of complex interpolation theory.  For example,
\begin{itemize}
\item[$(i)$] \tbf{Segal's Lemma.}  
\tit{
Let $A$ and $B$ be semi-bounded, self-adjoint
operators in a Hilbert space $\CH$ such that $A+B$ is essentially self-adjoint on
 $D(A) \cap D(B)$. Then 
\begin{equation} 
\| e^{-(A+B)}\|\quad \le\quad \|e^{-A/2}\;e^{-B}\,e^{-A/2}\|\quad \le \quad \|e^{-B}\;e^{-A}\|.
\end{equation} 
}
\item[$(ii)$] \tbf{The Golden--Thompson inequality.} \tit{Let $A$ and $B$ be as in Segal's Lemma
and assume that in addition $\CH$ is separable.  Then
\begin{equation} 
\tr \lf(e^{-(A+B)}\rt) \;\le\; \tr \lf(e^{-A/2}\,e^{-B}\, e^{-A/2}\rt).
\end{equation} 
In particular, if $e^{-A}\;$ is trace class, then $\;e^{-(A+B)}\;$ is trace class and
\begin{equation} 
\tr \lf(e^{-(A+B)} \rt) \;\le \; \tr \lf(e^{-A/2}\;\, e^{-B}\, e^{-A/2}\rt) \;\le\;
\tr \lf(e^{-B} \, e^{-A}\rt) \;< \; \infty.
\end{equation} 
}
\end{itemize} 

By way of illustration, the proof of Segal's Lemma using commutation, proceeds as follows: \\[1mm]
{\it Proof of Segal's Lemma.} For any nonnegative integer $n$, 
\begin{align}
& \Big\|\big[e^{- 2^{-n-1}A} e^{- 2^{-n}B} e^{- 2^{-n-1}A} \big]^{2n}\Big\| 
= \Big\|e^{- 2^{-n-1}A} e^{- 2^{-n}B} e^{- 2^{-n-1}A}\Big\|^{2n}   \no \\
& \quad = \big|\sup_{|\lambda| \in [0,\infty)} \big\{\lambda \in \sigma\big(e^{- 2^{-n-1}A} e^{- 2^{-n}B} e^{- 2^{-n-1}A}\big)\big\} \big|^{2n}     \no \\
& \quad = \big|\sup_{|\lambda| \in [0,\infty)} \big\{\lambda \in \sigma\big(e^{- 2^{-n}B} e^{- 2^{-n}A}\big)\big\} \big|^{2n}  \quad \text{(by commutation)}   \no \\
& \quad \leq \Big\|e^{- 2^{-n}B} e^{- 2^{-n}A}\Big\|^{2n}    \no \\
& \quad = \Big\|e^{- 2^{-n}A} e^{- 2^{-n+1}B} e^{- 2^{-n}A}\Big\|^{2n-1}  \quad \text{(since $\|T\|^2 = \|T^*T\|$)}    \no \\
& \quad = \Big\|\big[e^{- 2^{-n}A} e^{- 2^{-n+1}B} e^{- 2^{-n}A}\big]^{2n-1} \Big\| \leq \dots \leq 
\big\|e^{-2^{-1}A} e^{-B} e^{-2^{-1}A} \big\|   \no \\
& \quad = \sup_{|\lambda| \in [0,\infty)} \big\{\lambda \in \sigma\big(e^{- 2^{-1}A} e^{-B} e^{- 2^{-1}A}\big)\big\}    \no \\
& \quad = \sup_{|\lambda| \in [0,\infty)} \big\{\lambda \in \sigma\big(e^{-B} e^{- A}\big)\big\}  
 \quad \text{(by commutation)}   \no \\
& \quad \leq \big\|e^{-B}e^{-A}\big\|.
\end{align}
The result then follows from the Trotter product formula, that is, 
\begin{equation}
\big[e^{-2^{-n}A} e^{- 2^{-n+1}B} e^{-2^{-n}A}\big] 2^{n-1} \underset{n \to \infty}{\longrightarrow} e^{-(A + B)}
\end{equation}
in the strong operator topology. 

Similar arguments can be used to prove the Golden--Thompson inequality.

\section{A Variety of Developments in the Past Fifty Years} \lb{s3}

Many other applications of commutation have been recognized over the last fifty years 
and we will now describe a certain selection of them.

\subsection{Determinants and Dimension Reduction}\label{sec5}
If $\la \neq 0$ is an eigenvalue of $AB$, then by \eqref{eq3} it is also an eigenvalue of 
$BA$.  But more is true.  A simple calculation shows that if $\la$ is an isolated eigenvalue,
then the algebraic (as well as geometric) multiplicity of $\la$ for $AB$ is the same as the algebraic (as well as geometric) multiplicity of $\la$ for $BA$.  By Lidskii's theorem it follows that if either one of the bounded
operators $A,\;B$ is trace class (alternatively, if $AB$ and $BA$ are trace class), then we obtain a commutation proof of the well known identity 
\begin{equation}\label{1.4}
\det \lf(I+AB \rt) = \det \lf(I+BA\rt). 
\end{equation}
In linear algebra \eqref{1.4} is known as the \tit{Weinstein--Aronszajn identity} or sometimes the \tit{Sylvester determinant theorem}.
We note that the algebraic multiplicity for $\la=0$ can be different for $AB$
and $BA$, but this has no effect on the determinant.

Some 12 years or so ago, P.D. was giving a talk in which the commutation formula \eqref{1.2}
came up.  Somewhat tongue in cheek, he referred to the formula as ``the most important
identity in mathematics''.  Terry Tao was in the audience.  He took note and the next
day he showed how he could use \eqref{1.2}, and in particular \eqref{1.4}, to analyze
\tit{outliers} for small low rank perturbations of i.i.d.\ matrices, see \cite{Ta13}. .
For $X_n$ composed of i.i.d.\ entries $\lf((X_n)_{j,k}\rt)_{1 \le j,k \le n}\;$
one considers the characteristic polynomial equation
\begin{equation}\label{eq7}
\det \lf(n^{-1/2} \;X_n + A_n\, B_n-z I \rt) =0
\end{equation}
for some perturbation 
$\displaystyle{ \mathop{ \mathrm{A_n} }_{\begin{subarray}{c} \uparrow\\ n \times k 
\end{subarray}} }$, 
$\;\displaystyle{
\mathop{ \mathrm{B_n} }_{\begin{subarray}{c} \uparrow\\ k \times n \end{subarray}} }$.
By commutation, for $|z|$ sufficiently large, \eqref{eq7} reduces to 
\begin{equation}\label{eq8}
\det \lf(I + B_n \lf(n^{-1/2}\; X_n-z I \rt)^{-1} \, A_n \rt)=0.
\end{equation}
If $k$ is fixed, $\;k<n$, we see that the eigenvalue problem \eqref{eq7} in $\CC^n$
is \tit{reduced} to a problem in $\CC^k$ with $k$ fixed,  as $n$ is allowed to grow.
In particular if $A_n B_n=u_n\, v^*_n$ for vectors $u_n$, $v_n$ in $\CC^n$,
then \eqref{eq8} is reduced to a scalar problem 
\begin{equation} 
1+ \lf( \lf(n^{-1/2} \; X_n -z I \rt)^{-1}\, u_n,\; v_n\rt)=0, 
\end{equation} 
and the appearance, or non-appearance, of outliers can be analyzed effectively using 
elementary calculus.  The reduction \eqref{eq7}$\to$\eqref{eq8} is now standard
technology in random matrix theory to analyze outliers.

This is an example where commutation gives dimension reduction in a way  that is 
similar to the use of the Woodbury matrix identity in Kalman filter theory.

We will now encounter cases where commutation is used for dimension \tit{expansion}.

\subsection{Random Matrix Theory and Dimension Expansion}\label{sec6}
Invariant ensembles of random $N\times N$ Hermitian matrices $M$ have probability
distributions of the form 
\begin{equation} 
P_N(M)\,dM = \frac{1}{Z_N} \det \lf(w(M)\rt) dM,
\end{equation} 
where $w(x)\ge 0$ and $w(M)$ is given by the spectral calculus (see \cite{Me04}).
Here $dM$ is Lebesgue measure on the set $\cM$ of algebraically independent entries of $M=\lf\{M_{ij}\rt\}$
and $Z_N$ is the normalization coefficient.

If $g(x)$ is a bounded function on $\RR$, let 
\begin{equation} 
f(M) = \det \lf(I+g(M)\rt). 
\end{equation} 
The expected value of $f$
\begin{equation} 
\langle f \rangle = \int_{\cM} f(M) \, P_N (M)\, dM, 
\end{equation} 
is of special interest.  In particular if $\Om$ is a Borel  set in $\RR$
and $\chi_\Om(x)=1$ for $x \in \Om$ and $0$ otherwise,
then with 
\begin{equation} 
\langle f_\Om \rangle = \int_{\cM} \det \lf(I-\chi_\Om (M)\rt) P_N(M)\, dM,  \, \text{ and } \,  
g=-\chi_\Om, 
\end{equation} 
$\langle f_\Om \rangle$ is the probability that there are no eigenvalues of $M$ in $\Om$.
Standard calculations is random matrix theory lead to the result that
\begin{equation}\label{eq9}
\langle f_\Om \rangle = \det \lf( \delta_{jk} + \int_{\cM} \phi_j(x)\, \phi_k(x)\, g(x)\, dx\rt)_{
0\le j,\;k\le N-1}
\end{equation}
where $\phi_j(x) = P_j(x) \lf(w(x)\rt)^{\frac{1}{2}}, \; 0 \le j\le N-1$
and $P_j(x) = \ga_j\, \chi^j + \dots, \; 0 \le j\le N-1, \; \ga_j>0$,
are the orthonormal polynomials with respect to the weight $w(x)\, dx$ on $\RR$.

\tit{Commutation comes into the analysis in the following way.}  We are interested in 
the situation where $N\to\infty$.  We see from \eqref{eq9} that $\langle f_\Om\rangle$
is expressed in terms of matrices of larger and larger size.  Limits of this kind are
generically very difficult to control.  However, fortunately, commutation can be used
to express \eqref{eq9} in terms of determinants on a \tit{fixed} space.  Such limits
are, generally speaking, easier to control.

We proceed as follows:

\noindent

Let $A:L^2 (\RR) \to \CC^N$ denote the bounded operator given by 
\begin{equation} 
(Ah)_j = \int_{\bbR} \phi_j(x)\, g(x)\, h(x)\,dx,\quad 0\le j\le N-1,\; h\in L^2(\RR)
\end{equation}  
and let $B: \CC^N \to L^2(\RR)$ denote the bounded operator given by 
\begin{equation} 
(Ba)(x) = \sum^{N-1}_{j=0} \phi_j(x)\, a_j, \quad a=\lf(a_0, a_1, \dots, a_{N-1}
\rt)^{\top} \in \CC^N,
\end{equation} 
then $AB$ maps $\CC^N \to \CC^N$, and for $a \in \CC^N$, one finds
\begin{equation} 
\lf((AB)a\rt)_j = \sum^{N-1}_{k=0} \lf(\int_{\bbR} \phi_j(x) \, \phi_k(x) \,g(x)\,dx\rt)a_k, 
\quad 0\le j \le N-1, 
\end{equation} 
and so by \eqref{eq9} one concludes that 
\begin{equation}\label{eq10}
\langle f_\Om \rangle = {\det}_{\bbC^N} \lf(I_{\CC^N} +  AB\rt).
\end{equation}
On the other hand, for $h\in L^2(\RR)$, one finds
\begin{align}
(BA\; h)(x) &= \int_{\bbR} K_N (x,y)\; g(y)\;h(y)\,dy, \\
\noalign{\hbox{where}}
K_n (x,y) &= \sum^{N-1}_{ j=0} \,\phi_j(x)\;\phi_j(y),
\end{align}
and so by \eqref{1.4},  
\begin{align} \label{eq11}
\begin{split} 
& \langle f\rangle_{\Om} = {\det}_{\bbC^N} \lf(I_{\CC^N} + AB\rt) = {\det}_{L^2(\bbR)} \lf(I_{L^2(\RR)} + BA\rt)
= {\det}_{L^2(\bbR)} \lf(I_{L^2(\RR)} + K_N \, \chi_g\rt),   \\
& \quad \text{where } \;\chi_g\, h = gh.
\end{split} 
\end{align}
Formula \eqref{eq11} is due to Gaudin--Mehta \cite[Appendix, A.16]{Me04} and the above calculations follow Tracy and Widom \cite{TW94}.

Thus we have an expression for $\langle f\rangle_{\Om}$ in terms of an operator on a
\tit{fixed} space.  Whereas in the previous example \eqref{eq7} $\to$ \eqref{eq8}, 
commutation converts the problem into a lower dimensional one, 
commutation now converts the problem to one in a fixed larger, in fact, infinite dimensional problem.
This is dimension \tit{expansion}.  To control the limit as $N\to \infty$, we have to show that $K_N\:g$ converges in the trace norm in $L^2(\RR)$.
This can be done in many cases of interest using the Riemann--Hilbert/steepest-descent
method, introduced by Deift-Zhou \cite{DZ92}, \cite{DZ93}.

For example, for potentials $w(x) = e^{-N\, V(x)}$, where $V(x)$ is real analytic 
and
\begin{equation} 
\frac{V(x)}{\log (1+x^2)} \to +\infty \, \text{ as } \, |x|\to\infty, 
\end{equation} 
the associated equilibrium measure has the form $d\mu(x) = \psi(x)\,dx$ for some
continuous function $\psi(x) \ge 0$.  In particular if $a\in \RR$ is a point for 
which $\psi(a)>0$, then after scaling, for $g=-X_{\lf(a,\;a+\frac{x}{N\,\psi(a)}\rt)}$, $x>0$, 
one has  
\begin{align}
&\quad \text{The probability that there are no eigenvalues in the interval}    \no \\
& \quad \Om_N= \lf(a, \; a+ \frac{s}{N \, \psi(a)}\rt) \, \text{ equals } \, \langle f_{\Om_N} \rangle =\det \lf(I-\widehat{K}_n\rt),  \\
& \quad \text{where } \widehat{K}_N \lf(\xi, \eta\rt) = \frac{1}{N\,\psi(a)} 
K_N \lf(a+\frac{\eta}{N\,\psi(a)}, \; q+\frac{\eta}{N\,\psi(a)}\rt) \text{ acts on } L^2(0,\eta).   \no
\end{align} 
Using the RH-steepest descent method, one finds as $N\to \infty$
\begin{equation} 
\langle f_{\Om_N} \rangle \to \det \lf(I-K_s\rt), 
\end{equation} 
where, after rescaling, $K_s$ is the trace class operator acting on $L^2((0,1))$ with kernel
\begin{equation} 
K_s(\xi,\xi') = \frac{\sin \, s(\xi - \xi')}{\pi (\xi - \xi')}, \quad \xi, \xi'\in [0,1]. 
\end{equation} 
$K_s$ is an example of a so-called \tit{Integrable Operator} which can in turn be analyzed by 
Riemann--Hilbert problems, as we describe next.

\subsection{Integrable Operators} \label{sec7}   
We refer to Deift \cite{De99} for background on this topic. 
${}$ \\[1mm] 
\tbf{Step 1.}  \tbf{What is an integrable operator?}
\tit{Let $\Sg$ be an oriented contour in $\CC$}. By convention, the $(\pm)$-side lies to the left (resp.\ right) side as one moves along the contour in the direction of the orientation (see Figure~1 below).
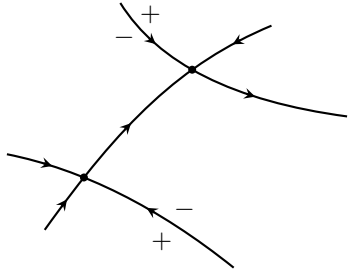
\begin{figure}[ht]
\centering
\begin{tikzpicture}[>=stealth, line width=0.8pt]

\draw (1.5,4) .. controls (2,3.2) and (3,2.7) .. (4.5,2.5);

\draw (0,2) .. controls (1,1.8) and (2,1.3) .. (3,.5);

\draw (.5,1) .. controls (1.5,2.4) and (2.3,3.2) .. (3.5,3.7);

\fill (2.4507,3.1177) circle (1.5pt);
\fill (1.0186,1.6926) circle (1.5pt);

\draw[->] (1.8781,3.5348) -- (1.9688,3.4563);
\draw[->] (3.1706,2.7972) -- (3.2850,2.7610);

\draw[->] (0.4857,1.8806) -- (0.6000,1.8440);
\draw[<-] (1.8466,1.2907) -- (1.9500,1.2298);

\draw[->] (0.7224,1.3064) -- (0.7946,1.4023);
\draw[->] (1.5590,2.3249) -- (1.6424,2.4112);
\draw[->] (3.0903,3.5112) -- (2.9850,3.4537);

\node at (1.9,3.85) {$+$};
\node at (1.55,3.55) {$-$};

\node at (2.35,1.265) {$-$};
\node at (2.05,0.85) {$+$};

\end{tikzpicture}
\caption{An oriented contour in $\bbC$.}
\end{figure}

We say that an operator $K$ acting in $L^2 \lf(\Sg, |dz|\rt)$ is \tit{integrable}
if it has a kernel of the form 
\begin{equation}\label{eq12}
K(z,z')=  \Bigg[ \sum^N_{j=1} f_j(z)\, g_j(z') \Bigg]\Bigg{/}(z-z'), \quad z, z' \in \Sg. 
\end{equation}
Such operators were singled out as a distinguished class by 
Its, Izergin, Korepin, and Slavnov in 1990, \cite{IIKS90}, and with appropriate conditions on $\Sg$, $f_i$, and $\,g_j$, $K\in \cB\lf(\,L^2(\Sg)\rt)$.

One notes that if $Ah(z) =zh(z)$ is the operator of multiplication by $z$, then the 
commutator of $A$ and $K$ is of the form 
\begin{equation} 
[A,K]=\sum^N_{j=1} \, f_j(z)\, g_j(z'). 
\end{equation} 
One has 
\begin{equation} 
\lf[A, (I-K)^{-1}\rt] = (I-K)^{-1} \,\lf[A, K\rt] (I-K)^{-1}, 
\end{equation} 
and one then concludes that $R=(I-K)^{-1} - I$ is also an integrable operator with kernel
\begin{align}
\begin{split} 
& R(z,z') = \Bigg(\sum^N_{j=1} \, F_j (z) \, G_j(z') \Bigg) \Bigg{/} (z-z'),   \\
& \quad \text{where } F_j = (I-K)^{-1} f_j, \quad G_j= \big(I-K^\top\big)^{-1} g_j, \quad
1\le j \le N,
\end{split}
\end{align}
where $K^\top$ denotes the real adjoint of $K$. 

\noindent
Remarkably, $F_j$, $G_j$, and hence the inverse $(I-K)^{-1}$ of the integrable
operator $K$ can be computed in terms of a \tit{canonical auxiliary Riemann--Hilbert Problem} (RHP), as follows. \\[1mm] 
\noindent
\tbf{Step 2.} \tbf{What is a RHP?}

\noindent
For $h\in L^2(\Sg)$, define the Cauchy operator with integration in the direction of the orientation, 
\begin{equation} 
(Ch)(z) = \frac{1}{2\pi i} \int_\Sg \, \frac{h(z')}{z'-z}\;dz',\quad z\in \CC \backslash  \Sg, 
\end{equation} 
and set 
\begin{equation} 
\lf( C_{\pm}\, h\rt) (z) = \lim_{\substack{
z'\to z\\
z' \in (\pm)-\text{side of} \, \Sg 
}}
\lf(C\, h\rt) (z'), \quad z\in \Sg.
\end{equation} 
This limit exists almost everywhere on $\Sg$ for (appropriate) contours.  Standard
computations show that
\begin{equation}\label{eq13}
C_{\pm} = \pm \frac{1}{2} I - \frac{1}{2} \:H,
\end{equation}
where $H$ is the the  Hilbert transform on $\Sg$, 
\begin{equation} 
(H h)(z)=\displaystyle{\lim_{\epsilon \downarrow 0}} \;\frac{1}{\pi i} 
\int_{\{z' \in \Sg \,|\, |z'-z| > \epsilon\}}
\:\frac{h(z)}{z-z'}\; dz', \quad z\in \Sg, 
\end{equation} 
and hence
\begin{equation}\label{eq14}
C_+- C_- = I \; ,\quad C_+ + C_- =-H.
\end{equation}
The Riemann--Hilbert Problem (RHP) $(\Sg, v)$ on $\Sg$ with \tit{jump matrix} 
$v:\Sg \to GL(k, \CC)$, where 
\begin{equation} 
v, v^{-1} \in L^\infty (\Sg), \, \text{ and } \, v(z) \to I \, \text{as} \, z\to\infty \text{ on } \Sigma,
\end{equation} 
consists of the following:  Show that there exists a (unique) $k\times k$ matrix function
$m=m(z)$ satisfying
\begin{equation}\label{eq15}
\begin{aligned}
&\bullet m(z) \text{ is analytic in } \CC/\Sg,    \\
&\bullet m_+(z) = m_-(z) v(z), \quad z\in \Sg,   \\
&\quad\text{where } m_\pm (z) = \lim_{\substack{
z' \to z  \\
z' \in (\pm)-\text{ side of } \Sg}}
\; m(z'), \\
&\bullet m(z) \to I \text{ as } z\to \infty.
\end{aligned} 
\end{equation}
The solution of the RHP follows from the solution of an associated singular integral operator equation 
on $\Sg$, as follows. For a given (pointwise) factorization of $v$ on $\Sigma$,   
\begin{equation} 
v=\lf(I-w_-\rt)^{-1} \, \lf(I + w_+\rt), 
\end{equation} 
define the operator 
\begin{equation}\label{eq16}
C_w \, h = C_+ (h\, w_-) + C_- (h\, w_+), \quad w=(w_+, w_-), 
\end{equation}
for $k\times k$ matrix-valued functions $h$ in $L^2 \lf(\Sg, |dz|\rt)$.  Now let
\begin{equation} 
\mu \in \, I + L^2 (\Sg)
\end{equation} 
be the solution of the singular integral equation on $\Sigma$, 
\begin{equation}\label{eq17}
\lf(I -C_w\rt)\,\mu = I, 
\end{equation}
and set
\begin{equation}\label{eq18}
m(z) \equiv I + C\lf( \mu \lf(w_+ + w_-\rt) \rt)(z), \quad z\in \CC \backslash  \Sg,
\end{equation}
where $C$ is the Cauchy operator on $\Sg$.  Then a straight forward computation using \eqref{eq14},
$C_+ - C_- =1$, shows that $m(z)$ \tit{solves the RHP} \eqref{eq15}. In particular, one finds that 
$m_{\pm}(z)=\mu(z)(I \pm w_{\pm})$ and so 
\begin{equation}
m_+(z) = m_-(z) (I-w_-)^{-1} (I + w_+) = m_-(z) v(z), \quad z \in \Sigma.
\end{equation}

\noindent
\tbf{Step 3.}  \tbf{Commutation}

\noindent
\tit{Commutation now enters} in the following way.  Let $K$ be an integrable operator as in \eqref{eq12}
and set $f= \lf(f_1, \dots, f_N\rt)^{\top}$, $g=\lf(g_1, \dots, g_N\rt)^{\top}$, and let $R_f$ denote the 
map of right multiplication by the column $N$-vector $f$, taking row $N$-vector 
functions to scalar functions
\begin{equation}\label{eq19}
\lf( R_f h \rt) = h(z) f(z) = \sum^N_{i=1}\, h_i(z) f_i(z), \quad h=\lf(h_1, \dots, h_N\rt). 
\end{equation}
Now let $R_{g^{\top}}$ denote the map of right multiplication by the row vector of $g^{\top}$
taking scalar functions to row $N$-vector functions
\begin{equation} 
\lf( R_{g^{\top}}\, h\rt) (z) = h(z) \, g^{\top}(z) = \lf(h(z)\, g_1(z), \dots, h(z)\, g_N(z) \rt).
\end{equation} 
Direct computation now shows that
\begin{equation}\label{eq20}
K h = \lf(D E\rt) h,
\end{equation}
where 
\begin{equation}\label{eq21}
D=R_f \text{ and } E=i\pi H R_g^{\top},
\end{equation}
with $H$ the Hilbert transform as above.

On the other hand, we find that $E D$ takes row $N$-vector functions $u$ to row $N$-vector 
functions,
\begin{equation} 
(E D)(u) = C_+ \lf(u \lf[ -i\pi f g^{\top}\rt]\rt) + C_- \lf(u \lf[-i\pi f g^{\top}\rt]\rt),
\end{equation} 
which is \tit{precisely the operator $C_w$ in \eqref{eq16}} associated with the solution of the 
$N\times N$ RHP $\lf(\Sg, v\rt)$, where $w_+ = w_- = -i \pi f g^{\top}$, and 
\begin{equation}\label{eq22}
v= \lf(I-w_-\rt)^{-1} \lf(I+w_+\rt) = I-\frac{2\pi i}{1+ i\pi \langle g,f\rangle}\, f  g^{\top}, 
\end{equation}
with 
\begin{equation} 
\langle g, f\rangle = \sum^N_{j=1} g_j(z) f_j(z).   \lb{3.44}
\end{equation} 

We now use commutation $DE \leftrightarrow ED$ to compute $F$ and $G$ in terms of the solution of the RHP 
$(\Sigma,v)$, where $v$ is given in \eqref{eq22}. First, we note from \eqref{eq14} and \eqref{eq21} that 
\begin{equation}
C_w(I) = C_+ \big(\big[-i\pi f g^{\top}\big]\big) + C_- \big(\big[-i\pi f g^{\top}\big]\big) = i \pi H \big(fg^{\top}\big).  \lb{3.45}
\end{equation}
One has 
\begin{align}
F &= (F_1,\dots,F_N)^{\top} = (I - K)^{-1} f     \no \\
&= f + R_f (I - C_w)^{-1} \big(i\pi H\big(R_{g^{\top}}\big)\big) f, \quad \text{by commutation}   \no \\
&= f +R_f (I - C_w)^{-1} \big(i\pi H\big(f g^{\top}\big)\big)    \no \\
&= f + R_f I - C_w)^{-1} (C_w I),    \quad \text{by \eqref{3.45}}    \no \\
&= f + R_f (-I) + R_f \big(\big(I - C_w)^{-1} I\big)   \no \\ 
&= R_f \mu,  \quad \text{by \eqref{eq17}}   \no \\
&= m_+ \big(I - i \pi f g^{\top}\big)^{-1} f = m_- \big(I + i \pi f g^{\top}\big)^{-1} f      \no \\
&= (I \mp i \pi \langle g, f \rangle)^{-1} m_{\pm} f,     \lb{3.46} 
\end{align}
with a similar formula for $G=\lf(G_1, \dots, G_N\rt)^{\top}$,
\begin{equation}
G = (I \pm i \pi \langle g, f \rangle)^{-1} \big(m^{\top}\big)_{\pm}^{-1} g     \lb{3.47} 
\end{equation}
(with $\langle \dott, \dott \rangle$ as in \eqref{3.44}).

For the integrable operator 
\begin{equation} 
K_s(z, z') = \sin \, s(z-z')/[\pi (z-z')] = 
e^{isz} \, e^{-i\,s\,z'} - e^{-s\,z} \,e^{i\,s\,z'}/[2\pi\, i(z-z')],
\end{equation} 
arising in the random matrix problem above, the RHP on the $\Sg=\RR$ with $v$ as in \eqref{eq22},
is $2\times 2$ with 
\begin{equation}\label{eq24}
v=\begin{pmatrix}
0 & e^{2i\,z\,s}\\
-e^{-2\,i\,z\,s} &2
\end{pmatrix}. 
\end{equation}
Suppose $P_s$ denotes the probability that there are no eigenvalues in the interval 
$\lf(a,\; a+ s/[N \, \psi(a)]\rt)$ with $N\to \infty$ as above, then
\begin{equation} 
\frac{d}{ds} \log(P_s) = \frac{d}{ds} \tr (\log \lf(I- K_s\rt)) = - \tr 
\lf((I-K_s)^{-1} \frac{d}{ds} K_s\rt),
\end{equation} 
and substituting for $F$, $G$ in $\lf(I-K_s\rt)^{-1}$ using the solution $m_\pm$ of the RHP
as in \eqref{3.46}, \eqref{3.47}, and utilizing the steepest-descent method for RHP to evaluate $m_\pm$ as 
$s\to\infty$, one obtains the celebrated result  
\begin{equation} 
\frac{d}{ds} \log (P_s) = \frac{-s^2}{2} + \frac{1}{4} \; \log(s) + \text{ const }
+ o(1) \text{ as $s\to \infty$}.
\end{equation}

\subsection{Poncelet's Porism}\label{sec8}
A critical element that makes commutation so useful is that the mapping $AB \to BA$ is 
isospectral (apart from $\la =0$).  As we will see, commutation can also be useful for 
similarity transforms, $A\to U\, AU^{\top}$.

An elegant example concerns a generalization of \tit{Poncelet's Porism} (see Gibson, Saldanha, and Tomei~\cite{GST24}, Hunziker, Martinez-Finkelshtein, Poe, and Simanek \cite{HM-FS21}, \cite{HM-FS22}, and  Martinez-Finkelshtein, Simanek, and Simon \cite{M-FSS19}). Poncelet's porism states that whenever a polygon is inscribed in one conic section and circumscribes another one, the polygon must be part of an infinite family of polygons with the same property (see Figure~2).  The notion of a porism goes back at least to Euclid.

Here are some examples:
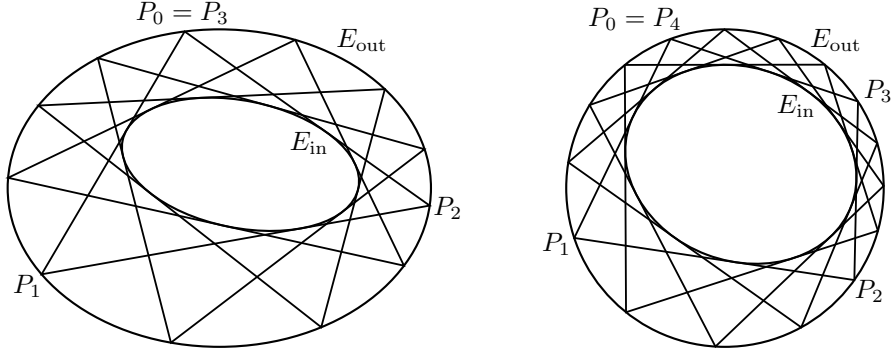
\begin{figure}[ht]   \lb{Fig.2}
\begin{tikzpicture}[scale=0.7, line width=0.7pt]

\draw[thick] (0,0) ellipse (4 and 3);
\node at (2.7,2.8) {$E_{\rm out}$};

\draw[thick, rotate around={-11.3758:(0.4,0.45)}]
  (0.4,0.45) ellipse (2.27263 and 1.20058);
\node at (1.7,.85) {$E_{\rm in}$};

\draw (-3.35856,-1.62943) -- (3.97565,-0.330502) -- (-0.655172,2.95948) -- cycle;

\draw (3.48958,-1.4664) -- (1.43473,2.80038) -- (-3.99158,0.194584) -- cycle;

\draw(1.92627,-2.62923) -- (3.12834,1.86951) -- (-3.41653,1.56016) -- cycle;

\draw(-0.914699,-2.92051) -- (3.87796,0.735383) -- (-2.296,2.45656) -- cycle;

\node at (-3.6236,-1.8580) {$P_1$};
\node at (4.3219,-0.3817) {$P_2$};
\node at (-0.6984,3.3068) {$P_0=P_3$};

\end{tikzpicture}
\qquad
\begin{tikzpicture}[scale=0.7, line width=0.7pt]

\draw[thick] (0,0) circle (3);
\node at (2.1,2.8) {$E_{\rm out}$};

\draw[thick, rotate around={-20.5561:(0.3,0.45)}]
  (0.3,0.45) ellipse (2.23049 and 1.82554);
\node at (1.3442,1.55) {$E_{\rm in}$};

\draw  (-0.180031,-2.99459) -- (2.99971,0.0415644) -- (1.01626,2.82262) -- (-2.55649,1.56984) -- cycle;

\draw (-1.87002,-2.34585) -- (2.8897,-0.805984) -- (1.88657,2.33256) -- (-1.89248,2.32777) -- cycle;

\draw (-2.84821,-0.942185) -- (2.4451,-1.73824) -- (2.5144,1.6364) -- (-1.02845,2.81821) -- cycle;

\draw  (1.44115,-2.63117) -- (2.87726,0.849352) -- (-0.00966471,2.99998) -- (-2.96022,0.486952) -- cycle;

\node at (-3.1805,-1.0521) {$P_1$};
\node at (2.7304,-1.941) {$P_2$};
\node at (2.9,1.8273) {$P_3$};
\node at (-1.7,3.2) {$P_0=P_4$};

\end{tikzpicture}

\caption{Poncelet's porism, $k=3$ and $k=4$.}
\end{figure}

Poncelet's porism in full generality, which was preceded by the results of Euler and Fuss 
for 3 and 4 vertices, is 
a result which continues to be studied today because of the connections to many other branches
of mathematics.  For example, Hitchin \cite{Hi92} has shown how
to use the porism to construct new and explicit solutions of the Painlev\'e VI equation.

Gibson, Saldanha, and Tomei~\cite{GST24} have generalized Poncelet's porism in the following way.  Consider 2 distinct 
ellipsoids $E_{\tbf{out}},\: E_{\tbf{in}}\subset \RR^n$ centered at the origin with $E_{\tbf{in}}
\subset \text{conv} (E_{\tbf{out}})$, the convex hull of $E_{\tbf{out}}$. 
By applying a linear transformation one can assume that
$E_{\tbf{out}} =\BSS^{n-1}$, the unit sphere.  Clearly there exists a unique positive symmetric matrix
$A$ with $A E_{\tbf{out}}=E_{\tbf{in}}$.

A closed, convex polytope $P\subset \RR^n$ \tit{fits} between $E_{\tbf{out}}$ and $E_{\tbf{in} }$ if 
$E_{\tbf{in}} \subset P \subset \text{conv} \lf(E_{\tbf{out}}\rt)$.  The polytope $P$ is 
\tit{inscribed} in $E_{\tbf{out}}$ if all its vertices belong to $E_{\tbf{out}}$ and $P$ is 
\tit{circumscribed} to $E_{\tbf{in}}$ if all its hyperfaces are tangent to $E_{\tbf{in}}$.  It fits 
\tit{tightly} between $E_{\tbf{out}}$ and $E_{\tbf{in}}$ if it is inscribed in $E_{\tbf{out}}$ and
circumscribed to $E_{\tbf{in}}$.   
We will now describe one of the results of Gibson, Saldanha, and Tomei ~\cite{GST24}.

A centrally symmetric parallelotope in $\RR^n$ is a convex polytope with $2^n$ vertices of the form
$\pm v_1\, \pm v_2\, \pm\, \dots\, \pm v_n$ where the vectors $v_1,\dots, v_n$ form a basis.  Thus
for $n=2$, the polytope is a parallelogram and for $n=3$, a parallelepiped.  A \tit{label}
for a parallelotope is a family $(v_k)_{i\le k\le n}$ of vectors as above so that each vertex
in turn is labeled by a sequence of signs.

\tit{ 
Let $P\lf( E_{\tbf{in}}, \, E_{\tbf{out}}\rt)$ be the set of all labeled parallelotopes 
fitting tightly between $E_{\tbf{out}}$ and $E_{\tbf{in}}$.  Define the map
\begin{equation} 
\phi: P\lf(E_{\tbf{in}}, E_{\tbf{out}}\rt) \to O(n) \; \text{$($the real orthogonal group\,$)$}
\end{equation} 
taking a parallelotope with label $(v_k)$ to the matrix $Q \in O(n)$ whose
columns are obtained from $(v_k)$ by Gram--Schmidt orthonormalization.
}
\begin{theorem}
Let $E_{\tbf{out}}= \BSS^{n-1}$ and  $E_{\tbf{in}} = A E_{\tbf{out}}$, as above.  The set
$P\lf(E_{\tbf{in}}, E_{\tbf{out}}\rt)$ is nonempty if and only if $\tr\big(A^2\big)=1$.
In this case, the map $\phi : P\lf(E_{\tbf{in}}, E_{\tbf{out}}\rt) \to O(n)$ 
is a diffeomorphism.
\end{theorem}

\tit{Commutation} enters the problem in the following way:  It provides a key estimate in 
the proof of the theorem that is a consequence of the following lemma.  The proof of the
lemma in turn is a simple consequence of the commutation formula.

\begin{lemma}
Let $\la >0$ and consider a positive definite matrix $B>0$ with $\tr (B)=1$
acting in $\RR^n$.  Let $M$ be a map from $\RR^n \to \RR^n$.  Then 
\begin{equation} 
\tr \lf(M \lf(M^{\top} B^{-1}\, M+\lambda I\rt)^{-1}\, M^{\top}\rt) \le \tr(B) = 1. 
\end{equation} 
\end{lemma}
\begin{proof}
Write
\begin{equation} 
\begin{aligned}
&\tr \lf(M \lf(M^{\top} B^{-1} M + \lambda I\rt)^{-1} M^{\top}\rt)    \\
& \quad = \tr \big(B K (LK + \lambda I)^{-1} L\big), \, 
\text{ where } \, K=B^{-\frac{1}{2}} M \, \text{ and } \, L = \big(B^{-\frac{1}{2}} M\big)^{\top}.
\end{aligned}
\end{equation} 
But by \eqref{1.2}, 
\begin{equation}
\lambda (KL + \lambda I)^{-1} + K (LK+\lambda I)^{-1} L = I,    \lb{3.55} 
\end{equation} 
and so 
\begin{align} 
\begin{split}
& \tr \Big( M \lf( M^{\top} B^{-1} \, M+\lambda I\rt)^{-1} M^{\top} \Big)   \\
& \quad = \tr (B) - \tr \big(B \lambda (KL + \lambda I)^{-1}\big)    \\
& \quad = \tr (B) - \tr\Big(\lambda B^{1/2} \big(B^{1/2} M M^{\top}\ B^{1/2} + \lambda I\big)^{-1} B^{1/2}\Big). 
\end{split}
\end{align} 
But the second matrix is positive definite and so we obtain
\begin{equation} 
\tr \Big(M\lf( M^{\top}\, B^{-1} M + \lambda I \rt)^{-1} M^{\top}\Big) \le \tr(B) = 1,
\end{equation} 
which proves the Lemma.
\end{proof}

\noindent
Now apply the Lemma to $B=A^2$.

Thus, one concludes that commutation (cf.\ \eqref{3.55}) is useful not only for spectral problems, but also for
similarity transforms, $A\to U\, A\, U^{\top}$.

\subsection{Interpolation}\label{sec9}
One can use commutation to prove the following result, without using complex interpolation.
\begin{theorem}
Let $A, B \in \cB(\CH)$ with $A\ge 0$.  Suppose that
$\|AB\| \le 1$ and $\|B A\| \le 1$. Then
\begin{equation}\label{eq25}
\|A^\alpha \,B\, A^{1-\alpha}\| \le 1 \ , \quad 0 \le \alpha \le 1. 
\end{equation}
\end{theorem}
\begin{proof}
We recall that since $\|T \| = \|T^*\|$ for any operator $T$, $\|A^\alpha B\,A^{1-\alpha}\| \le 1$ 
for $0 \le \alpha \le 1$ implies 
\begin{equation}\label{eq26}
\|A^\alpha B^* A^{1-\alpha}\| \le 1, \quad 0\le \alpha \le 1.
\end{equation}
(Indeed, one takes adjoints and replaces $\alpha \in [0,1]$ by $\beta= (1-\alpha) \in [0,1]$.)

By taking limits, it is enough to prove \eqref{eq25} for $\alpha =j/2^n$, $0 \le j
\le 2^n$.  This is clearly true for $n=0$, that is, $\alpha=0$ or $1$.  Assuming
\eqref{eq25} is true for $n$, we show that it is true for $n+1$, that is, for  
$\alpha = j/2^{n+1}$, $j=0, 1, \dots, 2^{n+1}$. 
For $j\in \{0, 1, \dots, 2^n\} \subset \{ 0, 1, \dots, 2^{n+1}\}$, one infers, using
$\|T\|^2= \|T^*\, T\|$ and the fact that $\|T\| =\sup \{ |\lambda| \,|\, \lambda \in \sigma (T)\}$ if $T=T^*$,
\begin{equation} 
\big\|A^{j/2^{n+1}} B A^{1- (j/2^{n+1})} \big\|^2 = \sup \Big\{|\lambda| \, \Big| \,  \lambda \in \sigma
\big(A^{1- (j/2^n)} B^* A^{j/2^n} BA\big) \Big\} .
\end{equation} 
But for \tit{any} $T$, $\sup \{ |\lambda| \,|\, \lambda \in \sg(T)\} \le \|T\|$ and hence
\begin{equation} 
\big\|A^{j/2^{n+1}} B\;A^{1- (j/2^{n+1})}\big\|^2 \le \big\|A^{1- (j/2^n)} B^*\;A^{j/2^n}\big\|
\;\|BA\| \le 1,
\end{equation} 
by the induction hypothesis as $0\le j \le 2^n$.

If $2^n <j \le 2^{n+1}$, 
\begin{equation} 
\big\|A^{j/2^{n+1}} B\:A^{1-(j/2^{n+1})}\big\| = \big\|A^{1- (j/2^{n+1})} B^*\:A^{j/2^{n+1}}\big\|
= \big\|A^{j'/2^{n+1}} B^*\:A^{1- (j'/2^{n+1})}\big\|, 
\end{equation} 
where $j'=2^{n+1}-j$.  As $2^n< j \le 2^{n+1}$, $0\le j' \le 2^{n+1}-2^n = 2^n$, 
the induction is true for $n+1$.  This proves \eqref{eq25}.
\end{proof}

\begin{remark}
One notes the similarity to the arguments used to prove Segal's Lemma in Subsection~\ref{sec4}. 
\hfill $\diamond$
\end{remark}

\subsection{Eigenvalue Computation}\label{sec10}
Suppose we want to compute the eigenvalues of an invertible symmetric matrix $A_0$.

\noindent
\tbf{Step 1.} Factor $A_0=Q_0 R_0$, where $Q_0$ is orthogonal and $R_0$ is upper triangular with positive diagonal entries.
Such a ``$QR$'' factorization is just Gram--Schmidt orthogonalization applied to the 
columns of $A_0$, starting from the left.  It is easy to see that such a factorization
is unique.

\noindent
\tbf{Step 2.} Set $A_1=R_0 Q_0$.  This is the ``$QR$'' step. As $0\notin \sg (A_0)$, $\sg(A_1)= \sg (A_0)$.  Moreover, $A_1=\lf(
Q_0^{\top} A_0\rt)Q_0 = Q_0^{\top} A_0 Q_0$, so $A_1$ is also symmetric.

\noindent
\tbf{Step 3.} Factor $A_1=Q_1 R_1$ and set $A_2=R_1 Q_1$.  Again by commutation,
$\sg(A_2) = \sg(A_1) = \sg(A_0)$ and $A_2=A^{\top}_2$.

\noindent
\tbf{Step 4.}  Factor $A_2=Q_2 R_2$ and set $A_3 = R_2 Q_2$, etc.
Continuing, we obtain a sequence $A_0, A_1,\ A_2, \dots, A_n, \,\dots$
of symmetric matrices, all with $\sg(A_n)=\sg(A_0)$.  Now generically,
$A_n \to A_\infty$ as $n\to\infty$, where $A_\infty$ is a diagonal matrix. 
Necessarily, the diagonal entries of $A_\infty$ are the eigenvalues of $A_0$.

This elegant scheme is the famous so-called $QR$ \tit{algorithm} introduced by
Francis and (independently) by  Kublanovska in the late 1950's.  With suitable
modifications (use ``shifts'') the $QR$ algorithm is the go-to algorithm for the 
computation of the eigenvalues of mid-size matrices of dimension $< 500$.  The algorithm
holds center-stage in LINPACK, the go-to collection of methods in numerical linear
algebra.

The algorithm has the structure of an integrable Hamiltonian system, and has many other
remarkable properties, in particular it is intimately connected to the Toda flow, see Deift, Dubach, Tomei, and Trogdon \cite{DDTT25}.

\subsection{The Kardar--Parisi--Zhang (KPZ) Equation}\label{sec11}
In 2011, Amir, Corwin and Quastel \cite{ACQ11} considered the free energy $\msF(T,X)$ of the
continuum directed random polymer in $1+1$ dimensions.  Associated with $\msF,\,
\CL(T,X) \equiv P(T,X) \exp \lf(\msF(T,X)\rt)$ with $P(T,X) = \frac{1}{\sqrt{2\pi\,T}} e^{-\chi^2/2T}\;$, solves the stochastic heat equation
\begin{align} \label{eq27}
\p_T\, \CL = \frac{1}{2}\;\p^2_x \, \CL - \CL\, \dW, \, \text{ with initial condition } \CL(T=0,X)= \de_{x=0}. 
\end{align}
Here $\dW (T,X)$ is Gaussian space-time white noise such that
\begin{equation} 
E\big( \dW (T, X), \;\dW (S,Y)\big) = \de(T-S) \; \de(Y-X). 
\end{equation} 
Corresponding to $\CL$, $h=-\log(\CL(T,X))$ is the Cole--Hopf solution of the KPZ equation, 
\begin{equation} 
\p_T\,h = -\frac{1}{2} \lf(\p_x\, h\rt)^2 + \frac{1}{2} \p^2_x h + \dW,
\end{equation} 
with so-called narrow wedge initial conditions.

In their paper the authors derive a remarkable formula for the distribution function
$F_T(s)$ for $\msF(T,X)$,  
\begin{align}\label{eq28}
\begin{split} 
F_T(s) &= P\lf(\msF(T,X) + [T/4!] \le s\rt)    \\
&= \int_{\tCC}
\, \frac{d  \tmu}{\tmu} \, e^{-\tmu} 
{\det}_{L^2((K^{-1}_T\, a, \infty))} \lf(I-K_{\sg_{(T,\tmu)}} \rt). 
\end{split} 
\end{align}
Here
\begin{align}
&\sg_{\lf(T_1,\,\tilde{\mu}\rt)}(t) = \frac{\tilde{\mu}}{\tilde{\mu}-e^{-K_T\,t}},    \\
&K_T=2^{-1/3}\, T^{1/3}, 
\end{align} 
and for any function $\sg(\dott)$ which is smooth except for a finite number of bounded
jumps, and $\sg(t)\to 0$ as $t\to -\infty$ and $\sg(t) \to 1$ as $t \to +\infty$, 
\begin{equation} 
K_\sg (x,y) = \int^\infty_{-\infty} \sg (t) \, A_i (x+t) \, A_i(y+t) \quad \text{($A_i(\dott)$ is the Airy function)}. 
\end{equation} 
One notes that for $\sg(t) = \chi_{t\ge 0}$, $K_{\sg(t)} = K_{A_i} \text{ equals the Airy kernel}$.

Commutation enters the analysis through the following formula for $K_\sg$.
\begin{lemma}
For any function $\sg(\dott)$ as above, 
\begin{equation}\label{eq29}
{\det}_{L^2((s, \infty))} \lf(I-K_\sg\rt) = {\det}_{L^2((-\infty, \infty))} \lf(I-\widehat{K}_s\rt), 
\end{equation}
where $\widehat{K}_s (x,y) = \sqrt{\sg(x-s)}\; K_{A_i}(x,y) \; \sqrt{\sg(g-s)}$. \\
In particular, \eqref{eq29} holds for $K_{\sg_{(T,\tmu)}}$ in \eqref{eq28}.
\end{lemma}
\begin{proof}
Let  $\chi_s(x) = 1_{x \ge s}$ and let 
\begin{equation} 
L_s: \begin{cases} L^2 ((s,\infty)) \to L^2 ((-\infty, \infty))  \\
f \mapsto (L_s f)(x) = \int^\infty_s\, A_i (x+y) f(y) \, dy. 
\end{cases} 
\end{equation} 
Then,  
\begin{align}
\chi_s(x) \:L_{-\infty} \, \sg \;L_s\, f &= \chi_s(x) \int^\infty_{-\infty}\;
A_i(x+u)\, \sg(u) \, \lf(L_s\, f\rt)(u)\, du   \no \\
&= \chi_s(x) \int^\infty_{-\infty} A_i (x+u) \, \sg (u) \, \bigg(\int^\infty_s A_i
(u+y)\, f(y)\, dy \bigg) du    \no \\
&= \chi_s(x) \int^\infty_s \lf(\int^\infty_{-\infty} A_i (x+u) \,
\sg (u)\, A_i(y+u)\, du\rt) f(y)\, dy    \no \\
&= \chi_s(x) \int^\infty_s \, K_\sg (x,y)\, f(y)\, dy     \no \\
&= K_\sg \, f(x), 
\end{align}
with $K_\sg: L^2((s, \infty)) \to L^2((s, \infty))$, that is, $K_\sg= \chi_s L_{-\infty} \sg L_s = A B$, where 
$A=\chi_s\, L_{-\infty} \sqrt{\sg}$, $B=\sqrt{\sg} L_s$.
On the other hand, 
\begin{equation} 
\begin{aligned}
B\:A f &= \lf(\sqrt{\sg} L_s \chi_s L_{-\infty} \sqrt{\sg} f\rt)(x)   \\
&= \sqrt{\sg(x)}\int^\infty_s A_i (x+u) \chi_s(u)\lf(L_{-\infty} \sqrt{\sg} f \rt)(u)\,du    \\
&=\sqrt{\sg(x)} \int^\infty_s A_i (x+u) \lf(\int^\infty_{-\infty} A_i(u+y) \sqrt{\sg(y)} f(y)\,dy\rt) du \\
&= \sqrt{\sg(x)} \int^\infty_{-\infty} \lf(\int^\infty_s \:A_i (x+u) A_i(y+u) \,du \rt) \sqrt{\sg(y)} f(y)\, dy\\
&= \int^\infty_{-\infty} \lf[ \sqrt{\sg(x)} \lf( \int^\infty_0 A_i (x+u+s) A_i (y+u+s)\,du\rt) \sqrt{\sg(y)} \rt] f(y)\,dy. 
\end{aligned}
\end{equation} 
By commutation, the operator 
\begin{equation} 
D \equiv \sqrt{\sg(x)} \lf(\int^\infty_0
A_i (x+u+s) A_i(y+u+s)\,du\rt) \sqrt{\sg(y)} \, dy\, \text{ on $L^2((-\infty,\infty))$,} 
\end{equation} 
has the same spectrum as $K_\sg$ on $L^2((s, \infty))$.
But $D$ is unitarily equivalent to $\widehat{K}_s (x,y)
= \sqrt{\sg(x-s)} K_{A_i}(x,y) \sqrt{\sg (y-s)}$ on $L^2 ((-\infty, \infty))$ and 
the Lemma follows.
\end{proof}

\begin{remark}
Cafasso and Claeys \cite{CC22} used the RHP associated with the integrable 
operator $\widehat{K}_s (x,y) = \sqrt{\sg (x-s)} K_{A_i} \sqrt{\sg (y-s)}$ 
to evaluate precise lower tail asymptotics for 
the Cole--Hopf solution $h(T,X)$ of the KPZ equation with narrow wedge initial data.  More
precisely, they evaluated the behavior of $\log(P_{KPZ} \lf(Y_T < -s\rt))$ as 
$s\to +\infty$, where $Y_T= [h(2T, 0)+ T/12]/T^{1/3}$. 
\hfill $\diamond$
\end{remark}
 
A similar reduction of an operator-valued RHP arises in the work of Its, Bothner, Simon, 
and Kozlowsky \cite{IBSK26}, when they consider the Emptiness Formation Probability in the 1D impenetrable
Bose gas. Here they use $AB \to BA$ to reduce an operator-valued RHP to a standard $2\times 2$
RHP associated with the deformed sine-kernel integrable operator
\begin{equation} 
\frac{\gamma}{\pi} \sqrt{W_t(\lambda)} \frac{\sin (\lambda-\mu)}{\lambda-\mu}
\sqrt{W_t(\mu)}, \quad 0< \gamma <1, 
\end{equation} 
for some explicit function $W_t(\dott)$.

\subsection{Pfaffians}\label{sec3.8}
We recall the following result of Cayley (see, e.g., \cite{Ca47}): If $A \in \bbC^{n \times n}$ is a skew-symmetric matrix, then there exists a function, $\Pf(A)$, the {\it Pfaffian of $A$}, which is a polynomial in the entries $A_{j,k}$ of $A$, such that 
\begin{equation}
\Pf(A)^2 = \det(A). 
\end{equation}   
Going back to Pfaff's analysis of differential systems in 1815, Pfaffians have found applications in a great variety of mathematical and physical problems, including perfect matchings on planar graphs, tilings of the plane by dimers, the solution of the two-dimensional Ising model on planar graphs, and in differential geometry where the Pfaffian serves as a characteristic form related to the topology of oriented manifolds. In addition, Pfaffians play an important role in random matrix theory and related areas: this is the subject of this subsection.  

Pfaffians can be extended to infinite dimensions as {\it Fredholm Pfaffians} as follows (see, \cite{BJ25}, \cite{FTZ22}, \cite{Ra00}, and \cite{OQR17}). Let $(X;d\mu)$ be a measure space and let 
\begin{equation}
K(x,y) = \begin{pmatrix} K_{1,1}(x,y) & K_{1,2}(x,y) \\ K_{2,1}(x,y) & K_{2,2}(x,y) \end{pmatrix}, \quad x, y \in X,
\end{equation}
be a $2 \times 2$ matrix-valued skew-symmetric kernel that induces an integral operator on $L^2(X; d\mu) \times L^2(X; d\mu)$,
\begin{align}
\begin{split} 
& (K f)(x) = \int_{X} K(y,x) f(y) \, d\mu(y) \, \text{ for a.e.~$x \in X$,}    \\ 
& f(\dott) = (f_1(\dott), f_2(\dott))^{\top} \in L^2(X; d\mu) \times L^2(X; d\mu). 
\end{split} 
\end{align}
Skew-symmetry in this context means that 
\begin{align}
& K_{1,1}(x,y) = - K_{1,1}(y,x), \; K_{2,2}(x,y) = - K_{2,2}(y,x), \; K_{1,2}(x,y) = - K_{2,1}(y,x) \no \\
& \hspace*{8.5cm} \text{for a.e.~$x, y \in X$.}     \lb{3.74} 
\end{align}
Denoting by $J$ the operator 
\begin{equation}
(J f)(\dott) = (f_2(\dott), - f_1(\dott))^{\top}, \quad f = (f_1, f_2)^\top \in L^2(X; d\mu) \times L^2(X; d\mu),
\end{equation}
then, if $K$ is a trace class operator on $L^2(X; d\mu) \times L^2(X; d\mu)$,
\begin{equation}
\Pf(J-K) := I + \sum_{\ell=1}^{\infty} \f{(-1)^{\ell}}{\ell!} \int_{X^{\ell}} \Pf\big((K(x_j,x_k))_{j,k=1}^{\ell}\big) \, \prod_{j=1}^{\ell} d\mu(x_j)
\end{equation}
defines the {\it Fredholm Pfaffian of $K$}. Here $\Pf(\dott)$ under the integral denotes the ordinary Pfaffian of the $2 \ell \times 2 \ell$ skew-symmetric matrix $(K(x_j,x_k))_{j,k=1}^{\ell}$.  

Direct computations (see \cite{OQR17}) yield that 
\begin{equation}
(\Pf(J-K))^2 = \det(I - M),     \lb{3.77}
\end{equation}
where $\det(I - M)$ denotes the Fredholm determinant of the trace class operator $M$, 
\begin{equation}
M = \begin{pmatrix} - K_{2,1}(x,y) & -K_{2,2}(x,y) \\ K_{1,1}(x,y) & K_{1,2}(x,y) \end{pmatrix}.    \lb{3.78} 
\end{equation}
Formula \eqref{3.77} represents the analog of Cayley's formula $\Pf(A)^2 = \det(A)$. Most importantly, \eqref{3.77} turns asymptotic questions for the Fredholm Pfaffian $\Pf(J-K) $ into asymptotic questions for the Fredholm determinant $\det(I - M)$, $\Pf(J-K) = \det(I - M)^{1/2}$, a subject about which much is known. 

In \cite{BJ25} the authors consider operators $K$ on $L^2(\Delta)= L^2(\Delta;dx)$, where 
\begin{equation}
\Delta = \bigcup_{j=1}^m (a_{2j-1}, a_{2j}) \subset \bbR, \quad -\infty < a_1 < a_2 < \cdots < a_{2m-1} < a_{2m} < \infty, 
\end{equation}
with the property that $\Pf(J-K)$ or the associated Fredholm determinant $\det(I - M)$ is computable in terms of a canonical RHP. We will discuss one class of such operators -- the {\it symplectic derived class}; for other classes, see \cite{BJ25}. As we will see, commutation plays a key role in this context. 

We say that $K = (K_{j,k})$ is {\it of symplectic derived type} if $K_{j,k}$ in \eqref{3.74} are of the form 
\begin{align}
& K_{1,2}(y,x) = - K_{2,1}(x,y) = S(x,y),    \no \\
& K_{2,2}(x,y) = \partial_y S(x,y), \\ 
& \partial_x K_{1,1}(x,y) = S(x,y); \quad (x,y) \in \Delta \times \Delta.    \no 
\end{align}
Under mild conditions on $S(\dott,\dott)$, $K_{j,k}(\dott,\dott)$ in \eqref{3.78} induce trace class operators on $L^2(\Delta)$. Given these definitions, $M$ in \eqref{3.78} takes on the form
\begin{equation}
M = \begin{pmatrix} S & G \\ H & S^{\top} \end{pmatrix},    \lb{3.81}
\end{equation}
where 
\begin{equation}
G(x,y) = - \partial_y S(x,y), \quad \partial_x H(x,y) = S(x,y),
\end{equation}
and $S^{\top}$ denotes the real adjoint of $S$.

Commutation now implies the following result:

\begin{lemma} \lb{l3.7}
Denote by $D_M(\Delta)$ the Fredholm determinant on $L^2(\Delta) \times L^2(\Delta)$ of $M$, with $M$ as in \eqref{3.81},  given the aforementioned hypotheses placed on $M_{j,k}$. Then one has the identity,
\begin{equation}
D_M(\Delta) = D_{2S}(\Delta) \det\big((\delta_{j,k} - F_{j,k}(\Delta))_{j,k=1}^{2m}\big).     \lb{3.83} 
\end{equation}
Here
\begin{equation}
F_{j,k} = (-1)^k \Big(\big(I- 2 S^{\top}\big)^{-1} H\Big)(a_k,a_j),    \lb{3.84} 
\end{equation}
provided $I- 2 S$ is boundedly invertible in $L^2(\Delta)$ and $D_{2S}(\Delta)$ represents the Fredholm determinant of $2S$, with $S$ a trace class operator on $L^2(\Delta)$. 
\end{lemma}

Formula \eqref{3.83} is a significant simplification of $D_M(\Delta)$ as the right-hand side of \eqref{3.83} involves a finite-size determinant and a Fredholm determinant on $L^2(\Delta)$ instead of $L^2(\Delta) \times L^2(\Delta)$. Further simplification of \eqref{3.83} requires more assumptions on the main kernel $S(\dott,\dott)$ as will be shown.

\begin{proof}[Proof of Lemma~\ref{l3.7}]
Employing the properties of $S$, one observes that $M$ can be expressed as 
\begin{equation}
M = \begin{pmatrix} \partial_x & 0 \\ 0 & I \end{pmatrix} \begin{pmatrix} H & S^{\top} \\ H & S^{\top} \end{pmatrix}.    \lb{3.85}
\end{equation} 
On the other hand, 
\begin{align}
C &\equiv \begin{pmatrix} H & S^{\top} \\ H & S^{\top} \end{pmatrix} \begin{pmatrix} \partial_x & 0 \\ 0 & I \end{pmatrix}     \no \\
&= \begin{pmatrix} S^{\top} + \sum_{k=1}^{2m} (-1)^k H(\delta_{a_k} \otimes \delta_{a_k}) & S^{\top} \\[1mm] 
S^{\top} + \sum_{k=1}^{2m} (-1)^k H(\delta_{a_k} \otimes \delta_{a_k}) & S^{\top}   \end{pmatrix}.    \lb{3.90A} 
\end{align}
Here 
\begin{equation}
(H(\delta_{a_k} \otimes \delta_{a_k}) f)(x) = H(x,a_k) f(a_k),
\end{equation}
and the expression \eqref{3.90A} for the operator $C$ is obtained by integrating 
\begin{equation}
(H f_x)(x) = \int_{\Delta} H(x,y) f_y(y) \, dy
\end{equation} 
by parts for smooth functions $f$. It follows by commutation that $C$, and consequently, $N=\sigma_3\sigma_1 C \sigma_1\sigma_3$, are isospectral to $M$ (with $\sigma_j$, $1\leq j \leq 3$, the standard Pauli matrices). 

Hence, 
\begin{equation}
D_M(\Delta) = D_N(\Delta) = \begin{pmatrix} I - S^{\top} & S^{\top} + \sum_{k=1}^{2m} (-1)^k H(\delta_{a_k} \otimes \delta_{a_k}) 
\\[1mm]
S^{\top} & I - S^{\top} - \sum_{k=1}^{2m} (-1)^k H(\delta_{a_k} \otimes \delta_{a_k})  \end{pmatrix}. 
\end{equation}
Adding the first row to the second row, and then subtracting the second column from the first column results in \eqref{3.83}, \eqref{3.84}. 
\end{proof}

In order to study the asymptotic behavior of $\Pf(J-K)$ as the size of the intervals $\Delta$ becomes large, the authors in \cite{BJ25} specify $S(\dott,\dott)$ to the following form
\begin{equation}
S(x,y) = \f{1}{2} \bigg[\int_0^{\infty} \phi(x+u) \phi(u+y)\,du - 2^{-1} \phi(x) \int_y^{\infty} \phi(v) \, dv\bigg], \quad 
x,y \in \Delta = (t,\infty),    \lb{3.90} 
\end{equation}
where $\phi \in C^1(\bbR)$. Under appropriate conditions on $\phi$ and $\phi'$, the expressions $S(x,y)$, $G(x,y)=- \partial_y S(x,y)$, and $H(x,y) = - \int_x^{\infty} S(u,y) \, du$ are admissible choices for a trace class operator $M$ as in \eqref{3.85}. For such $S$, $G$, and $H$, set 
\begin{equation} 
D_M(t) = \lim_{a_2 \to \infty} D_M((t,a_2)), \quad t \in \bbR. 
\end{equation} 
Then one obtains 
\begin{equation}
D_M(t) = D_Q(t) \bigg[1 + \f{1}{2} \int_t^{\infty} ((I-Q)^{-1} \phi)(x) \Phi(x)\, dx\bigg], \quad \Phi(x) = \int_x^{\infty} \phi(y) \, dy, 
\lb{3.91}
\end{equation}
provided $I - 2S$ and $I-Q$ are boundedly invertible on $L^2((t,\infty))$. Here
$Q \in \cB_1(L^2((t,\infty)))$ is the trace class operator with integral kernel
\begin{equation}
Q(x,y) = \int_0^{\infty} \phi(x+u) \phi(u+y) \, du, \quad x,y \in (t,\infty),
\end{equation}
and $D_Q(t)$ represents the Fredholm determinant of $Q$ on $L^2((t,\infty))$. 

Formula \eqref{3.91} was obtained by Krajenbrink \cite{Kr21} in 2021 for functions $\phi$ as above. In the case that $\phi(\dott)$  is the Airy function $Ai(\dott)$, \eqref{3.91} was already known to Tracy and Widom since the mid 1990's, see \cite{TW96}. In \cite{BJ25} the authors obtain \eqref{3.91} directly from \eqref{3.83} by an application of the Sherman--Morrison identity.

From \eqref{3.91} one infers that the asymptotic behavior as $t \to \infty$ of Fredholm determinants, and hence, Pfaffians, in the symplectic derived class with $S(\dott,\dott)$ given by \eqref{3.90}, reduces to analyzing $D_Q(t)$ as $t \to \infty$, where 
\begin{equation}
Q(x,y) = \int_0^{\infty} \phi(x+u) \phi(u+y) \, du.
\end{equation}
In the special case where $\phi(x) = Ai(x)$, $Q(\dott,\dott)$ is an integrable operator given by 
\begin{equation}
Q(x,y) = \f{Ai(x) Ai(y) - Ai'(x) Ai'(y)}{x-y}
\end{equation}
and so $D_Q(t)$ can be analyzed as $t \to \infty$ by the familiar Riemann--Hilbert methods as in Subsection~\ref{sec7}. If $\phi(\dott)$ can be analytically continued into $\bbC$ in an appropriate manner, it is possible to convert $Q$ into an integrable operator in Fourier variables for which Riemann--Hilbert methods apply: this is the approach of Bertola and Cafasso \cite{BC12} and also of Baik and Bothner \cite{BB20}. For general $\phi$, however, the evaluation of $D_Q(t)$ as $t \to \infty$, presents a significant new challenge, and in remarkable, ingenious work, Krajerbrink \cite{Kr21}, inspired by earlier work of Le Doussal, Majumdar and Schehr, showed how to relate (in fact, quite mysteriously) $D_Q(t)$ to the solution of a RHP for the Zakharov--Shabat system. More precisely (see \cite{Bo23}): Consider the solution $X(\dott)$ of the RHP, \\[1mm]
$(1)$ $X(z) \in \bbC^{2 \times 2}$ is analytic in $\bbC\backslash \bbR$. \\[1mm] 
$(2)$ $X_+(\xi) = X_-(\xi) \begin{pmatrix} 1 - |r(\xi)|^2 & - \ol{r(\xi)} e^{-it\xi} \\[1mm] 
r(\xi) e^{it\xi} & 1 \end{pmatrix}$, $\xi \in \bbR$,  \\[1mm] 
\hspace*{4.8mm}  where $X_{\pm}(\xi) = \lim_{\varepsilon\downarrow 0} X(\xi \pm i \varepsilon)$, \, $r(\xi) = -i \int_{\bbR} \phi(y) e^{-i \xi y} \, dy$,
\, $\xi \in \bbR$. \\[1mm] 
$(3)$ $X(z) = I + X_1 z^{-1} + \oh\big(z^{-1}\big)$ as $z\to \infty$ in $\bbC \backslash \bbR$, where $X_1 = (X_{1,j,k}(t,\phi))_{j,k=1}^2$. \\[2mm]
\indent 
Then, provided $I - Q$ is boundedly invertible in $L^2((t,\infty))$, the above RHP is uniquely solvable and one obtains
\begin{equation}
\f{d}{dt} \ln(D_Q(t)) = i X_1''(t,\phi).
\end{equation}

Applying the steepest decent method to the RHP $(1)$\,$(2)$\,$(3)$ yields the asymptotic behavior of $D_M(t)$ as $t \to \infty$ to order $\Oh(t^{-\infty})$ (see \cite[Theorem~3.4]{BJ25}). 

Fredholm Pfaffians occur in many different areas of mathematics and mathematical physics, including interacting particle systems, increasing subsequences, random walker analysis, and random matrix theory for orthogonal and symplectic ensembles (see \cite{BJ25} and the references cited therein). 

We conclude this subsection with the following observation due to E.\ Rains and reported to P.D. by C.\ Sinclair at Snowbird in 2007:
\begin{remark} \lb{r3.8}
For $n \in \bbN$, $n$ odd, let $A, B, C \in \bbC^{n \times n}$ with $B$ and $C$ skew-symmetric and invertible. Then
\begin{equation}
\f{\Pf\Big(\big(C^{-1}\big)^{\top} - ABA^{\top}\Big)}{\Pf\Big(\big(C^{-1}\big)^{\top}\Big)} = 
\f{\Pf\Big(\big(B^{-1}\big)^{\top} - A^{\top}BA\Big)}{\Pf\Big(\big(B^{-1}\big)^{\top}\Big)}.          \lb{3.96}
\end{equation}
Relation \eqref{3.96} follows from $\det(I - EF) = \det(I - FE)$ for suitable, $E,F$, and hence can be viewed as the ``Pfaffian version'' of $\det(I - EF) = \det(I - FE)$. 
\hfill $\diamond$
\end{remark} 

\subsection{The Initial Boundary Value Problem for the Focusing Gross--Pitaevs\-kii Equation}\label{sec3.9}
In 2009, Holmer and Zworski \cite{HZ09} considered solutions to the focusing Gross--Pitaevskii equation with a delta potential supported at the origin,
\begin{equation}
\begin{cases} i u_t + 2^{-1} u_{xx} + |u|^2 u + q \delta_0(\dott) u = 0,  \quad q \in \bbR,    \\
u(x,0) = v_{\lambda} (x) + w(x); \quad (x,t) \in \bbR^2,
\end{cases}     \lb{3.97}
\end{equation}
where $|q|$ is small and $w(\dott)$ is even and of order $\Oh(q)$. In addition, $v_{\lambda}(\dott)$ has the special form,
\begin{equation}
v_{\lambda}(x) = \lambda \sech\big(\lambda |x| + \tanh^{-1}(q/\lambda)\big), \quad \lambda > |q|,
\end{equation}
and corresponds to the nonlinear ground state of a condensate variational problem in one dimension. Associated with $v_{\lambda}$ one has the stationary solution
\begin{equation}
u_{\lambda}(x,t) = e^{i \lambda^2 t/2} v_{\lambda}(x), \quad (x,t) \in \bbR^2,
\end{equation}
for \eqref{3.97} corresponding to $w \equiv 0$. The main result in \cite{HZ09} concerns the stability of this ground state condensate under even perturbations $w= \Oh(q)$, $q \ll 1$.

In \cite{HZ09}, using dynamical systems methods, Holmer and Zworski proved that for even $w = \Oh(q)$, the solution $u(\dott,t)$ of \eqref{3.97} is stable for times $1 \leq t \leq c |q|^{-2/7}$, 
$c$ a constant, and evolves at the origin $x=0$, in particular, as
\begin{equation}
u(0,t) = e^{i \hatt \lambda (t/2)} \bigg[\hatt \lambda - [2/(\pi t)]^{1/2} e^{i [(\hatt \lambda)^2 (t/2) + (\pi/4)]} 
\int_0^{\infty} w(x)\,dx \bigg] + \Oh\big(q\big/t^{3/2}\big)     \lb{3.100}
\end{equation}
for some explicit $\hatt \lambda$, $\hatt \lambda \sim v_{\lambda}(0)$, for $q$ sufficiently small.

For $u(x,t=0) = u_0(x)$, $x \in \bbR$, even, \eqref{3.97} reduces to an equation in the first quadrant $(x,t) \in (0,\infty)^2$. This is because the delta function supported at $x=0$ introduces a jump in the derivative of $u$ at $x=0$, 
\begin{equation}
2^{-1} [u_x(0_+,\dott) - u_x(0_-,\dott)] + q u(0,\dott) = 0,
\end{equation}
which, when $u(x,\dott)$ is even in $x \in \bbR$, reduces to 
\begin{equation}
u_x(0_+,\dott) + q u(o,\dott) = 0. 
\end{equation}
Thus, in the case where $u(x,0)$, and hence $u(x,t)$ is even in $x$, \eqref{3.97} reduces to the initial boundary value problem (IBV) for the focusing NLS equation on the half-line,
\begin{equation}
\begin{cases}
i u_t + 2^{-1} u_{xx} + |u|^2 u = 0, &(x,t) \in (0,\infty)^2, \\
u(x,0) = u_0(x), &x \in (0,\infty),   \\
u_x(0,t) + q u(0,t) = 0, &t \in (0,\infty) \quad \text{(a Robin boundary condition at $x=0$).}
\end{cases}     \lb{3.103} 
\end{equation}

As noted by A.\ Fokas \cite{Fo89}, \cite{Fo02}, the IBV \eqref{3.103} is ``integrable'' in the sense that it can be solved by using linear equations only (see \cite{IS13}).

Now, on the whole line, $x \in \bbR$, solutions of the NLS equation can be evaluated asymptotically as $t \to \infty$ using Riemann--Hilbert/Steepest-Descent methods, but these methods break down on the half-line as in \eqref{3.103}. In the linear case, $i u_t + 2^{-1} u_{xx} = 0$, problems on the half-line with Dirichlet (i.e., $u(0,t)=0$) or Neumann (i.e., $u_x(0,t)=0$) boundary conditions extend easily to problems on the full line upon setting $u(x,0) = - u(-x,0)$ or $u(x,0)=u(-x,0)$, respectively, for $x < 0$. The same applies to the Robin boundary condition in the case $q<0$ provided one extends $u_x(x,0) + q u(x,0)$ as an odd function to $x<0$. These extensions make it possible to solve the IVB's simply by using the Fourier transform. In the nonlinear case \eqref{3.103} such method of images extensions clearly work for the Dirichlet and Neumann cases, but not for the Robin boundary condition. In \cite{DP11}, the authors show how to use some remarkable calculations of Bikbaev and Tarasov \cite{BT91}, \cite{Ta88}, \cite{Ta91}, which are based in part on Khabibullin \cite{Kh91}, to construct a nonlinear method of images for the IVB \eqref{3.103}. More precisely, if $u = u(x,t)$ solves \eqref{3.103} for $(x,t) \in (0,\infty)^2$, then Bikbaev and Tarasov construct a new solution $\wti u = \wti u(x,t)$ of \eqref{3.103} for $(x,t) \in (0,\infty)^2$, with the property that 
\begin{equation}
\hatt u(x,t) = \begin{cases}
u(x,t), & x \in [0,\infty), \\
\wti u(-x,t), & x \in (-\infty,0),
\end{cases} \quad t \in (0,\infty),     \lb{3.104} 
\end{equation} 
solves NLS on the whole line 
\begin{equation}
\begin{cases}
i \hatt u_t + 2^{-1} \hatt u_{xx} + \big|\hatt u\big|^2 \hatt u =0, & (x,t) \in \bbR \times (0,\infty), \\
\hatt u(x,0) = u_0(|x|), & x \in \bbR,
\end{cases}     \lb{3.105} 
\end{equation}
and, {\it automatically,}
\begin{equation}
\hatt u_x(0,t) + q \hatt u (0,t) = 0, \quad t \in (0,\infty).   \lb{3.106} 
\end{equation}
It follows that the IVP \eqref{3.103} can then be evaluated asymptotically as $t \to \infty$, employing the well-developed Riemann--Hilbert/Steepest Descent methods. As shown in \cite{DP11}, the solution $u=u(x,t)$ for \eqref{3.97}, in fact, remains stable as in \eqref{3.100} for times up to $q^{-2}$, but for times $t \gg q^{-2}$ the solution breaks into two soliton solutions moving independently. 

To construct $\wti u$, Bikbaev and Tarasov used B\"acklund transformations. We recall that a B\"acklund transformation takes a solution $\omega(t)$ of some dynamical system to the solution $\wti \omega(t)$ of some, in general different, dynamical system. One calls $\wti \omega(t)$ the B\"acklund transform of $\omega(t)$. For example, 
\begin{equation}
\omega(x,t) = e^{\beta x - 4 \beta^3 t} + \alpha e^{- \beta x + 4 \beta^3 t}, \quad \alpha \in (0,\infty), \; (x,t) \in \bbR^2,
\end{equation}
solves the linear equation
\begin{equation}
\omega_t + 4 \omega_{xxx} = 0,   \lb{3.108} 
\end{equation}
but 
\begin{equation}
\wti \omega (x,t) = - 2 \f{d^2}{dx^2} \, \ln(\omega(x,t)), \quad (x,t) \in \bbR^2,
\end{equation}
solves the $\KdV$ equation
\begin{equation}
\wti \omega_t - 6 \wti \omega \wti \omega_x + \wti \omega_{xxx} = 0, \quad (x,t) \in \bbR^2.    \lb{3.110} 
\end{equation}
Thus, $\omega \mapsto \wti \omega$ is a B\"acklund transformation from \eqref{3.108} to \eqref{3.110}. 

The first step in the construction of $\wti u(x,t)$, $(x,t) \in (0,\infty)^2$, in \eqref{3.104} consists in noting that if $u=u(x,t)$ with $u(\dott,t) \in L^1(\bbR)$, $t \in (0,\infty)$, and $Q(t) = \left(\begin{smallmatrix} 0 & u(x,t) \\ - \ol{u(x,t)} & 0 \end{smallmatrix}\right)$, $t \in (0,\infty)$, then the equation
\begin{equation}
\begin{cases}
P_x(t) = (Q(t) - i [\sigma,P(t)])P(t), \quad (x,t) \in (0,\infty)^2,  \\
P(0) = - i q \sigma_3,
\end{cases}     \lb{3.111} 
\end{equation}
has a unique $2 \times 2$ matrix-valued solution $P=P(x,t)$, $(x,t) \in [0,\infty)^2$. Here $\sigma_3 = \begin{pmatrix} 1 & 0 \\ 0 & -1 \end{pmatrix}$, $\sigma = 2^{-1} \sigma_3$, and $q\in \bbR$ is the constant in \eqref{3.97}. Next, set 
\begin{equation}
\wti u(x,t) = u(x,t) -i  P_{1,2}(x,t), \quad (x,t) \in (0,\infty)^2,    \lb{3.112}
\end{equation}
and then
\begin{equation}
\wti Q(x,t) \equiv Q(x,t) - i [\sigma,P(x,t)] = \begin{pmatrix} 0 & \wti u(x,t) \\[1mm] - \ol{\wti u(x,t)} & 0 \end{pmatrix}, 
\quad (x,t) \in (0,\infty)^2.
\end{equation}

Commutation now enters the story in the following manner: one recalls that NLS has a zero curvature representation, that is, if 
\begin{align}
& L = L(t) = iz \sigma - Q(\dott,t), \quad Q(x,t) = \begin{pmatrix} 0 & u(x,t) \\[1mm] - \ol{u(x,t)} & 0 \end{pmatrix},   \lb{3.114} \\
& E = E(t) = - (i/2) z^2 \sigma - 2^{-1} z Q(\dott,t) + \f{1}{2} \begin{pmatrix} i |u(x,t)|^2 & i u_x(x,t) \\[1mm] i \ol{u_x(x,t)} & -i |u(x,t)|^2 \end{pmatrix};   \lb{3.115} \\
& \hspace*{8.85cm}  (x,t) \in \bbR^2.    \no 
\end{align}
Then the differential expressions $\partial_x I_2 - L$ and $\partial_t I_2 - E$ commute, that is,
\begin{equation}
(\partial_x I_2 - L)(\partial_t I_2 - E) = (\partial_t I_2 - E) (\partial_x I_2 - L),   \lb{3.116}
\end{equation} 
if and only if $u = u(x,t)$ solves NLS on $\bbR^2$. In particular, it follows from \eqref{3.116} that $\partial_x I_2 - L(t)$ undergoes an isospectral deformation under NLS. 

In the language of commutation, NLS represents an ``auto-commutation,'' $AB \to BA=AB$, that is, a commutation where $AB$ equals $BA$ (with $A$ and $B$ unbounded). Said differently, NLS is a fixed point for commutation. 

A direct computation then shows that if $P$ solves \eqref{3.111}, and 
\begin{equation}
\wti L(t) = i z \sigma - \wti Q(\dott,t), \quad \wti Q (x,t) = \begin{pmatrix} 0 & \wti u(x,t) \\[1mm] - \ol{\wti u(x,t)} & 0 \end{pmatrix}, 
\quad (x,t) \in (0,\infty) \times \bbR,
\end{equation}
then
\begin{equation}
(z I_2 + P)(\partial_x I_2 - L) = \big(\partial_x I_2 - \wti L\big)(z I_2 + P). 
\end{equation}
In particular, it follows that if $u=u(x,t)$ solves NLS, and hence 
\begin{equation}
L(t) = i z \sigma - \begin{pmatrix} 0 & u(x,t) \\[1mm] - \ol{u(x,t)} & 0 \end{pmatrix}
\end{equation} 
undergoes an isospectral deformation, then so does 
\begin{equation}
\wti L(t) = i z \sigma - \begin{pmatrix} 0 & \wti u(x,t) \\[1mm] - \ol{\wti u(x,t)} & 0 \end{pmatrix},
\end{equation} 
where $\wti u(x,t)$ is given by \eqref{3.112}, as $\wti L(t)$ is given by conjugation of $L(t)$,
\begin{equation}
\partial_x I_2 - \wti L(t) = (z I_2 +P(t)) (\partial_x I_2 - L(t)) (z I_2 +P(t))^{-1},
\end{equation} 
apart from the irrelevant singularities of $zI_2 +P$. 

The question now becomes, which isospectral deformations is $\wti L(t)$ undergoing?

This is settled by noting that 
\begin{equation} 
\big(\partial_t I_2 - \wti E\big)(z I_2 +P) = (z I_2 + P)(\partial_t I_2 - E)     \lb{3.122} 
\end{equation}
if and only if 
\begin{equation} 
\big(\partial_t I_2 - \wti E\big) \big(\partial_x I_2 - \wti L\big) = \big(\partial_x I_2 - \wti L\big) \big(\partial_t I_2 - \wti E\big), 
\lb{3.123}
\end{equation}
 where $\wti E$ is given by \eqref{3.115} with $u$ replaced by $\wti u$. Again, \eqref{3.122} is proved by a simple, direct computation. In particular, it follows from \eqref{3.116} and \eqref{3.123} that $\wti u(x,t)$ solves NLS for $(x,t) \in (0,\infty)^2$. 
 
 From these calculations one infers that for $q$ fixed, 
 \begin{align}
 \begin{split} 
& u(x,t) \longrightarrow Q(x,t) = \begin{pmatrix} 0 & u(x,t) \\ - \ol{u(x,t)} & 0 \end{pmatrix} \longrightarrow P(x,t)   \\
& \quad  \longrightarrow \wti Q(x,t) = \begin{pmatrix} 0 & \wti u(x,t) \\ - \ol{\wti u(x,t)} & 0 \end{pmatrix} \longrightarrow \wti u (x,t),  
\end{split} \\ 
& \hspace*{5.1cm} (x,t) \in (0,\infty)^2,     \no 
 \end{align} 
is a B\"acklund transformation taking a solution $u=u(x,t)$ of NLS on $(0,\infty)^2$ to a new solution $\wti u = \wti u(x,t)$ of NLS on $(0,\infty)^2$. Upon insertion of $\wti u(x,t)$ into \eqref{3.104}, it becomes a matter of direct computations to verify that \eqref{3.106} is satisfied automatically. 

Finally, we note that Its and Shepelsky \cite{IS13} showed how to use Fokas' method in \cite{Fo89}, \cite{Fo02} to provide another method to analyze the long-time behavior of the solution $u=u(x,t)$ of \eqref{3.97}. Fokas' method leads to a Riemann--Hilbert problem on a cross in the energy plane. In \cite{IS13} the authors show how to ``unfold'' the cross to the real axis where the standard Riemann--Hilbert/Steepest Descent method for NLS can now be applied.

\subsection{A Generalization of Commutation}\label{sec4a}
In 2009, Jim Portegies, then a student at the Courant Institute, noted the following generalizations of the commutation formulas \eqref{1.2} and \eqref{eq3}.

Let $A, B, C \in \cB(X)$ for some Banach space $X$ and suppose in addition that 
\begin{align}
& A \text{ is boundedly invertible,} \\ 
& [A,C] =0 
\end{align}
(with $[\dott,\dott]$ abbreviating the commutator). Then, if $A+CB$ is boundedly invertible, so is $A+BC$ and 
\begin{equation}
A(A+BC)^{-1} + B(A+CB)^{-1} C = I.      \lb{2.19}
\end{equation}
Conversely, if $A+BC$ is boundedly invertible, so is $A+CB$ and 
\begin{equation}
(A+CB)^{-1}A + C(A+BC)^{-1} B = I. 
\end{equation}
As a result, if $A$ is boundedly invertible, then
\begin{equation}
A+CB \, \text{ is boundedly invertible if and only if } \, A+BC \, \text{ is boundedly invertible.} 
\end{equation} 
Replacing $A$ by $A-zI$, $z \in \rho(A)$, one infers that 
\begin{equation}
\sigma(A+BC) \setminus \sigma(A) = \sigma(A+CB) \setminus \sigma(A).    \lb{2.22}
\end{equation}

If $A$ is boundedly invertible and $[A,B]=0$, but $[A,C] \neq 0$, then \eqref{2.19} and \eqref{2.22} may fail and must be replaced by 
\begin{equation}
(A+BC)^{-1}A + B(A+CB)^{-1} C = I.       \lb{2.23} 
\end{equation}
and 
\begin{equation}
A(A+CB)^{-1} + C(A+BC)^{-1} B = I,       \lb{2.24} 
\end{equation}
in the sense that if $A+CB$ is boundedly invertible, then so is $A+BC$ and \eqref{2.23} holds, and if $A+BC$ is boundedly invertible, then so is $A+CB$ and \eqref{2.24} holds. Thus, \eqref{2.22} also remains valid in this case. 

The proof of this generalization follows those of \eqref{1.2} and \eqref{eq3}.

\subsection{A $2\times2$ Block Operator Approach to Commutation Formulas: The  Banach Space Case}\label{sec12}
In this subsection we return to the beginning of our subject and recall in detail the principal commutation formulas associated with a pair of closed, densely defined operators $A$ and $B$ in complex Banach spaces based on a $2 \times 2$ block operator approach. In particular, 
we derive a few additional facts regarding the algebraic multiplicity of nonzero, discrete eigenvalues of $AB$ and $BA$ and the underlying $2 \times 2$ block operator\footnote{In the physics literature, $Q$ is sometimes called a supercharge (a.k.a. a supersymmetric Dirac-type operator).} $Q = \left(\begin{smallmatrix} 0_{X_1} & B \\ A & 0_{X_2} \end{smallmatrix}\right)$.

We start by making the following assumptions:

\begin{hypothesis} \lb{h2.1}
Suppose $X_j$, $j=1,2$, are complex Banach spaces.  \\
$(i)$ Consider the closed and densely defined operators, 
\begin{align}
\begin{split} 
& A \colon \dom(A) \to X_2, \quad \dom(A) \subseteq X_1,   \lb{2.1} \\
& B \colon \dom(B) \to X_1, \quad \dom(B) \subseteq X_2.
\end{split}
\end{align}
$(ii)$ Assume that $AB$ and $BA$ $($both naturally defined\,$)$ are closed and densely defined in $X_2$ and $X_1$, respectively. \\
$(iii)$ $\rho(AB) \neq \emptyset$ and $\rho(BA) \neq \emptyset$. 
\end{hypothesis}

Given Hypothesis \ref{h2.1} we also introduce $2 \times 2$ block operator $Q$ in $X = X_1 \dotplus X_2$ given by 
\begin{equation}
Q = \begin{pmatrix} 0_{X_1} & B \\ A & 0_{X_2} \end{pmatrix}, \quad \dom(Q) = \dom(A) \dotplus \dom(B).
\end{equation}

Then a summary of the results in \cite{HKM00} and \cite{HM01} reads as follows:

\begin{theorem} \lb{t2.2} $($\cite{HKM00}, \cite{HM01}$)$. 
Assume Hypothesis \ref{h2.1}. Then the following items $(i)$--$vii)$ hold: \\
$(i)$ One has 
\begin{equation}
\rho(AB) \cap \rho(BA) \neq \emptyset.
\end{equation} 
$(ii)$ $Q$ is closed and densely defined in $X$, and 
\begin{equation}
\rho(Q) = \big\{z \in \bbC \, \big| \, z^2 \in \rho(AB) \cap \rho(BA)\big\}.
\end{equation}
$(iii)$ There holds
\begin{align}
& \{\lambda \in \bbC \, | \, \lambda^2 \in \sigma(AB) \backslash \{0\}\} = \sigma(Q) \backslash \{0\} 
= \{\lambda \in \bbC \, | \, \lambda^2 \in \sigma(BA) \backslash \{0\}\},   \\
& \{z \in \bbC \, | \, z^2 \in \rho(AB)\} \backslash \{0\}\}= \rho(Q) \backslash  \{0\} 
= \{z \in \bbC \, \big| \, z^2 \in \rho(BA) \backslash \{0\}\}, 
\end{align}
in particular,
\begin{equation}
\sigma(AB) \backslash \{0\} = \sigma(BA) \backslash \{0\}.
\end{equation}
Moreover,
\begin{equation}
\sigma_p(AB) \backslash \{0\} = \sigma_p(BA) \backslash \{0\}.
\end{equation}
$(iv)$ In addition one has 
\begin{align}
& \{\lambda \in \bbC \, | \, \lambda^2 \in \sigma_{ess,5}(AB) \backslash \{0\}\} = \sigma_{ess,5}(Q) \backslash \{0\} 
= \{\lambda \in \bbC \, | \, \lambda^2 \in \sigma_{ess,5}(BA) \backslash \{0\}\},   \\
& \{\lambda \in \bbC \, | \, \lambda^2 \in \sigma_d(AB) \backslash \{0\}\} = \sigma_d(Q) \backslash \{0\} 
= \{\lambda \in \bbC \, | \, \lambda^2 \in \sigma_d(BA) \backslash \{0\}\}.
\end{align}
$(v)$ There holds
\begin{equation}
\sigma_p(Q) \cup \sigma_p(-Q) = \{\lambda \in \bbC \, | \, \lambda^2 \in \sigma_p(AB) \cup \sigma_P(BA)\}.  
\end{equation}
$(vi)$ The operator 
\begin{equation}
Q^2 = \begin{pmatrix} BA & 0_{X_2} \\ 0_{X_1} & AB \end{pmatrix}, \quad \dom\big(Q^2\big) = \dom(BA) \dotplus \dom(AB), 
\end{equation}
is closed and densely defined in $X$. \\
$(vii)$ The resolvent of $Q$ in $X$ is of the form
\begin{align}
(Q - z I_X)^{-1} = \begin{pmatrix} z (BA - z^2 I_{X_1})^{-1} & B(AB - z^2 I_{X_2})^{-1} \\
A(BA - z^2 I_{X_1})^{-1} & z (AB - z^2 I_{X_2} \end{pmatrix}, \quad z \in \rho(Q) \backslash\{0\},
\end{align}
and the commutation formulas 
\begin{align} 
& - \f{1}{z} I_{X_1} + \f{1}{z} \ol{B(AB - z^2 I_{X_2})^{-1}A\big|_{\dom(A)}} = z (BA - z^2 I_{X_1})^{-1}, \\
&  - \f{1}{z} I_{X_2} + \f{1}{z} \ol{A(BA - z^2 I_{X_1})^{-1}B\big|_{\dom(B)}} = z (AB - z^2 I_{X_2})^{-1}, \\ 
& \,\, \ol{(AB - z^2 I_{X_2})^{-1}A\big|_{\dom(A)}} = A (BA - z^2 I_{X_1})^{-1}, \\ 
& \,\, \ol{(BA - z^2 I_{X_1})^{-1}B\big|_{\dom(B)}} = B (AB - z^2 I_{X_2})^{-1};    \lb{2.17} \\ 
& \hspace*{2.8cm} z \in \rho(AB) \backslash \{0\} = \rho(BA) \backslash \{0\},   \no 
\end{align}
hold. 
\end{theorem}

Here the notions of discrete and essential spectra are defined as follows: Let $Y$ be a complex Banach space and $T \colon \dom(T) \to Y$, $\dom(T) \subseteq Y$, a closed and densely defined operator. Then the discrete spectrum of $T$ is given by
\begin{align}
\sigma_d(T) &= \{\lambda \in \sigma_p(T) \,|\, \text{$\lambda$ is an isolated point of $\sigma(T)$}   \no  \\
& \hspace*{2.3cm} \text{ of finite algebraic multiplicity}\}    \no \\
&= \{\lambda \in \sigma_p(T) \,|\, \text{$\lambda$ is an isolated point of $\sigma(T)$}   \no \\
& \hspace*{2.3cm} \text{ with $\dim(\ran(P(\lambda; T))) < \infty$}\}     \lb{2.18} \\
&= \{\lambda \in \sigma_p(T) \,|\, \text{$\lambda$ is an isolated point of $\sigma(T)$ such that}    \no \\
& \hspace*{1.45cm} \text{$(T-\la I_{\cH})$ is Fredholm and $\ind(T-\la I_{\cH}) = 0$}\}.   \no
\end{align}
Here $P(\lambda_0; T)$, $\lambda_0 \in \sigma_p(T)$ and $\lambda_0$ an isolated point of $\sigma(T)$, represents the Riesz projection associated with $T$ and $\lambda_0 \in \sigma_p(T)$ in $Y$, that is, 
\begin{equation}
P(\lambda_0; T) = \f{-1}{2\pi i} \ointctrclockwise_{C(\lambda_0;r_0)} (T- \zeta I_Y)^{-1} \, d\zeta,
\end{equation} 
where $C(\lambda_0;r_0) = \partial D(\lambda_0;r_0)$ denotes the counterclockwise oriented circle with center $\lambda_0$ and radius $r_0 > 0$ chosen sufficiently small so that $\sigma(T) \cap \ol{D(\lambda_0;r_0)} = \{\lambda_0\}$, with $D(\lambda_0;r_0)$ the open disc with with center $\lambda_0$ and radius $r_0$.  
Any element of $\sigma_d(T)$ in \eqref{2.18} is called a {\it discrete eigenvalue of $T$} (sometimes also a {\it normal eigenvalue of $T$}). Among the possible definitions of the discrete spectrum of $T$ this definition singles out the smallest such set.

An alternative description of $\lambda_0 \in \sigma_d(T)$ utilizes the property of $(T - z I_Y)^{-1}$ being finitely meromorphic at $z=\lambda_0$ (i.e., $(T - z I_Y)^{-1}$ has a pole at $z=\lambda_0$ with residue given by $-P(\lambda_0; T)$ and the (finitely-many) coefficients of its Laurent expansion are finite-rank operators. 

The essential spectrum $T$ then is defined via 
\begin{equation}
\sigma_{ess,5}(T) = \sigma(T) \backslash \sigma_d(T).    \lb{3.158}
\end{equation} 
In this manner $\sigma_{ess,5}(T)$ is the largest set associated with the standard possible definitions of essential spectra of $T$. For the definitions $\sigma_{ess,j}(T)$, $1 \leq j \leq 5$, and a comparison between them, see, for instance 
\cite[Sect.~III.7]{BC19}, \cite[Sect.~9.1]{EE18}, \cite{ELZ83}.

For general background regarding linear operators and their spectra, see, for instance, \cite[Chs.~1--3]{BHS20}, 
\cite[Chs.~4, 5, 8, 9, 11]{Da07}, \cite[Chs.~2, 3]{DM25}, \cite[Chs.~I--VIII]{GGK90}, \cite[Chs.~I--III]{GK69}, \cite[Ch.~V]{Ka80}, 
\cite[Ch.~X]{RS75}, \cite[Ch.~XI]{RS79}, \cite[Ch.~XIII]{RS78}, \cite[Parts~I--IV]{Sc12}, \cite[Chs.~2, 3, 5, 7]{Si15}, \cite[Chs.~4--9, 11]{We80}

\begin{remark} \lb{r2.3}
$(i)$ As discussed in \cite{HKM00}, and\cite{HM01}, several of the facts in Theorem~\ref{t2.2} hold under more general hypotheses on $A$ and $B$. We chose the current somewhat stronger assumptions on $A$ and $B$ to avoid a number of technicalities.  \\
$(ii)$ We also recall the fact (see, e.g., \cite[Theorem~III.5.29]{Ka80}): Suppose that  $X_1,X_2$ are reflexive and $T\colon \dom(T) \to X_2$, $\dom(T) \subseteq X_1$, is closable. Then $T^*$ is closed and densely defined and $\ol{T} = (T^*)^*$. \hfill $\diamond$
\end{remark} 

Next, we recall the following fact from \cite{De78}:

\begin{theorem} \lb{t2.4} $($\cite{De78}$)$. 
In addition to Hypothesis \ref{h2.1} suppose that $A \in \cB(X_1,X_2)$ and $B \in \cB(X_2,X_1)$. Then the following items $(i)$ and $ii)$ hold: \\
$(i)$ Let $z^2 \in \bbC\backslash \{0\}$, then the maps 
\begin{align}
& A \colon \ker(BA - z^2 T_{X_1}) \to \ker(AB - z^2 I_{X_2}) \, \text{ is surjective,}   \\
& B \colon \ker(AB - z^2 T_{X_1}) \to \ker(BA - z^2 I_{X_2}) \, \text{ is surjective.}
\end{align}
In particular, the geometric multiplicities of nonzero eigenvalues of $AB$ and $BA$ coincide. \\
$(ii)$ Suppose $\lambda_0 \in \sigma_d(AB) \backslash \{0\} = \sigma_d(BA) \backslash \{0\}$. Then the geometric and algebraic multiplicities of $\lambda_0$ as a nonzero eigenvalue of $AB$ and $BA$ coincide. 
\end{theorem}
\begin{proof}
Since algebraic multiplicities were not addressed in \cite{De78}, we now provide a proof that algebraic multiplicities of nonzero discrete eigenvalues of $AB$ and $BA$ coincide. (The case of geometric multiplicities is already treated in item $(i)$.) Let $P(\lambda_0,AB)$ denote the Riesz projection associated with $\lambda_0 \in \sigma_d(AB)\backslash \{0\}$ and suppose that $f \in \ran(P(\lambda_0,AB))$, that is, $P(\lambda_0,AB) f = f$. Then 
\begin{equation}
P(\lambda_0,AB) f = \f{-1}{2\pi i} \ointctrclockwise_{C(\lambda_0;r_0)} (AB- \zeta I_{X_1})^{-1} f \, d\zeta = f,
\end{equation}
and hence 
\begin{align} 
\begin{split} 
B f &= \f{-1}{2\pi i} \ointctrclockwise_{C(\lambda_0;r_0)} B (AB- \zeta I_{X_1})^{-1} f \, d\zeta   \\
&= \f{-1}{2\pi i} \ointctrclockwise_{C(\lambda_0;r_0)} (BA- \zeta I_{X_1})^{-1} B f \, d\zeta,   
\end{split} 
\end{align}
since by \eqref{2.17}, $B (AB- \zeta I_{X_1})^{-1} = (BA- \zeta I_{X_2})^{-1} B$. Thus,
\begin{equation}
P(\lambda_0,BA) g = g, \quad g = B f. 
\end{equation}
Consequently, 
\begin{equation}
B \ran(P(\lambda_0,AB)) \subseteq \ran(P(\lambda_0,BA)).
\end{equation}
It remains to show injectivity of the map $B \colon \ran(P(\lambda_0,AB)) \to \ran(P(\lambda_0,BA))$. To this end suppose that $Bf = 0$, we will show this implies $f=0$.

For $\zeta \in C(\lambda_0;r_0)$ one obtains that
\begin{equation}
f = (AB - \zeta I_{X_1})^{-1} (AB -  \zeta I_{X_1}) f = - \zeta (AB - \zeta I_{X_1})^{-1} f,
\end{equation}
equivalently,
\begin{equation}
(AB - \zeta I_{X_1})^{-1} f = - \zeta^{-1} f.
\end{equation}
But then
\begin{equation}
f = \f{-1}{2\pi i} \ointctrclockwise_{C(\lambda_0;r_0)} (AB- \zeta I_{X_1})^{-1} f \, d\zeta 
= \f{1}{2\pi i} \ointctrclockwise_{C(\lambda_0;r_0)} \zeta^{-1} \, d\zeta \, f = 0,
\end{equation}
since $\zeta^{-1}$ is analytic in a sufficiently small open neighborhood of $\ol{D(\lambda_0;r_0)}$. Thus,
\begin{equation}
\dim(\ran(P(\lambda_0,AB))) \leq \dim(\ran(P(\lambda_0,BA))), 
\end{equation}
and interchanging the role of $A$ and $B$ yields 
\begin{equation}
\dim(\ran(P(\lambda_0,AB))) = \dim(\ran(P(\lambda_0,BA))), 
\end{equation}
as was to be proven.
\end{proof}

Historically, we note that commutation formulas were mentioned by S.\ Sakai \cite[p.~3]{Sa71} in the context of $C^*$-algebras.

\subsection{A $2\times2$ Block Operator Approach to Commutation Formulas: The Hilbert Space Case}\label{sec13}
In this subsection we now focus on the Hilbert space situation which permits us to go some steps beyond the Banach space setting.  

We start by making the following assumption:

\begin{hypothesis} \lb{h3.1}
Suppose $\cH_j$, $j=1,2$, are complex, separable Hilbert spaces, and assume that 
\begin{equation}
A \colon \dom(A) \to X_2, \quad \dom(A) \subseteq X_1,   \lb{3.1} 
\end{equation} 
is closed and densely defined.
\end{hypothesis}

Choosing $B =A^*$, one confirms that Hypothesis~\ref{h2.1} is clearly satisfied. This either follows from Remark~\ref{r2.3}, or, alternatively, from appealing to a celebrated result of J.\ von Neumann that asserts that $AA^*$ and $A^*A$ are nonnegative, self-adjoint operators in $\cH_2$ and $\cH_1$, respectively. 

In addition,
\begin{equation}
Q = \begin{pmatrix} 0_{\cH_1} & A^* \\ A & 0_{\cH_2} \end{pmatrix}, \quad \dom(Q) = \dom(A) \oplus \dom(A^*),
\end{equation}
is self-adjoint in $\cH = \cH_1 \oplus \cH_2$. In turn, self-adjointness of $Q$ in $\cH$ implies that of $Q^2$ (e.g., by the spectral theorem) and since 
\begin{align}
\begin{split} 
& \, Q^2 = \begin{pmatrix} A^*A & 0_{\cH_2} \\ 0_{\cH_1} & AA^* \end{pmatrix} = A^*A \oplus AA^* \geq 0_{\cH}, \\
&\dom\big(Q^2\big) = \dom(A^*A) \oplus \dom(AA^*),
\end{split}
\end{align}
this yields a very efficient alternative proof of von Neumann's result as noticed by E.\ Nelson (unpublished).

Then Theorem~\ref{t2.2} applies verbatim to the pair $(A, B=A^*)$ and Theorem~\ref{t2.4} now also extends to the case of unbounded operators (see \cite{De78} for details). Theorem~\ref{t3.2} below will present more results of this genre. 

We note that the Dirac-type operator $Q = \begin{pmatrix} 0_{\cH_1} & A^* \\ A & 0_{\cH_2} \end{pmatrix}$ in connection with $Q^2 = A^*A \oplus AA^*$, is a special case of what is known as supersymmetric quantum mechanics. Without going into any details, we will use the expression ``supersymmetry'' in this paper whenever we refer to the triple $(A^*A, AA^*, Q)$. 

Next, we slightly rephrase \cite[Theorem~3]{De78} a bit, following the surveys in \cite[App.~A]{EGNT14}, 
\cite[App.~A]{GGHT12}, and \cite{Th88}. For this purpose we need the polar decompositions of $A$ and $A^*$, that is, the representations
\begin{align}
\begin{split} 
A& = V_A |A| = |A^*| V_A = V_A A^* V_A  \, \text{ on } \, \dom(A) = \dom(|A|),     \\
A^*& = V_{A^*} |A^*| = |A|V_{A^*} = V_{A^*} A V_{A^*}   \, \text{ on } \, 
\dom(A^*) = \dom(|A^*|),     \\
|A|& = V_{A^*} A = A^* V_A = V_{A^*} |A^*| V_A \, \text{ on } \, \dom(|A|),     \\
 |A^*|& = V_A A^* = A V_{A^*} = V_A |A| V_{A^*}  \, \text{ on } \, \dom(|A^*|),   
\end{split} 
\end{align}
where
\begin{align}
& |A| = (A^* A)^{1/2}, \quad |A^*| = (AA^*)^{1/2},  \quad
V_{A^*} = (V_A)^*,      \\
& V_{A^*} V_A = P_{\ol{{\ran}(|A|)}}=P_{\ol{{\ran}(A^*)}} \, ,  \quad
V_A V_{A^*} = P_{\ol{{\ran}(|A^*|)}}=P_{\ol{{\ran}(A)}} \, .     
\end{align} 
In particular, $V_A$ is a partial isometry with initial set $\ol{{\ran}(|A|)}$
and final set $\ol{{\ran}(A)}$ and hence $V_{A^*}$ is a partial isometry with initial 
set $\ol{\ran(|A^*|)}$ and final set $\ol{\ran(A^*)}$. In addition, 
\begin{equation}
V_A = \begin{cases} \ol{A (A^*A)^{-1/2}} = \ol{(AA^*)^{-1/2} A} & \text{on }  (\ker (A))^{\bot},  \\
0 & \text{on }  \ker (A).  \end{cases}     
\end{equation} 

Next, we collect some properties relating $A^*A$ and $AA^*$.

\begin{theorem} [\cite{De78}] \lb{t3.2}  
Assume Hypothesis \ref{h3.1} and let $\phi$ be a bounded Borel measurable 
function on $\bbR$. \\ 
$(i)$ One has
\begin{align}
& \ker(A) = \ker(A^*A) = (\ran(A^*))^{\bot}, \quad \ker(A^*) = \ker(AA^*) = (\ran(A))^{\bot},   \\
& V_A (A^*A)^{n/2} = (AA^*)^{n/2} V_A, \; n\in\bbN, \quad 
V_A \phi(A^*A) = \phi(AA^*) V_A.   
\end{align}
$(ii)$ $AA^*$ and $A^*A$ are essentially isospectral, that is, 
\begin{equation}
\sigma(AA^*)\backslash\{0\} = \sigma(A^*A)\backslash\{0\},    
\end{equation}
in fact, 
\begin{equation}
AA^* [I_{\cH_2} - P_{\ker(A^*)}] \, \text{ is unitarily equivalent to } \, A^*A [I_{\cH_1} - P_{\ker(A)}].   \lb{3.179} 
\end{equation} 
In addition,
\begin{align}
\begin{split} 
& g\in \dom(AA^*)\, \text{ and } \, AA^* \, g = \mu^2 g, \; \mu \neq 0,      \\
& \quad \text{implies }  \, A^* g \in \dom(A^*A)\, \text{ and } \, A^*A(A^* g) = \mu^2 (A^* g),    \\
& f\in \dom(A^*A) \, \text{ and } \, A^*A f = \lambda^2 f, \; \lambda \neq 0,    \\
& \quad \text{implies }  \,  A f \in \dom(AA^*) \, \text{ and } \, AA^* (A f) = \lambda^2 (A f),     
\end{split} 
\end{align}
with multiplicities of eigenvalues preserved. \\
$(iii)$ One has for $z \in \rho(AA^*) \cap \rho(A^*A)$,
\begin{align}
& I_{\cH_2} + z (AA^* - z I_{\cH_2})^{-1} \supseteq A (A^*A - z I_{\cH_1})^{-1} A^*,    \lb{3.14} \\
& I_{\cH_1} + z (A^*A - z I_{\cH_1})^{-1} \supseteq A^* (AA^* - z I_{\cH_2})^{-1} A,    \lb{3.15}
\end{align}
and 
\begin{align}
& A^* \phi(AA^*) \supseteq \phi(A^*A) A^*, \quad 
A \phi(H_1) \supseteq \phi(AA^*) A,    \\
& V_{A^*} \phi(AA^*) \supseteq \phi(H_1) V_{A^*}, \quad 
V_A \phi(A^*A) \supseteq \phi(AA^*) V_A.   
\end{align}
\end{theorem}

As noted by E.\ Nelson (unpublished), Theorem~\ref{t3.2} follows from the spectral theorem and the 
elementary identities, 
\begin{align}
& Q = V_Q |Q| = |Q| V_Q,     \\
& \ker(Q) = \ker(|Q|) = \ker (Q^2) = (\ran(Q))^{\bot} 
= \ker (A) \oplus \ker (A^*),     \\
& I_{\cH_1 \oplus \cH_2} + z (Q^2 - z I_{\cH_1 \oplus \cH_2})^{-1} 
= Q^2 (Q^2 -z I_{\cH_1 \oplus \cH_2})^{-1} \supseteq Q (Q^2 -z I_{\cH_1 \oplus \cH_2})^{-1} Q,  \no \\
& \hspace*{9.3cm}   z \in \rho(Q^2),      \\
& Q \phi(Q^2) \supseteq \phi(Q^2) Q,   
\end{align}
where
\begin{equation}
V_Q = \begin{pmatrix} 0 & (V_A)^* \\ V_A & 0 \end{pmatrix} 
= \begin{pmatrix} 0 & V_{A^*} \\ V_A & 0 \end{pmatrix}.   
\end{equation}

In particular,
\begin{equation}
\ker(Q) = \ker(A) \oplus \ker(A^*), \quad  
P_{\ker(Q)} = \begin{pmatrix} P_{\ker(A)} & 0 \\ 0 & P_{\ker(A^*)} \end{pmatrix},    
\end{equation}
and we also recall that
\begin{equation}
\mathfrak{S}_3 Q \mathfrak{S}_3 = - Q, \quad \mathfrak{S}_3 = \begin{pmatrix} I_{\cH_1} & 0 \\ 0 & - I_{\cH_2} 
\end{pmatrix},     
\end{equation}
that is, $Q$ and $-Q$ are unitarily equivalent. (For more details on Nelson's 
trick see also \cite{EGNT14}, \cite{GGHT12}, \cite[Sect.\ 8.4]{Te14}, \cite[Subsect.\ 5.2.3]{Th92}.) 
We also note that
\begin{equation}
\psi(|Q|) = \begin{pmatrix} \psi(|A|) & 0 \\ 0 & \psi(|A^*|) \end{pmatrix}    
\end{equation}
for Borel measurable functions $\psi$ on $\bbR$, and 
\begin{equation}
\ol{[Q |Q|^{-1}]} = \begin{pmatrix} 0 & (V_A)^*\\ V_A & 0 \end{pmatrix} = V_Q 
\, \text{ if } \, \ker(Q) = \{0\}.   
\end{equation}

Finally, we recall the following relationships between $Q$ and $A^*A$, $AA^*$.

\begin{theorem} [\cite{BGGSS87}, \cite{EGNT14}, \cite{GGHT12}, \cite{Th88}, \cite{Th92},~Ch.~5] \lb{t3.3}
Assume Hypothesis \ref{h3.1}. \\
$(i)$ Introducing the unitary operator $U$ on $(\ker(Q))^{\bot}$ by
\begin{equation}
U = 2^{-1/2} \begin{pmatrix} I_{\cH_1} & (V_A)^* \\ -V_A & I_{\cH_2} \end{pmatrix} 
\, \text{ on } \,  (\ker(Q))^{\bot},     
\end{equation}
one infers that
\begin{equation}
U Q U^{-1} = \begin{pmatrix}  |A| & 0 \\ 0 & - |A^*| \end{pmatrix} 
\, \text{ on } \,  (\ker(Q))^{\bot}.     
\end{equation}
$(ii)$ One has
\begin{align}
\begin{split}
(Q - \zeta I_{\cH_1 \oplus \cH_2})^{-1} = \begin{pmatrix} \zeta (A^*A - \zeta^2 I_{\cH_1})^{-1} 
& A^* (AA^* - \zeta^2 I_{\cH_2})^{-1}  \\  A (A^*A - \zeta^2 I_{\cH_1})^{-1}  & 
\zeta (AA^* - \zeta^2 I_{\cH_2})^{-1}  \end{pmatrix},&    \\
\zeta^2 \in \rho(A^*A) \cap \rho(AA^*).&   
\end{split}
\end{align}
$(iii)$ In addition, 
\begin{align}
\begin{split} 
& \begin{pmatrix} f_1 \\ f_2 \end{pmatrix} \in \dom(Q) \, \text{ and } \, 
Q \begin{pmatrix} f_1 \\ f_2 \end{pmatrix} = \eta \begin{pmatrix} f_1 \\ f_2 \end{pmatrix}, \; \eta \in \bbR \backslash \{0\},  
 \, \text{ implies } \\
& \,\, \quad f_1 \in \dom (A^*A), \, f_2 \in \dom (AA^*) \, \text{ and } \, A^*A f_1 = \eta^2 f_1, \, 
AA^* f_2 = \eta^2 f_2.    
\end{split} 
\end{align}
Conversely, 
\begin{align}
\begin{split} 
& g \in \dom(AA^*) \, \text{ and } AA^* \, g = \mu^2 g, \; \mu \neq 0, \\
& \quad \text{implies } \, \begin{pmatrix} \mu^{-1} A^* g \\ g \end{pmatrix} \in \dom(Q) \, \text{ and } \, 
Q \begin{pmatrix} \mu^{-1} A^* g \\ g \end{pmatrix} 
= \mu \begin{pmatrix} \mu^{-1} A^* g \\ g \end{pmatrix}.  
\end{split} 
\end{align}
Similarly,
\begin{align}
\begin{split} 
& f \in \dom(A^*A) \, \text{ and } A^*A f = \lambda^2 f, \; \lambda \neq 0, \\
& \quad \text{implies } \, \begin{pmatrix} f \\ \lambda^{-1} A f \end{pmatrix} \in \dom(Q) \, \text{ and } \, 
Q \begin{pmatrix} f \\ \lambda^{-1} A f \end{pmatrix} 
= \lambda \begin{pmatrix} f \\ \lambda^{-1} A f \end{pmatrix}.  
\end{split} 
\end{align}
\end{theorem}

A thorough discussion of the commutation formulas \eqref{3.14}, \eqref{3.15} and of related formulas together with applications in the Hilbert space context can be found in \cite{De78}. For a great variety of additional results in this area, see, for instance, \cite{BH22}, \cite{BGGSS87}\, \cite{DW14}, \cite{EGNT14}, \cite{EF85}, \cite{Ge91}, \cite{Ge92}, \cite{GGHT12}, \cite{GSS91}, \cite{GS90}, \cite{HW22}, \cite{KST12}, \cite{OS03}, \cite{PS14}, \cite{Sc78}, \cite{Sc03}, \cite[Sect.\ 8.4]{Te14}, \cite{Th88}, \cite[Ch.~5]{Th92}, \cite[p.~106]{We80}.

\subsection{Applications to One-Dimensional Scattering Theory}\label{sec14}
In this subsection we consider scattering theory for one-dimensional Schr\"odinger and Dirac-type operators in a supersymmetric context. In this case $A$ becomes an $L^2(\bbR)$-realization of the differential expression $A= (d/dx) + \phi(x)$, $x \in \bbR$, for an appropriate real-valued function $\phi(\dott)$ on $\bbR$ tending to asymptotic values $\phi_{\pm} \in \bbR$ as $x \to \pm \infty$. We follow the sources \cite{BGGSS87}, \cite{Ge86}, \cite{GSS91}, 

We start with scattering theory and hence introduce the following convenient set of assumptions.

\begin{hypothesis} \lb{h4.1}
Suppose that $\phi\colon \bbR \to \bbR$, $\phi_{\pm} \in \bbR$ satisfy
\begin{align}
\begin{split} 
& \phi, \phi' \in L^{\infty}(\bbR),     \\
& \lim_{x \to \pm \infty} \phi(x) = \phi_{\pm},     \lb{3.200} \\
& \pm \int_0^{\pm \infty} (1 +x^2) |\phi(x) - \phi_{\pm}| \, dx < \infty, \quad \int_{\bbR} (1+x^2) |\phi'(x)| \, dx < \infty.
\end{split} 
\end{align}
\end{hypothesis}

In particular, Hypothesis~\ref{h4.1} implies $\phi \in AC_{loc}(\bbR)$. For the remainder of Subsection~\ref{sec14} we will, without loss of generality, make the choice 
\begin{equation} 
\phi_-^2 \leq \phi_+^2 
\end{equation} 
to avoid further case distinctions, and note that the case $\phi_+^2 \leq \phi_-^2$ is entirely analogous. 

Consider\footnote{For simplicity of notation, we will use the same symbol for differential expressions (like $A = \f{d}{dx} + \phi(x)$, $x \in \bbR$, and operators (like $A$ in $L^2(\bbR)$) throughout this paper.} in $L^2(\bbR)$,
\begin{equation}
A = \f{d}{dx} + \phi, \quad \dom(A) = H^1(\bbR),    \lb{4.2}
\end{equation}
such that
\begin{equation}
A^* = - \f{d}{dx} + \phi, \quad \dom(A^*) = H^1(\bbR).     \lb{4.3}
\end{equation}
The associated Schr\"odinger operators $H_j \geq 0$, $j=1,2$, in $L^2(\bbR)$ are then given by
\begin{align}
& H_1 = A^*A = - \f{d^2}{dx^2} + V_1 \geq 0, \quad \dom(H_1) = H^2(\bbR),     \lb{4.4} \\
& H_2 = AA^* = - \f{d^2}{dx^2} + V_2 \geq 0, \quad \dom(H_2) = H^2(\bbR),     \lb{4.5} \\
& V_j(x) = \phi(x)^2 + (-1)^j \phi'(x), \quad x \in \bbR, \; j=1,2.    \lb{4.6}
\end{align}
The maps 
\begin{equation} 
\phi \mapsto V_j = \phi^2 + (-1)^j \phi', \quad j=1,2,     \lb{4.6a} 
\end{equation}
are known under the name of Miura transforms. 

The supersymmetric Dirac operator $Q$ in $L^2(\bbR) \oplus L^2(\bbR)$ is then of the form
\begin{equation}
Q = \begin{pmatrix} 0 & A^* \\ A & 0 \end{pmatrix} = \begin{pmatrix} 0 & - (d/dx) + \phi \\ (d/dx) + \phi & 0 \end{pmatrix}, 
\quad \dom(Q) = H^2(\bbR) \oplus H^2(\bbR),    \lb{4.7}
\end{equation}
and 
\begin{equation}
Q^2 = \begin{pmatrix} A^*A & 0 \\ 0 & AA^* \end{pmatrix} = H_1 \oplus H_2 \geq 0.
\end{equation}

Next, we introduce the scattering (Jost) solutions for $H_j$, $j=1,2$, in terms of (modified) Volterra integral equations defined by 
\begin{align}
& f_{j,\pm} (k_{\pm},x) = e^{\pm i k_{\pm} x} - \int_x^{\pm \infty} k_{\pm}^{-1} \sin(k_{\pm}(x-x') \big[V_j(x') - \phi_{\pm}^2\big] 
f_{j,\pm}(k_{\pm}, x') \, dx',    \no \\
& \hspace*{2.5cm} k_{\pm} = (z-\phi_{\pm}^2)^{1/2}, \; \Im(k_{\pm}) \geq 0, \; z \in \bbC, \; x \in \bbR, \; j=1,2,
\end{align}
and note that in the distributional (or weak) sense,
\begin{equation}
H_j f_{j,\pm} (k_{\pm},\dott) = z f_{j,\pm} (k_{\pm},\dott), \quad z \in \bbC, \; j=1,2. 
\end{equation}
(In particular, $f_{j,\pm} (k_{\pm},\dott), f_{j,\pm}' (k_{\pm},\dott) \in AC_{loc}(\bbR)$, and a.e.~on $\bbR$, 
$[- f_{j,\pm}'' (k_{\pm},\dott) + q(\dott) f_{j,\pm} (k_{\pm},\dott)] = z f_{j,\pm} (k_{\pm},\dott)$.)
We also note that
\begin{equation}
W(f_{j,\pm} (-k_{\pm},\dott) ,f_{j,\pm} (k_{\pm},\dott) ) = \pm 2i k_{\pm} \neq 0, \quad z \in \bbC \backslash \{\phi_{\pm}^2\}, \; 
j=1,2,
\end{equation}
where $W(f,g)(x) = f(x) g'(x) - f'(x) g(x)$ denotes the Wronskian of $f$ and $g$. 

Recalling that all essential spectra mentioned after \eqref{3.158} coincide for self-adjoint operators, hence, denoting them simply by $\sigma_{ess}(\dott)$, and denoting by $\sigma_{ac}(\dott)$ and $\sigma_{sc}(\dott)$ the absolutely and singularly continuous spectra for self-adjoint operators, one then has the following result regarding the spectra of $H_j$, $j=1,2$:

\begin{theorem} [\cite{GSS91}] \lb{t4.2}  
Assume Hypothesis~\ref{h4.1}. Then 
\begin{equation}
\sigma_{ess}(H_j) = \sigma_{ac}(H_j) = [\phi_-^2,\infty), \quad \sigma_{sc}(H_j) = \emptyset, \quad  j=1,2.
\end{equation}
Moreover, $H_j$, $j=1,2$, have simple spectrum in the interval $(\phi_-^2,\phi_+^2)$ $($assuming $\phi_-^2 < \phi_+^2$$)$ and spectral multiplicity equal to two on the interval $(\phi_+^2,\infty)$. In addition, $H_j$, $j=1,2$, have finitely many simple eigenvalues $\lambda_{j,\ell} = \phi_{\pm}^2 - \kappa_{j,\pm,\ell}^2$ in the interval $[0,\phi_-^2)$ determined by
\begin{equation}
W(f_{j,-}(i\kappa_{j,-,\ell},\dott), f_{j,+}(i\kappa_{j,+,\ell},\dott)) = 0, \quad 1 \leq \ell \leq N_j, \;  j =1, 2.
\end{equation}
If $\phi_-^2 > 0$, the $($necessarily simple\,$)$ eigenvalues of $H_1$ and $H_2$ coincide in the interval $(0,\phi_-^2)$.

There are no eigenvalues embedded into the essential spectrum of $H_j$, $j=1,2$, and there are no threshold eigenvalues, that is,
\begin{equation}
\sigma_p(H_j) \cap [\phi_-^2,\infty) = \emptyset, \quad j=1,2. 
\end{equation}
\end{theorem}

At this point we turn to the scattering matrix $S_j(\dott)$ associated with $H_j$, $j=1,2$:

\begin{theorem} [\cite{GSS91}] \lb{t4.3}  
Assume Hypothesis~\ref{h4.1}. \\
$(i)$ If $\lambda \in (\phi_-^2, \phi_+^2)$, then the unimodular scattering function $S_j(\dott)$, $j=1,2$, is of the form,
\begin{align}
\begin{split} 
& S_j(\lambda) = - \f{\ol{W(f_{j,-}(k_-,\dott), f_{j,+}(k_+,\dott))}}{W(f_{j,-}(k_-,\dott), f_{j,+}(k_+,\dott))},    \\
& k_{\pm} = (\lambda - \phi_{\pm}^2)^{1/2} > 0, \; \lambda \in (\phi_-^2,\phi_+^2), \; j = 1,2.
\end{split}
\end{align}
$(ii)$ If $\lambda \in (\phi_+^2,\infty)$, then the unitary scattering matrix $S_j(\dott)$, $j=1,2$,  in $\bbC^2$ is given by 
\begin{equation} 
S_j(\lambda) = \begin{pmatrix} T_j(\lambda) & R_j^r(\lambda) \\ R_j^{\ell}(\lambda) & T_j(\lambda) \end{pmatrix}, 
\quad j=1,2,
\end{equation}
where the transmission and reflection coefficients from left and right incidence are of the form 
\begin{align}
\begin{split} 
& T_j(\lambda) = \f{2i (k_-k_+)^{1/2}}{W(f_{j,-}(k_-,\dott), f_{j,+}(k_+,\dott))},    \\
& R_j^{\ell}(\lambda) = - \f{W(f_{j,-}(-k_-,\dott), f_{j,+}(k_+,\dott))}{W(f_{j,-}(k_-,\dott), f_{j,+}(k_+,\dott))},    \\
& R_j^{\ell}(\lambda) = - \f{W(f_{j,-}(k_-,\dott), f_{j,+}(-k_+,\dott))}{W(f_{j,-}(k_-,\dott), f_{j,+}(k_+,\dott))};   \\
& k_{\pm} = (\lambda - \phi_{\pm}^2)^{1/2} > 0, \; \lambda \in (\phi_+^2,\infty), \; j = 1,2.
\end{split}
\end{align}
\end{theorem}

Explicitly, for $\lambda \in (\phi_+^2,\infty)$, unitarity of $S_j(\lambda)$, $j=1,2$, implies
\begin{align} 
\begin{split}
& |T_j(\lambda)|^2 + |R_j^{\ell}(\lambda)|^2 = 1 = |T_j(\lambda)|^2 + |R_j^r(\lambda)|^2,    \\
& |R_j^{\ell}(\lambda)| = |R_j^r(\lambda)|; \quad \lambda \in (\phi_+^2,\infty), \; j=1,2.
\end{split}
\end{align}

Thus far we treated $S_j$ separately for $j=1,2$. However, since $H_1=A^*A$ and $H_2=AA^*$, $S_1$ and $S_2$ are intimately connected as will be shown next. For this purpose we recall that 
\begin{equation}
\begin{cases}
f_{1,\pm} (k_{\pm},x),   \\
f_{2,\pm} (k_{\pm},x) = (\pm i k_{\pm} + \phi_{\pm})^{-1} (Af_{1,\pm})(k_{\pm},x); 
\end{cases} \quad x \in \bbR,
\end{equation}
are correctly normalized Jost solutions of $H_1$ and $H_2$. Thus, using the following elementary Wronskian identity,
\begin{equation}
W(A f(z,\dott), A g(z,\dott)) = z \, W(f(z,\dott), g(z,\dott)), \quad z \in \bbC,
\end{equation}
where $f(z,\dott)$ and $g(z,\dott)$ are any distributional solutions of 
\begin{equation}
(H_1 u(z,\dott))(x) = (A^*A u(z,\dott))(x) = z u(z,x), \quad (z,x) \in \bbC \times \bbR,  
\end{equation}
one obtains the following result:
 
\begin{theorem} [\cite{GSS91}] \lb{t4.4}  
Assume Hypothesis~\ref{h4.1}. \\
$(i)$ If $\lambda \in (\phi_-^2, \phi_+^2)$, then 
\begin{equation}
S_1(\lambda) = (ik_- + \phi_-)(- i k_- + \phi_-)^{-1} S_2(\lambda), \quad \lambda \in (\phi_-^2,\phi_+^2).   \lb{4.23}
\end{equation}
$(ii)$ If $\lambda \in (\phi_+^2,\infty)$, then 
\begin{align}
\begin{split} 
& T_1(\lambda) = (ik_- + \phi_-)(i k_+ + \phi_+)^{-1} T_2(\lambda),     \\
& R_1^{\ell}(\lambda) = (ik_- + \phi_-)(-i k_- + \phi_-)^{-1} R_2^{\ell}(\lambda),     \\
& R_1^r(\lambda) = (- ik_+ + \phi_+)(i k_+ + \phi_+)^{-1} R_2^r(\lambda); \quad \lambda \in (\phi_+^2,\infty).   \lb{4.24} 
\end{split}
\end{align}
\end{theorem}

In the special case where $\phi_- = \phi_+ = 0$, equation \eqref{4.23} and the connection between reflection coefficients in \eqref{4.24} considerably simplify.

One can also derive the connection between norming constants for nonzero eigenvalues of $H_j$, $j=1,2$, but we omit the details here (cf.\ \cite{GSS91}.  

Next, we turn to the supersymmetric Dirac-type operator $Q = \begin{pmatrix} 0 & A^* \\ A & 0 \end{pmatrix}$ defined in \eqref{4.7}. 

We start with facts on the spectrum of $Q$:

\begin{theorem} [\cite{GSS91}] \lb{t4.5}  
Assume Hypothesis~\ref{h4.1}. Then
\begin{equation}
\sigma_{ess}(Q) = \sigma_{ac}(Q) = (-\infty,-|\phi_-|] \cup [|\phi_-|,\infty), \quad \sigma_{sc}(Q) = \emptyset. 
\end{equation}
Moreover, $Q$ has simple spectrum in the interval $(-|\phi_+|, - |\phi_-|) \cup (|\phi_-|,\phi_+|)$ $($this assumes $\phi_-^2 < \phi_+^2$$)$ and spectral multiplicity equal to two on the union of intervals $(-\infty, -|\phi_+|,) \cup (|\phi_+|,\infty)$. In addition, $Q$ has finitely many simple eigenvalues in the interval $(-|\phi_-|, |\phi_-|)$ $($assuming $\phi_- \neq 0$$)$, symmetrically placed with respect to zero. 

There are no eigenvalues embedded into the essential spectrum of $Q$ and there are no threshold eigenvalues, that is,
\begin{equation}
\sigma_p(Q) \cap \{(-\infty, -|\phi_-|] \cup [|\phi_-|, \infty)\} = \emptyset. 
\end{equation} 
\end{theorem}

Jost solutions for $Q$ are given as follows,
\begin{align}
\begin{split}
& F_{1,\pm}(\zeta,k_{\pm},x) = \begin{pmatrix} f_{1,\pm}(k_{\pm},x) \\[1mm] 
\zeta^{-1} (A f_{1,\pm})(k_{\pm},x) \end{pmatrix},    \\
&  F_{2,\pm}(\zeta,k_{\pm},x) = \begin{pmatrix} - \zeta^{-1} (A^* f_{2,\pm})(k_{\pm},x) \\[1mm] 
f_{2,\pm}(k_{\pm},x) \end{pmatrix},     \\
& k_{\pm} = (z - \phi_{\pm}^2)^{1/2}, \; \Im(k_{\pm}) \geq 0, \quad \zeta^2 = z, \; \zeta, z \in \bbC \backslash\{0\}, \; x \in \bbR, 
\end{split}
\end{align}
such that 
\begin{equation}
(Q F_{j,\pm})(\zeta,k_{\pm},\dott) = (-1)^{j+1} \zeta F_{j,\pm}(\zeta,k_{\pm},\dott), \quad \zeta \in \bbC \backslash \{0\}, \; j=1,2.
\end{equation}

The scattering matrix $S_Q(\dott)$ associated with $Q$ then can be described as follows:

\begin{theorem} [\cite{GSS91}] \lb{t4.6}  
Assume Hypothesis~\ref{h4.1}.  \\
$(i)$ If $|\mu| \in (|\phi_-|, |\phi_+|)$, then the unimodular scattering function $S_Q(\dott)$ is of the form, 
\begin{equation}
S_Q(\mu) = - \f{\ol{W(F_-(\mu,k_-, \dott), F_+(\mu,k_+,\dott))}}{W(F_-(\mu,k_-, \dott), F_+(\mu,k_+,\dott))}, \quad \mu \in 
(|\phi_-||, |\phi_+|).
\end{equation}
$(ii)$ If $\mu \in (-\infty,-|\phi_+|) \cup (|\phi_+|,\infty)$, then the unitary scattering matrix $S_Q(\dott)$ in $\bbC^2$ is given by
\begin{equation}
S_Q(\mu) = \begin{pmatrix} T_Q(\mu) & R_Q^r(\mu) \\[1mm] R_Q^{\ell}(\mu) & T_Q(\mu) \end{pmatrix}, 
\end{equation}
where the transmission and reflection coefficients from left and right incidence are of the form 
\begin{align}
\begin{split} 
& T_Q(\mu) = \f{2i (k_-k_+)^{1/2}}{W(F_-(\mu,k_-, \dott), F_+(\mu,k_+,\dott))},    \\
& R_Q^{\ell}(\mu) = - \f{W(F_-(\mu,-k_-, \dott), F_+(\mu,k_+,\dott))}{W(F_-(\mu,k_-, \dott), F_+(\mu,k_+,\dott))},   \\
& R_Q^r(\mu) = - \f{W(F_-(\mu,k_-, \dott), F_+(\mu,-k_+,\dott))}{W(F_-(\mu,k_-, \dott), F_+(\mu,k_+,\dott))},   \\
& \hspace*{2.15cm} \mu \in (-\infty,-|\phi_+|) \cup (|\phi_+|,\infty).  
\end{split}
\end{align}
Here we used 
\begin{equation}
F_{\pm}(\mu,k_{\pm},x) = \begin{cases} F_{1,\pm}(\mu,k_{\pm},x), & \mu \in (|\phi_-|,\infty),  \\
F_{2,\pm}(-\mu,k_{\pm},x),  & \mu \in (-\infty, -|\phi_-|),
\end{cases} \quad x \in \bbR,
\end{equation}
and $W(F,G)$ denotes the $2 \times 2$ determinant of the column vectors $F, G \in \bbC^2$.
\end{theorem}

Finally, taking into account the Wronskian identity 
\begin{align} 
\begin{split} 
& W(F_-(\mu,\sigma k_-,\dott), F_+(\mu,\sigma' k_+,\dott))      \\ 
& \quad = \begin{cases} \mu^{-1}
W(f_{1,-}(\sigma k_-,\dott), f_{1,+}(\sigma' k_+,\dott)), & \mu \in (|\phi_-|,\infty), \\
\mu^{-1} W(f_{2,-}(\sigma k_-,\dott), f_{2,+}(\sigma' k_+,\dott)), & \mu \in (-\infty,-|\phi_-|), 
\end{cases}      \\
& \hspace*{7.25cm} \sigma, \sigma' \in \{-1,1\}, 
\end{split}  
\end{align}
one confirms the following connection between the scattering matrices of $Q$ and $H_j$, $j=1,2$:

\begin{theorem} [\cite{GSS91}] \lb{t4.7}  
Assume Hypothesis~\ref{h4.1}. Then 
\begin{equation}
S_Q(\mu) = \begin{cases} S_1(\mu^2), & \mu \in (|\phi_-|,\infty), \\
S_2(\mu^2), & \mu \in (-\infty, -|\phi_-|).
\end{cases}
\end{equation}
\end{theorem}

\subsection{Applications to Floquet Theory}\label{sec15}
In this subsection we consider Floquet theory for one-dimensional Schr\"odinger and Dirac-type operators in the supersymmetric context. In this case $A$ becomes an $L^2(\bbR)$-realization of the differential expression $A= (d/dx) + \phi(x)$, $x \in \bbR$, for an appropriate real-valued periodic function $\phi$ on $\bbR$. The material of this section is taken from \cite{Ge92}, \cite[Sect.~7.5]{GNZ24}, \cite{GSS91}, \cite{GW95a}. 

We start by reviewing some elements of Floquet theory and hence introduce the following convenient set of assumptions.

\begin{hypothesis} \lb{h5.1}
Let $\omega \in (0,\infty)$ and suppose that $\phi \colon \bbR \to \bbR$ satisfies
\begin{align}
\phi, \phi' \in L^{\infty}(\bbR), \quad \phi(x + \omega) = \phi(x), \quad x \in \bbR.
\end{align}
\end{hypothesis}

Consider a fundamental system of distributional solutions $\vartheta_j(z,\dott), \varphi_j(z,\dott)$ of $H_j u(z,\dott) =z u(z,\dott)$, $z \in \bbC$, $j=1,2$, normalized by 
\begin{align}
\begin{split} 
& \vartheta_j(z,0) = 1, \quad \vartheta_j'(z,0) = 0, \\
& \varphi_j(z,0) = 0, \quad \varphi_j'(z,0) = 1, \quad z \in \bbC.    \lb{5.2}
\end{split} 
\end{align}
Then for fixed $x \in \bbR$, $\vartheta_j(z,x)$ and $\varphi_j(z,x)$ are entire with respect to $z \in \bbC$. 

Thus, $\Phi_j(z,\dott)$ given by 
\begin{equation}
\Phi_j(z,\dott) = \begin{pmatrix} \vartheta_j(z,\dott) & \varphi_j(z,\dott) \\
\vartheta_j'(z,\dott) & \varphi_j'(z,\dott) \end{pmatrix}, \quad z \in \bbC, \; j=1,2, 
\end{equation}
represents a fundamental solution matrix of $H_j u(z,\dott) =z u(z,\dott)$, $z \in \bbC$, $j=1,2$, normalized by 
\begin{equation}
\Phi_j(z,0) = I_2, \quad z \in \bbC, \; j=1,2.
\end{equation}
The monodromy matrix $\cM_j(\dott)$ associated with $H_j$, $j=1,2$, is then given by
\begin{equation}
\cM_j(z) = \Phi_j(z,\omega) = \begin{pmatrix} \vartheta_j(z,\omega) & \varphi_j(z,\omega) \\
\vartheta_j'(z,\omega) & \varphi_j'(z,\omega) \end{pmatrix}, \quad z \in \bbC, \; j=1,2,    \lb{5.5}
\end{equation}
and the corresponding Floquet discriminant $D_j(\dott)$, $j=1,2$, equals one half the trace of the monodromy matrix, that is,
\begin{equation}
D_j(z) = \tr_{\bbC^2}(\cM_j(z))/2 = [\vartheta_j(z,\omega) + \varphi_j'(z,\omega)]/2, \quad z \in \bbC, \; j=1,2.    \lb{5.6}
\end{equation}
Taking into account that 
\begin{equation} 
H_1=A^*A = - (d^2/dx^2) + V_1, \quad H_2=AA^* = - (d^2/dx^2) + V_2, 
\end{equation} 
where once again 
$V_j = \phi^2 + (-1)^j \phi'$, $j=1,2$, and $\phi$ are related via Miura's transformation \eqref{4.6}, one infers that 
\begin{align}
\begin{split} 
& \vartheta_1(z,\dott), \quad \varphi_1(z,\dott),    \\
& \vartheta_2(z,\dott) = A \big\{z^{-1} \phi(0) \vartheta_1(z,\dott) +[1 - z^{-1} \phi(0)^2] \varphi_1(z,\dott)\big\},   \\
& \varphi_2(z,\dott) = A \big\{- z^{-1} \vartheta_1(z,\dott) + z^{-1} \phi(0) \varphi_1(z,\dott)\big\}; \quad z \in \bbC \backslash \{0\},    \lb{5.7}
\end{split} 
\end{align}
satisfy the boundary conditions in \eqref{5.2}. As a consequence one actually infers equality of $D_1(\dott)$ and $D_2(\dott)$, that is,
\begin{equation}
D_1(z) = D_2(z) = D(z), \quad z \in \bbC.    \lb{5.8} 
\end{equation}
To investigate $D(\dott)=D_j(\dott)$ further, we now introduce the following families of operators $H_{\omega,j}(\varphi,x_0)$, $\varphi \in [-\pi,\pi]$, $j=1,2$, in $L^2((x_0,x_0+\omega))$ as follows:
\begin{align}
& H_{\omega,j}(\varphi,x_0) f =  - f'' +V_j f,  \quad \varphi \in [- \pi,\pi],      \lb{6.5.37} \\
& f \in \dom(H_{\omega,j}(\varphi,x_0)) = \big\{g \in L^2((x_0,x_0+\omega)) \,\big|\, g, g' \in AC([x_0,x_0+\omega]);
\no \\
& \hspace*{1.95cm}
g^{(k)}(x_0+\omega) = e^{i \varphi} g^{(k)}(x_0), \, k=0,1; \, [-g'' + V_j g] \in L^2((x_0,x_0+\omega))\big\}, \no
\end{align}
in particular, $\varphi \in [-\pi,\pi]$ represents a boundary condition parameter in $H_{\omega,j}(\varphi,x_0)$. The operators $H_{\omega,j}(\varphi,x_0)$ have $($$x_0$-independent\,$)$ discrete spectra of the type 
\begin{align}
\begin{split} 
\sigma(H_{\omega,j}(\varphi,x_0)) = \{\lambda_{\omega,n}(\varphi)\}_{n \in \bbN_0}, 
\lambda_{\omega,n}(\varphi) \leq \lambda_{\omega,n+1}(\varphi),&  \\ 
n \in \bbN_0, \; \varphi \in [- \pi,\pi], \; j=1,2,&
\end{split} 
\end{align}
and 
\begin{equation}
\lambda \in \sigma(H_{\omega,j}(\varphi,x_0)) \, \text{ if and only if } \, D(\lambda) = \cos(\varphi),
\quad \varphi \in [- \pi,\pi], \; j=1,2.
\end{equation}
In particular, $H_{\omega,j}(\varphi,x_0)$ and $H_{\omega,j}(- \varphi,x_0)$, $j=1,2$, are antiunitarily equivalent via
complex conjugation, hence they have identical eigenvalues and complex conjugate eigenfunctions.

The precise band and gap structure of $\sigma(H_{\omega,j})$, depends on the discrete spectra of $H_{\omega,j}(\varphi,x_0)$, $j=1,2$, as $\varphi$ varies in $[- \pi, \pi]$, see Theorem~\ref{t5.2}. In this context, we now introduce the sequence
\begin{equation}
\{E_n\}_{n \in \bbN_0}, \quad E_n \leq E_{n+1}, \; n \in \bbN_0,
\end{equation}
via
\begin{equation}
\{\lambda_{\omega,n}(0)\}_{n \in \bbN_0} = \{E_0, E_{4n-1}, E_{4n}\}_{n \in \bbN}, \quad
\{\lambda_{\omega,n}(\pi)\}_{n \in \bbN_0} = \{E_{4n+1}, E_{4n+2}\}_{n \in \bbN_0},
\end{equation}
and note the identity
\begin{equation}
D(z)^2 - 1 = \omega^2 (E_0 - z) \prod_{n=1}^{\infty} \big[(E_{2n-1}-z) (E_{2n}-z) \omega^4 \pi^{-4} n^{-4}\big], 
\quad z \in \bbC.
\end{equation}

Weyl--Titchmarsh solutions for $H_j u(z,\dott) =z u(z,\dott)$, $z \in \bbC \backslash \bbR$, $j=1,2$, are then of the form
\begin{equation}
\psi_{j,\pm}(z,x) = \vartheta_j(z,x) + m_{j,\pm}(z) \varphi_j(z,x), \quad z \in \bbC \backslash \bbR, \; x \in \bbR, \; j=1,2,  \lb{5.9}
\end{equation}
and satisfy
\begin{equation}
\psi_{j,+} (z,\dott) \in L^2((0,\infty)), \; \psi_{j,-} (z,\dott) \in L^2((-\infty,0)), \quad z \in \bbC \backslash \bbR,    \lb{5.10}
\end{equation}
with the Weyl--Titchmarsh function $m_{j,\pm}(\dott)$, $j=1,2$, given by
\begin{equation}
m_{j,\pm}(z) = \big\{D(z) - \vartheta_j(z,\omega) \pm [D(z)^2 - 1]^{1/2}\big\}\big/\varphi_j(z,\omega), \quad z \in \bbC\backslash \bbR.     \lb{5.11} 
\end{equation}

In addition one can show that $\psi_{j,\pm}(z,\dott)$ are Floquet solutions satisfying 
\begin{equation}
\psi_{j,\pm}(z,x+\omega) = \rho_{j,\pm}(z) \psi_{j,\pm}(z,x), \quad z \in \bbC \backslash \bbR, \; x \in \bbR,    \lb{5.12} 
\end{equation}
and 
\begin{equation}
W(\psi_{j,+}(z,\dott), \psi_{j,-}(z,\dott)) = \f{\rho_{j,-}(z) - \rho_{j,+}(z)}{\varphi_j(z,\omega)}, \quad z \in \bbC \backslash \bbR, 
\lb{5.13} 
\end{equation}
where 
\begin{align}
& \sigma (\cM_j(z)) = \{\rho_{j,+}(z), \rho_{j,-}(z)\},    \lb{5.14} \\
& \rho_{j,\pm}(z) = D_j(z) \pm [D_j(z)^2 - 1]^{1/2},     \lb{5.15} \\
& \rho_{j,+}(z) + \rho_{j,-}(z) = 2 D_j(z), \quad \rho_{j,+}(z) \rho_{j,-}(z) =1,    \lb{5.16} \\
& \rho_{j,\pm}(z) \notin \{-1,1\}; \quad z \in \bbC \backslash \bbR, \; j=1,2.     \lb{5.17}
\end{align}
Moreover,
\begin{align}
&\psi_{j,\pm}(z,x) = e^{\mp \kappa(z) x} p_{j,\pm}(z,x), \quad x \in \bbR,   \lb{5.18} \\
& p_{j,\pm}(z,x+\omega) = p_{j,\pm}(z,x), \quad x \in \bbR,    \lb{5.19} \\
& \rho_{j,\pm}(z) = e^{\mp \kappa(z)}; \quad z \in \bbC \backslash \bbR, \; j=1,2.     \lb{5.20}
\end{align}

One notes that by analytic continuation, \eqref{5.9}--\eqref{5.20} extend to all $z \in \rho(H_j)$, $j=1,2$, in fact one can even show they extend to all $z \in \bbC$ as long as $\rho_{j,+}(z) \neq \rho_{j,-}(z)$ which occurs as long as $\rho_{j,\pm} \notin \{-1,1\}$, equivalently, as long as $D_j(z) \notin \{-1,1\}$ (see \cite[Example~7.5.1]{GNZ24} for a detailed discussion).

Regarding the spectra of $H_j$, $j=1,2$, we recall the following result:

\begin{theorem} [\cite{Ge89}, \cite{GNZ24},~Sect.~7.5, \cite{GSS91}] \lb{t5.2}  
Assume Hypothesis~\ref{h5.1}. Then the spe\-ctra of $H_1$ and $H_2$ coincide, they are purely absolutely continuous and of multiplicity two. In particular,
\begin{align} 
& \sigma(H_j) = \sigma_{ac}(H_j) = \{\lambda \in \bbR \, | \, |D(\lambda)| \leq 1\} = \bigcup_{n \in \bbN} \sigma_n 
= \bigcup_{\varphi \in [-\pi,\pi]} \sigma(H_{\omega,j}(\varphi,x_0)),    \no \\
& \sigma_n = [E_{2n-2}, E_{2n-1}], \quad n \in \bbN, \quad 0 \leq E_0 < E_1 \leq E_2 < E_3 \leq E_4 < \cdots , \no \\
& \sigma_p(H_j) = \sigma_{sc}(H_j) = \emptyset; \quad j=1,2.
\end{align}

\end{theorem}

Similarly, the spectrum of $Q$ can be described in terms of that of $H_1, H_2$ as follows: 

\begin{theorem} [\cite{GSS91}] \lb{t5.3}  
Assume Hypothesis~\ref{h5.1}. Then the spectrum of $Q$, is purely absolutely continuous, symmetric with respect to zero, and of multiplicity two. In particular,
\begin{align} 
\begin{split} 
& \sigma(Q) = \sigma_{ac}(Q) = \{\mu \in \bbR \, | \, |D(\mu^2)| \leq 1\} = \bigcup_{n \in \bbZ \backslash \{0\}} \Sigma_n,    \\
& \Sigma_n = \big[|E_{2(n-1)}|^{1/2}, |E_{2n-1}|^{1/2}\big], \quad \Sigma_{-n} = - \Sigma_n, \quad n \in \bbN,    \\
& \sigma_p(Q) = \sigma_{sc}(Q) = \emptyset.
\end{split} 
\end{align}
\end{theorem}

To connect the Floquet discriminant $D(\dott)$ of $H_j$, $j=1,2$, and that of $D_Q(\dott)$ of $Q$ we first introduce a fundamental solution matrix $\Phi_Q(\mu,\dott)$ of $Q U(\mu,\dott) = \mu U(\mu,\dott)$, $\mu \in \bbR \backslash \{0\}$ by
\begin{align}
& \Phi_Q(\mu,x)     \\
& \quad = \begin{cases}
\begin{pmatrix}
\vartheta_1(\mu^2,x) - \phi(0) \varphi_1(\mu^2,x) & \mu \, \varphi_1(\mu^2,x) \\
\mu^{-1} [A(\vartheta_1(\mu^2,x) - \phi(0) \varphi_1(\mu^2,x)] & A \varphi_1(\mu^2,x)
\end{pmatrix}, \;\; \mu \in (0,\infty), \\[4mm] 
\begin{pmatrix}
- A^* \varphi_2(\mu^2,x) & - \mu^{-1} [A^*(\vartheta_2(\mu^2,x) + \phi(0) \varphi_2(\mu^2,x)]  \\
\mu \, \varphi_2(\mu^2,x) & \vartheta_2(\mu^2,x) + \phi(0) \varphi_2(\mu^2,x)
\end{pmatrix}, \;\; \mu \in (-\infty,0). \\
\end{cases}     \no 
\end{align}
Then $\Phi_Q(\mu,\dott)$ is normalized,
\begin{equation}
\Phi_Q(\mu,0) = I_2, \quad \mu \in \bbR\backslash \{0\},
\end{equation}
the monodromy matrix $\cM_Q(\dott)$ associated with $Q$ is given by 
\begin{equation}
\cM_Q(\mu) = \Phi(\mu,\omega), \quad \mu \in \bbR\backslash \{0\},
\end{equation}
and hence the discriminant $D_Q(\dott)$ of $Q$ is given by 
\begin{equation}
D_Q(\mu) = \tr_{\bbC^2}(\cM_Q(\mu))/2 = D(\mu^2), \quad \mu \in \bbR\backslash \{0\},    \lb{5.26} 
\end{equation}
as a comparison with \eqref{5.5}--\eqref{5.8} shows. By analyticity with respect to $\mu \in \bbC$, \eqref{5.26} extends of course to $\mu = 0$.

\subsection{Applications to the $\KdV$ and $\mKdV$ Hierarchy}\label{sec16}
In this subsection we discuss the Korteweg--de Vries ($\KdV$) and modified Korteweg--de Vries ($\mKdV$) hierarchies and their interrelations via the Miura transform. We follow the treatments provided in \cite{Ge89}, \cite{Ge91}, \cite{Ge92}, \cite{GH00}, \cite[Ch.~1]{GH03}, \cite{GSS91}, \cite{GS90}, \cite{GS95}, \cite{GW93}, \cite{GW95}.  

To set the stage we start by recursively defining the $\KdV$ and $\mKdV$ hierarchies as follows: Assume 
\begin{equation}
V \in C^1(\bbR^2) \, \textit{real-valued}, \quad \partial_x^m V \in L^{\infty}(\bbR^2), \; m \in \bbN_0.
\end{equation}
(In this section, generally, $V = V(x,t)$, $(x,t) \in \bbR^2$.) 

Suppressing the time variable for a moment, consider the one-dimensional second-order differential expression
\begin{equation}
L = - \frac{d^2}{dx^2} + V, 
\lb{1.2.2}
\end{equation}
of Schr\"{o}dinger-type. To construct the $\KdV$ hierarchy we need
a second differential expression of order $2n+1$, denoted by $P_{2n+1}$,
$n\in\bbN_0$, defined recursively in the following. We take the quickest route
to the construction of $P_{2n+1}$ and hence to that of the $\KdV$ hierarchy by
starting from the recursion relation \eqref{1.2.3} below. 

We begin by recursively introducing the sequence 
\begin{align}
f_0 = 1, \quad 
f_{\ell,x} = - \frac{1}{4} f_{\ell-1,xxx} + V f_{\ell-1,x} + \frac{1}{2} V_x f_{\ell-1}, \quad \ell \in \bbN.   \lb{1.2.3}
\end{align}
Explicitly, one finds
\begin{align}
f_0 & =1, \no \\
f_1 & = \frac{1}{2} V + c_1,  \no \\
f_2 &= - \frac{1}{8} V_{xx} + \frac{3}{8} V^2 +
 c_1 \frac{1}{2} V + c_2, \lb{1.2.4} \\ 
f_3 & = \frac{1}{32} V_{xxxx} -
\frac{5}{16} V V_{xx} - \frac{5}{32} V_x^2 + \frac{5}{16} V^3\no \\
&\quad + c_1 \big(- \frac{1}{8} V_{xx} + \frac{3}{8} V^2 \big)+c_2 \frac{1}{2} V + c_3,  \no \\ 
& \quad \text{ etc.} \no
\end{align}
Here $\{c_\ell\}_{\ell\in\bbN}\subset\bbC$ denote integration 
constants which naturally
arise when solving \eqref{1.2.3} and we define 
\begin{equation}
c_0=1. \lb{1.2.4ca}
\end{equation}

\begin{remark}\lb{remark1.2.1}
Using the nonlinear recursion (D.8) in Theorem~D.1 of \cite{GH03}, one infers 
inductively that all elements $f_\ell$, $\ell\in\bbN_0$, are differential polynomials in
$V$,  that is, polynomials with respect to $V$ and (some of) its
$x$-derivatives.   
\hfill $\diamond$
\end{remark}

By construction, $f_{\ell}$ depend on $(x,t) \in \bbR^2$ and sometimes we will indicate this by writing $f_{\ell}(x,t)$, $\ell \in \bbN$. In some cases, however, the dependence of $f_{\ell}$ on $V$ and some of its $x$-derivatives, being a differential polynomial in $V$, is rather useful and hence we will also use the notation $f_{\ell}(V)$, $\ell \in \bbN$, whenever convenient. 

Next we  define differential expressions $P_{2n+1}$  of order $2n+1$ by 
\begin{equation}
P_{2n+1}  = \sum_{\ell=0}^n \bigg( f_{n-\ell}
\frac{d}{dx}-\frac12 f_{n-\ell,x} \bigg) L^{\ell}, \quad n\in\bbN_0.
\lb{1.2.5}
\end{equation}
We record the first few $P_{2n+1}$, 
\begin{align}
P_1&=\frac{d}{dx},\no \\
P_3&=-\frac{d^3}{dx^3}+\frac{3}{2} V \frac{d}{dx}
+\frac{3}{4} V_x+c_1\frac{d}{dx},\lb{1.2.6a} \\
P_5&=\frac{d^5}{dx^5}-\frac{5}{2} V \frac{d^3}{dx^3}
- \frac{15}{4} V_x\frac{d^2}{dx^2}+\bigg(\frac{15}{8} V^2 - \frac{25}{8} V_{xx}\bigg)\frac{d}{dx} 
+ \frac{15}{8} V V_x-\frac{15}{16}V_{xxx} \no \\
& \quad +c_1\bigg(-\frac{d^3}{dx^3}+\frac{3}{2} V \frac{d}{dx}
+\frac{3}{4} V_x\bigg)+c_2\frac{d}{dx},     \no \\
& \quad \text{  etc.} \no
\end{align}
Using the recursion \eqref{1.2.3}, the commutator of $P_{2n+1}$ and $L$
can be explicitly computed and yields\footnote{The recursion \eqref{1.2.3} is
constructed in such a manner that the commutator of $P_{2n+1}$ and $L$
ceases to be a higher-order differential expression but results in
multiplication by $2f_{n+1,x}$ only.} 
\begin{equation}
[P_{2n+1},L] = 2f_{n+1,x}, \quad n\in\bbN_0.
\lb{1.2.6}
\end{equation}
In particular, $(L,P_{2n+1})$, $n \in \bbN_0$, represent the celebrated {\it Lax pairs} of the $\KdV$ hierarchy. 

Next, $V$ is explicitly considered as a function of space and time, $V = V(x,t)$, $(x,t) \in \bbR^2$. The second-order differential expression $L$ (cf.\ \eqref{1.2.2}) now reads
\begin{equation}
L(t) = - \frac{d^2}{dx^2} + V(\dott,t).
\lb{1.2.18aa}
\end{equation}
The quantities $\{f_\ell(\dott,t)\}_{\ell\in\bbN_0}$ and $P_{2n+1}(t)$, $n\in\bbN_0$,
are still defined by \eqref{1.2.3} and \eqref{1.2.5}, respectively.
The time-dependent $\KdV$ hierarchy is then obtained by imposing the  
Lax commutator equations\footnote{We emphasize that our notion of a Lax operator $L$ in connection with NLS in Section~\ref{s3} is quite different from the current one for $\KdV$.} 
\begin{equation}
\frac{d}{dt} L(t) - [P_{2n+1}(t), L(t)] = 0, \quad t\in\bbR,
\lb{1.2.19}
\end{equation}
varying $n\in\bbN_0$. By \eqref{1.2.6}, the latter are equivalent to the collection of evolution 
equations
\begin{equation}
\KdV_n (V) = V_{t} - 2f_{n+1,x}(V)=0, \quad (x,t)\in\bbR^2, \;\;
n\in\bbN_0. \lb{1.2.20}
\end{equation}
Explicitly,  
\begin{align}
\KdV_0 (V) &= V_{t} - V_x = 0, \no \\
\KdV_1 (V) & = V_{t} + (1/4) V_{xxx} -
(3/2) V V_x + c_1 (-V_x) = 0,\lb{1.2.22}\\ 
\KdV_2 (V)&= V_{t} - (1/16) V_{xxxxx} + (5/8) V V_{xxx}
+ (5/4) V_x V_{xx} - (15/8) V^2 V_x \no \\ 
&\quad + c_1 \big((1/4) V_{xxx} - (3/2) V V_x\big) + c_2 (-V_x) = 0,   \no \\
& \quad \text{ etc.,} \no
\end{align}
represent the first few equations of the $\KdV$ hierarchy. 
The equation
$\KdV_1(V)=0$ (with $c_1=0$) is of course \textit{the} Korteweg--de Vries
equation. 

Turning to the $\mKdV$ hierarchy one conveniently assumes 
\begin{equation}
\phi \in C^1(\bbR^2) \, \textit{real-valued}, \quad \partial^m_x \phi \in L^{\infty}(\bbR^2), \; m \in \bbN_0
\end{equation}
and, suppressing once more the time variable for a moment, introduces the one-dimensional $2 \times 2$ first-order matrix-valued differential expression, equivalently, the supersymmetric Dirac operator $Q$ in $L^2(\bbR) \oplus L^2(\bbR)$ as in \eqref{4.7}, 
\begin{equation}
Q = \begin{pmatrix} 0 & A^* \\ A & 0 \end{pmatrix} = \begin{pmatrix} 0 & - (d/dx) + \phi \\ (d/dx) + \phi & 0 \end{pmatrix}, 
\quad \dom(Q) = H^2(\bbR) \oplus H^2(\bbR).     \lb{6.17F}
\end{equation}
In addition, one now considers the analog of $P_{2n+1}$ in the form,
\begin{equation}
R_{2n+1} = \begin{pmatrix} P_{2n+1,1}(V_1) & 0 \\[1mm]
0 & P_{2n+1,2}(V_2) \end{pmatrix}, \quad n \in \bbN_0.  
\end{equation}
where 
\begin{equation}
P_{2n+1,j}(V_j)  = \sum_{\ell=0}^n \bigg( f_{n-\ell}(V_j)
\frac{d}{dx}-\frac12 f_{n-\ell,x}(V_j) \bigg) L_j^{\ell}(t),     \lb{6.15}
\end{equation}
recalling the Miura transform (cf.\ \eqref{4.6a})
\begin{equation}
V_j = \phi^2 + (-1)^j \phi_x, \quad j=1,2.     \lb{6.17} 
\end{equation}
The pairs $(Q,R_{2n+1})$, $n \in \bbN_0$, represents the Lax pairs of the $\mKdV$ hierarchy. 

Incidentally, introducing $A$ and $A^*$ again as in \eqref{4.2} and \eqref{4.3}, that is, 
\begin{align}
\begin{split} 
& A = \f{d}{dx} + \phi, \quad \dom(A) = H^1(\bbR),    \lb{6.17A}   \\
& A^* = - \f{d}{dx} + \phi, \quad \dom(A^*) = H^1(\bbR),    
\end{split} 
\end{align}
one once more obtains the associated supersymmetric Schr\"odinger operators $H_j \geq 0$, $j=1,2$, in $L^2(\bbR)$ via 
\begin{align}
\begin{split} 
& H_1 = A^*A = - \f{d^2}{dx^2} + V_1 \geq 0, \quad \dom(H_1) = H^2(\bbR),     \lb{6.17B} \\
& H_2 = AA^* = - \f{d^2}{dx^2} + V_2 \geq 0, \quad \dom(H_2) = H^2(\bbR),    
\end{split} 
\end{align} 
with 
\begin{equation} 
Q^2 = \begin{pmatrix} A^*A & 0 \\ 0 & AA^* \end{pmatrix} = H_1 \oplus H_2 \geq 0.    \lb{6.17C} 
\end{equation}

One then verifies in analogy to \eqref{1.2.6} that 
\begin{equation}
[R_{2n+1}, Q] = g_{n+1,x} \begin{pmatrix} 0 & 1 \\ 1 & 0 \end{pmatrix}, \quad n \in \bbN_0,    \lb{6.18}
\end{equation}
where $g_{\ell}$ are recursively determined via
\begin{align}
\begin{split} 
g_0 &= 1, \\
g_{\ell+1,x} &= - (1/4) g_{\ell,xxx} + \phi^2 g_{\ell,x} + \phi_x \bigg[\int^x dx' \, (\phi g_{\ell,x'}) + c_{\ell}\bigg], 
\quad \ell \in \bbN_0,
\end{split} 
\end{align}
where the integration constants $\{c_{\ell}\}_{\ell \in \bbN} \subset \bbC$ are the same as those in \eqref{1.2.4} for $f_{\ell}$, and again $c_0=1$.
  
One confirms inductively that $\phi g_{\ell,x}$ as well as $g_{\ell,x}$, $\ell \in \bbN_0$, are differential polynomials in $\phi$, that is, polynomials in $\phi$ and some of its $x$-derivatives. (This does not necessarily seem to be true for $g_{\ell}$ itself as $\ell$ grows.) In this context, $\int^x$ represents homogeneous integration so all constants are explicitly included in $c_{\ell}$. 
Explicitly, one obtains
\begin{align}
g_0 &= 1,   \no \\
g_1 &= \phi + c_1,    \no \\
g_2 &= -(1/4) \phi_{xx} + (1/2) \phi^3 + c_1 \phi +c_2, \\
& \quad \text{ etc.}     \no 
\end{align}

Returning to the $\mKdV$ hierarchy, we now view $\phi$ as a function of space and time, $\phi = \phi(x,t)$, $(x,t) \in \bbR^2$. This then yields \begin{align}
Q(t) &= \begin{pmatrix} 0 & - \f{d}{dx} + \phi(x,t) \\[1mm] 
\f{d}{dx} + \phi(x,t) & 0 \end{pmatrix},  \quad t \in \bbR, \; n \in \bbN_0,   \\[2mm] 
R_{2n+1}(t) &= \begin{pmatrix} P_{2n+1,1}(V_1(\dott,t)) & 0 \\[1mm]
0 & P_{2n+1,2}(V_2(\dott,t)) \end{pmatrix}, \quad t \in \bbR, \; n \in \bbN_0.   
\end{align}
The time-dependent $\mKdV$ hierarchy is then obtained by imposing the  
Lax commutator equations 
\begin{equation}
\frac{d}{dt} Q(t) - [R_{2n+1}(t), Q(t)] = 0, \quad t\in\bbR,
\lb{6.23}
\end{equation}
varying $n\in\bbN_0$. By \eqref{6.17F} and \eqref{6.18}, the latter are equivalent to 
\begin{equation}
\big(\phi_{t} - g_{n+1,x}(\phi)\big) \begin{pmatrix} 0 & 1 \\ 1 & 0 \end{pmatrix} = 0, \quad (x,t)\in\bbR^2, \; n \in \bbN_0, 
\end{equation}
and hence to the collection of evolution equations
\begin{equation}
\mKdV_n (\phi) = \phi_{t} - g_{n+1,x}(\phi) = 0, \quad (x,t)\in\bbR^2, \; n\in\bbN_0. \lb{6.24}
\end{equation}
Explicitly,  
\begin{align}
\mKdV_0 (\phi) &= \phi_{t} - \phi_x = 0, \no \\
\mKdV_1 (\phi) & = \phi_{t} + (1/4) \phi_{xxx} -
(3/2) \phi^2 \phi_x + c_1 (-\phi_x) = 0,\lb{6.25}\\ 
\mKdV_2 (\phi)&= \phi_{t} - (1/16) \phi_{xxxxx} + (5/8) (\phi_x)^3+ (5/8) \phi^2 \phi_{xxx}
+ (5/2) \phi \phi_x \phi_{xx}  \no \\
& \quad - (15/8) \phi^4 \phi_x + c_1 \big((1/4) \phi_{xxx} -
(3/2) \phi^2 \phi_x \big) + c_2 (-\phi_x) = 0,   \no \\
& \quad \text{ etc.,} \no
\end{align}
represent the first few equations of the $\mKdV$ hierarchy. Of course, the equation
$\mKdV_1(\phi)=0$ (with $c_1=0$) is \textit{the} modified Korteweg--de Vries equation. 

\begin{remark} \lb{r3.27} ${}$ \\[1mm] 
$(i)$ The simple sign change substitution 
\begin{equation}
\phi \longleftrightarrow - \phi
\end{equation}
implies the following set of substitutions:
\begin{align}
\begin{split} 
& V_1 \longleftrightarrow V_2, \quad A(\phi) \longleftrightarrow A(-\phi) = - A(\phi)^*, \\
& H_1 = A^* A \longleftrightarrow H_2 = A A^*, \, \text{\bf commutation}, \\ 
&P_{2n+1,1}(V_1) \longleftrightarrow P_{2n+1,2}(V_2),  \\
& R_{2n+1}(\phi) \longleftrightarrow \begin{pmatrix} 0 & 1 \\ 1 & 0 \end{pmatrix} R_{2n+1}(- \phi) \begin{pmatrix} 0 & 1 \\ 1 & 0 \end{pmatrix},    \\
& Q(\phi) \longleftrightarrow - \begin{pmatrix} 0 & 1 \\ 1 & 0 \end{pmatrix} Q(-\phi) \begin{pmatrix} 0 & 1 \\ 1 & 0 \end{pmatrix};  
\quad n \in \bbN_0,
\end{split} 
\end{align}
and one infers that 
\begin{equation}
g_{n+1,x}(\phi) = - g_{n+1,x}(- \phi), \quad \mKdV_n (\phi) = - \mKdV_n (- \phi), \quad n \in \bbN_0. 
\end{equation}
In particular, 
\begin{equation}
\mKdV_n (\phi) = 0 \Longleftrightarrow \mKdV_n (- \phi) = 0, \quad n \in \bbN_0, 
\end{equation}
equivalently, the $\mKdV_n$ equations \eqref{6.24}, are left invariant with respect to a sign change of $\phi$. \\[1mm] 
$(ii)$ We note that Miura's transform $V_j = \phi^2 + (-1)^j \phi_x$, $j=1,2$, implies $V_2 - V_1 = 2 \phi_x$, $V_1 + V_2 = 2 \phi^2$, and $V_{1,x} + V_{2,x} = 4 \phi \phi_x$. Thus, $V_{1,x} + V_{2,x} = 2 \phi [V_2 - V_1]$, and hence,
\begin{equation}
\phi = \f{1}{2} \f{V_{1,x} + V_{2,x}}{V_2 - V_1}. 
\end{equation}
\hfill $\diamond$
\end{remark}

The fundamental identity relating the $\KdV$ hierarchy \eqref{1.2.20} and the $\mKdV$ hierarchy \eqref{6.24}, based on Miura's transform \eqref{6.17}, is then given by 
\begin{equation} 
\KdV_n(V_j) = \big[2 \phi + (-1)^j \partial_x\big] \mKdV_n(\phi), \quad n \in \bbN_0, \; j = 1,2.     \lb{6.26}
\end{equation}
In this context we refer to the seminal paper by Adler and Moser \cite{AM78}. 

Miura's identity \eqref{6.26} can now be exploited to yield the following result: 

\begin{theorem} \lb{t3.28} ${}$ \\[1mm] 
$(i)$ Suppose $\phi \in C^1(\bbR^2)$ is real-valued, $\partial_x^m \phi \in L^{\infty}(\bbR^2)$ for all $m \in \bbN_0$, and introduce $V_j = \phi^2 + (-1)^j \phi_x$, $j =1,2$. Then also 
\begin{equation} 
V_j \in C^1(\bbR^2) \, \textit{is real-valued}, \quad \partial_x^m V_j \in L^{\infty}(\bbR^2) \text{ for all $m \in \bbN_0$}, \; j = 1,2.    \lb{6.27}
\end{equation} 
Moreover, for each $n \in \bbN_0$, 
\begin{equation}
\mKdV_n(\phi) = 0 \, \text{ implies } \, \KdV_n(V_j) = 0, \quad j=1,2.     \lb{6.28} 
\end{equation}
$(ii)$ Suppose $V_1 \in C^1(\bbR^2)$ is real-valued, $\partial_x^m V_1 \in L^{\infty}(\bbR^2)$ for all $m \in \bbN_0$, and for $t \in \bbR$, let $0 < \psi_1(\dott,t)$ be a solution of $L_1(t) \psi_1(t) = 0$, where $L_1(t) = - (d^2/dx^2) + V_1(x,t)$, $t \in \bbR$. Introducing $\phi$ via 
\begin{equation}
\phi(x,t) = - \psi_{1,x} (x,t)/\psi_1(x,t), \quad (x,t) \in \bbR^2,
\end{equation}
then $\phi$ satisfies 
\begin{align} 
& \phi \in C^1(\bbR^2), \quad \partial_x^m \phi \in L^{\infty}(\bbR^2) \text{ for all $m \in \bbN_0$},    \lb{6.30} \\
& V_1(x,t) = \phi^2(x,t) - \phi_x(x,t), \quad (x,t) \in \bbR^2.
\end{align} 
Moreover, for each $n \in \bbN_0$, 
\begin{equation}
\KdV_n(V_1) = 0 \, \text{ implies } \, \mKdV_n(\phi) = 0 \, \text{ and } \, \KdV_n(V_2) = 0,     \lb{6.32} 
\end{equation}
where
\begin{equation}
V_2(x,t) = \phi^2(x,t) + \phi_x(x,t), \quad (x,t) \in \bbR^2.
\end{equation}
The same applies with $V_1$ and $V_2$ interchanged. 
\end{theorem}
\begin{proof}
$(i)$ Assertions \eqref{6.27} for $V_j$, $j=1,2$, are clear from the hypotheses on $\phi$. Similarly, the claim \eqref{6.28} is an instant consequence of Miura's identity \eqref{6.26}. \\[1mm] 
$(ii)$ Given $C \in (0,\infty)$ such that $V_1(x,t) \leq C$, $(x,t) \in \bbR^2$, one applies \cite[p.~358--360]{Ha02} to conclude that
\begin{equation}
- C^{1/2} \leq \psi_{1,x}(x,t)/\psi_1(x,t) = - \phi(x,t) \leq C^{1/2}, \quad (x,t) \in \bbR^2, 
\end{equation}
that is, $\phi \in L^{\infty}(\bbR^2)$ and hence also $\phi_x \in L^{\infty}(\bbR^2)$ since by hypothesis, $V_1 = \phi^2 - \phi_x \in L^{\infty}(\bbR^2)$. Differentiating $V_1$ then yields $\phi_{xx} \in L^{\infty}(\bbR^2)$, and continuing this process shows that $\partial_x^m \phi \in L^{\infty}(\bbR^2)$, $m \in \bbN_0$. Clearly, $L_1 \psi_1 = 0$ and $\phi = - \psi_{1,x}/\psi_1$ yields $V_1 = \phi^2 - \phi_x$. 

If $\KdV_n(V_1) = 0$, then Miura's identity \eqref{6.26} yields 
\begin{equation}
0 = [2 \phi - \partial_x] \mKdV_n(\phi), \quad n \in \bbN_0.     \lb{6.35}
\end{equation}
Since the first-order differential equation $[\partial_x - 2 \phi] f = 0$ has the solution 
\begin{equation}
f(x,t) = C(t) \exp\bigg(2 \int^x_0 dx' \, \phi(x',t)\bigg), \quad (x,t) \in \bbR^2,
\end{equation}
with $C(t) \in \bbC$ independent of $\phi$, one concludes that
\begin{equation}
\mKdV_n(\phi) = C(t) \exp\bigg(2 \int^x_0 dx' \, \phi(x',t)\bigg), \quad (x,t) \in \bbR^2.   \lb{6.37}
\end{equation}
Choosing a particular $\phi$ such that $\phi_t$ and $\partial_x^m \phi$, $m \in \bbN$, all vanish as $x \to +\infty$ (or, as $x \to - \infty$)\footnote{For a concrete example of $\phi$ one can choose the $\mKdV$ soliton solutions $\phi_{2N-1}$, or $\phi_{2N}$, where $\lim_{x \to \pm\infty} \phi_{2N-1}(x) = \phi_{2N-1,\pm}$, with $\phi_{2N-1,+} = - \phi_{2N-1,-} \in \bbR \backslash \{0\}$, or $\lim_{x \to \pm\infty} \phi_{2N}(x) = \phi_{2N,\pm}$, with $\phi_{2N,+} = \phi_{2N,-} \in \bbR \backslash \{0\}$, see \cite[Sect.~8]{GSS91}.} yields the vanishing of the left-hand side in \eqref{6.37} and hence implies the contradiction of a nonvanishing right-hand side of \eqref{6.37} unless $C(t) = 0$. This proves $\mKdV_n(\phi) = 0$ and hence also $\KdV(V_2) = 0$, employing \eqref{6.26} again.

Clearly, these arguments are symmetric with respect to $V_1$ and $V_2$.
\end{proof}

\begin{remark} \lb{r.3.29}
We comment a bit on the spectral properties of $0 \leq H_j(t)$, $t \in \bbR$, $j=1,2$, under the conditions imposed in Theorem~\ref{t3.28}. For this purpose we first recall two facts that will repeatedly be used below: First, if $H=T^*T \geq 0$ in the complex, separable Hilbert space $\cH$, with $T$ a densely defined and closed operator in $\cH$, then, for $f_0 \in \dom(H)$, 
\begin{equation}
\big\|H^{1/2} f_0\big\|^2_{\cH} = (f_0, T^*T f_0)_{\cH} = \|T f_0\|^2_{\cH}, 
\end{equation}
implies that 
\begin{equation}
0 \in \sigma_p(H) \, \text{ if and only if } \, 0 \in \sigma_p(T).
\end{equation}
Second, if a Schr\"odinger differential expression $L = - (d^2/dx^2) + V(x)$ has a generalized (i.e., distributional) solution $\psi_0$ of $L \psi_0 = 0$ that is positive on an interval $(a,b) \subseteq \bbR$, then a second linearly independent solution $\wti \psi_0$ of $L \psi = 0$ on that interval is given by
\begin{equation} 
\wti \psi_0 (x) = \psi_0(x) \bigg[C + D \int_{x_0}^x dx' \, [\psi_0(x')]^{-2}\bigg], \quad x_0, x \in (a,b),      \lb{3.315} 
\end{equation}
where $C \in \bbC$, $D \in \bbC\backslash \{0\}$ are constants. (In our present context $C, D$ will always be chosen to be real-valued.)   \\[1mm] 
Given these facts, we now suppose that 
\begin{align} 
\begin{split} 
& \phi(x,t) \underset{x \to \pm \infty}{\longrightarrow} \phi_{\pm} \, \text{ in the sense of \eqref{3.200}, viewing $t \in \bbR$ as a parameter,}    \\
& \quad \text{with $\phi_{\pm}$ independent of $t \in \bbR$.} 
\end{split} 
\end{align} 
Then,  
\begin{equation}
\sigma_{ess}(H_j(t)) = \sigma_{ac}(H_j(t)) = \big[\min\big(\phi_-^2,\phi_+^2\big), \infty\big), \quad t \in \bbR, \; j=1,2, 
\end{equation}
and one observes that the spectral multiplicity of $H_j(t)$ equals one on the interval $\big(\min\big(\phi_-^2,\phi_+^2\big), \max\big(\phi_-^2,\phi_+^2\big)\big)$, assuming the latter is nonempty, and the spectral multiplicity of $H_j(t)$ is two on the half-line $\big(\min\big(\phi_-^2,\phi_+^2\big), \infty\big)$, $j=1,2$. In particular, the interval $\big[0, \min\big(\phi_-^2,\phi_+^2\big)\big)$ contains discrete eigenvalues of $H_j(t)$ only (if any), $j=1,2$. The half-line $\big(\min\big(\phi_-^2,\phi_+^2\big), \infty\big)$ contains no embedded eigenvalues of $H_j(t)$, $t \in \bbR$, $j=1,2$. \\[1mm] 
\noindent 
$(i)$. Suppose that $\phi_{\pm} \neq 0$: Then, for each $t \in \bbR$,  
\begin{equation}
\begin{rcases}
0 \in \sigma_p(H_1(t)) \\
0 \notin \sigma_p(H_2(t))
\end{rcases} \Longleftrightarrow \phi_- < 0 < \phi_+.     \lb{3.318}
\end{equation}
In this case $H_1(t)$ and $H_2(t)$ are obviously not isospectral, however, when restricting $H_1(t)$ to the orthogonal complement of its one-dimensional kernel generated by 
\begin{equation}
\psi_1(x,t) = C_1(t) \exp\bigg(-2 \int_0^x dx' \, \phi(x',t)\bigg), \quad (x,t) \in \bbR^2,     \lb{3.319}
\end{equation}
then it becomes unitarily equivalent to $H_2(t)$ for each $t \in \bbR$, see \eqref{3.179}. 

Similarly, for each $t \in \bbR$, 
\begin{equation}
\begin{rcases}
0 \notin \sigma_p(H_1(t)) \\
0 \in \sigma_p(H_2(t))
\end{rcases} \Longleftrightarrow \phi_+ < 0 < \phi_-.     \lb{3.320} 
\end{equation}
In this case $H_1(t)$ and $H_2(t)$ are obviously once again not isospectral, however, when restricting $H_2(t)$ to the orthogonal complement of its one-dimensional kernel generated by 
\begin{equation}
\psi_2(x,t) = C_2(t) \exp\bigg(2 \int_0^x dx' \, \phi(x',t)\bigg), \quad (x,t) \in \bbR^2,    \lb{3.321}
\end{equation}
then it becomes unitarily equivalent to $H_1(t)$ for each $t \in \bbR$, see again \eqref{3.179}. 

Employing the analog of \eqref{3.315} one verifies that neither $H_1(t)$ nor $H_2(t)$ supports a zero-energy resonance that would be associated with a nontrivial $L^{\infty}(\bbR;dx)$-solution at zero energy (which has to be linearly independent of the zero-energy eigenfunction $\psi_j$ of $H_j$, $j=1,2$, in connection with \eqref{3.318} and \eqref{3.320}). 

Next, suppose that either $\phi_{\pm} > 0$ or $\phi_{\pm} < 0$: In these cases one shows with the help of \eqref{3.315}, \eqref{3.319}, and \eqref{3.321} that  for each $t \in \bbR$, $H_j(t)$, $j=1,2$, have neither a zero-energy eigenvalue nor a zero-energy resonance. Moreover, in both cases $\phi_{\pm} > 0$ and $\phi_{\pm} < 0$, $H_1(t)$ and $H_2(t)$ are unitarily equivalent for each $t \in \bbR$. \\[1mm]
\noindent 
$(ii)$. Suppose that $\phi_- = 0$, $\phi_+ > 0$: Then, once more employing \eqref{3.315}, \eqref{3.319}, and \eqref{3.321} one concludes that for each $t \in \bbR$, $H_1(t)$ and $H_2(t)$ have no zero-energy eigenvalue; in addition, $H_1(t)$ has a zero-energy resonance, but $H_2(t)$ does not. Moreover, $H_1(t)$ and $H_2(t)$ are unitarily equivalent for each $t \in \bbR$. \\[1mm] 
$(iii)$. Suppose that $\phi_- = 0$, $\phi_+ < 0$: Then item $(ii)$ with $H_1(t)$ and $H_2(t)$ interchanged, $t \in \bbR$, holds. \\[1mm] 
$(iv)$. Suppose that $\phi_- = 0 = \phi_+$: Then once more employing \eqref{3.315}, \eqref{3.319}, and \eqref{3.321} one concludes that for each $t \in \bbR$, neither $H_1(t)$ nor $H_2(t)$ has a zero-energy eigenvalue. However, both, $H_1(t)$ and $H_2(t)$ now have a zero-energy resonance for each $t \in \bbR$. Once more,  in this case $H_1(t)$ and $H_2(t)$ are unitarily equivalent for each $t \in \bbR$. 
\hfill $\diamond$
\end{remark}

Without going into details, we note that a discrete analog of this subsection exists that connects the Toda lattice hierarchy with its modified counterpart, the Kac--van Moerbeke lattice hierarchy. For pertinent references see, for instance, \cite[Ch.~2]{GHMT08}, \cite{GHSZ93}, \cite[Ch.~14]{Te00}, and the literature cited therein.

\section{Epilogue} \lb{s4}

As a final comment, we note that quantum mechanics is based on representing the commutator relation 
\begin{equation} 
[A,B] \subseteq i I,     \lb{3.307} 
\end{equation} 
in a complex, separable Hilbert space $\cH$, with $A$ and $B$ assumed to be self-adjoint operators in $\cH$. More precisely, avoiding operator domain technicalities, following H.\ Weyl \cite{We28}, one can rephrase this in the exponential form for unitary operators as 
\begin{equation} 
e^{i sA} e^{i t B} = e^{ist} e^{i t B} e^{i sA}, \quad s,t \in \bbR.     
\end{equation} 
In the latter case, J.\ von Neumann's uniqueness theorem \cite{vN31}) proves that such a pair $(A,B)$ is unitarily equivalent to a so-called Schr\"odinger couple $(Q,P)$ in $L^2(\bbR)$, where 
\begin{align}
& (Qf)(x) = x f(x), \quad f \in \dom(Q) = \big\{g \in L^2(\bbR) , \big| \, (\dott) g(\dott) \in L^2(\bbR)\big\},   \lb{3.309} \\
& (P f)(x) = -i f'(x), \quad f \in \dom(P) = \big\{g \in L^2(\bbR) \, \big| \, g \in AC_{loc}(\bbR); \, g' \in L^2(\bbR)\big\}  \no \\
& \hspace*{5.23cm} = H^1(\bbR),   \lb{3.310} 
\end{align}
or, unitarily equivalent to a direct sum of Schr\"odinger couples. In the particular quantum mechanical context, $Q$ represents the {\it position} of a quantum mechanical particle and $P$ its {\it momentum}. The fundamental noncommutativity of $Q$ and $P$ in the form $[Q,P] \subseteq i I$, is intimately related to Heisenberg's famous uncertainty relation and hence to the impossibility of simultaneously measuring the position and the momentum of a quantum mechanical particle. 

More generally, however, dropping self-adjointness of $A$ and $B$ for a moment, one can pose the following question: \\[1mm] 
\noindent 
{\bf Query.} {\it What kind of operators $A, B$ in $\cH$ can satisfy $[A,B] \subseteq i I$?} \\[2mm]
Clearly, $A,B$ cannot be $n \times n$ matrices in $\bbC^n$ as $[A,B]$, being a commutator, has trace zero, but the matrix $i I_n$ has trace $i n$, $n \in \bbN$. (We note in passing that $[A,B]=iC$ holds for $n \times n$ matrices $A,B,C$ if and only if ${\tr}_{\bbC^n}(C) =0$, cf.\ K.\ Shoda \cite{Sh36}.)

Next, suppose $A, B \in \cB(\cH)$ satisfy $[A,B] = i I$. In that case, $BA \in \cB(\cH)$ and hence $\sigma(BA) \neq \emptyset$ and there exists some $\lambda\in \bbC$ with $\lambda \in \sigma(BA)$. But then, as $AB = BA + i I$, it follows that also $(\lambda + i) \in \sigma(AB)$ and hence, by commutation, $(\lambda + i) \in \sigma(BA)$. Iterating this consideration, one concludes that $(\lambda + i m) \in \sigma(AB)$ for all $m \in \bbN$. (If $(\lambda + i m)=0$ for such an integer $m$, just reverse the role of $AB$ and $BA$.) However, since $AB \in \cB(\cH)$, its spectrum is necessarily bounded, a contradiction. Consequently, at least one of $A,B$ cannot be bounded on $\cH$. This fact, and this type of argument, were recorded by A.\ Wintner \cite{Wi47} in 1947. 

For many more details regarding this circle of ideas, see C.\ R.\ Putnam's monograph \cite[Chs.~I, II, IV]{Pu67}.

Returning to quantum mechanics, $A$ and $B$ should be self-adjoint in $\cH$, and looking specifically at the Schr\"odinger couple $(Q,P)$ in \eqref{3.309} and \eqref{3.310}, both operators are unbounded in $L^2(\bbR)$ and hence, consistent with the classical Hellinger--Toepliz Theorem, are not defined on all of $L^2(\bbR)$. As a point in fact, a typical (smooth) state $\psi \in L^2(\bbR) \cap C^{\infty}(\bbR)$ that lies neither in $\dom(Q)$ nor in $\dom(P)$ is  
\begin{equation}
\psi(x) = \big[1 + x^2\big]^{-1/2}\sin\big(x^2\big), \quad x \in \bbR. 
\end{equation} 
Consequently, there will be states in $L^2(\bbR)$ which cannot be assigned a definite position and/or momentum, hinting once more at Heisenberg's celebrated uncertainty relation. Thus, the latter has, mysteriously and intriguingly, some origin in commutation. (In the context of commutator inequalities and their relation to uncertainty principles see also \cite{Fa78}.)

\medskip 

\noindent
{\bf Acknowledgments.} We are very grateful to Markus Hunziker for his expert help in creating Figures~1 and 2 for us. In addition, we gratefully acknowledge the extraordinarily careful reading of our manuscript by the anonymous referee, and his numerous and very helpful suggestions to improve our presentation.

%


\end{document}